\documentclass[a4paper,12pt]{preprint}
\usepackage[margin=1.24in]{geometry}  
\usepackage[full]{textcomp}
\usepackage[osf]{newtxtext}

\usepackage{mathalfa}
\usepackage{microtype}
\usepackage{xcolor}
\usepackage{booktabs}

\usepackage{enumitem}
\usepackage{amsmath}
\usepackage{amsfonts} 
\usepackage{amssymb}             

\usepackage{comment}
\usepackage{breakurl}

\usepackage{mathrsfs}
\usepackage{booktabs}
\usepackage{mathtools}

\usepackage{tikz}
\usepackage{amsthm}                
\usepackage{mhequ}
\usepackage{hyperref}

\hypersetup{pdfstartview=}

\addtocontents{toc}{\protect\hypersetup{hidelinks}}

\DeclareSymbolFont{sfoperators}{OT1}{ptm}{m}{n}
\DeclareSymbolFontAlphabet{\mathsf}{sfoperators}

\makeatletter
\def\operator@font{\mathgroup\symsfoperators}
\makeatother

\numberwithin{equation}{section}

\newtheorem{thm}{Theorem}[section]
\newtheorem{defn}[thm]{Definition}
\newtheorem{lem}[thm]{Lemma}
\newtheorem{prop}[thm]{Proposition}

\newtheorem{cor}[thm]{Corollary}

\newtheorem{eg}[thm]{Example}

\newtheorem{assumption}[thm]{Assumption}
\theoremstyle{remark}
\newtheorem{remark}[thm]{Remark}
\newtheorem{rmk}[thm]{Remark}

\DeclareMathOperator{\id}{id}

\makeatletter
\def\th@newremark{\th@remark\thm@headfont{\bfseries}}

\def\bdiamond{\mathop{\mathpalette\bdi@mond\relax}}
\newcommand\bdi@mond[2]{%
	\vcenter{\hbox{\m@th
			\scalebox{\ifx#1\displaystyle 2.6\else1.8\fi}{$#1\diamond$}%
	}}%
}

\def\bDiamond{\mathop{\mathpalette\bDi@mond\relax}}
\newcommand\bDi@mond[2]{%
	\vcenter{\hbox{\m@th
			\scalebox{\ifx#1\displaystyle 2.6\else1.2\fi}{$#1\Diamond$}%
	}}%
}
\makeatletter

\definecolor{darkgreen}{rgb}{0.1,0.7,0.1}
\definecolor{darkred}{rgb}{0.7,0.1,0.1}
\definecolor{darkblue}{rgb}{0,0,0.7}
\newcommand{\aA}{\mathcal{A}}
\newcommand{\bB}{\mathcal{B}}
\newcommand{\cC}{\mathcal{C}}
\newcommand{\dD}{\mathcal{D}}

\newcommand{\iI}{\mathcal{I}}

\newcommand{\kK}{\mathcal{K}}
\newcommand{\lL}{\mathcal{L}}
\newcommand{\mM}{\mathcal{M}}
\newcommand{\nN}{\mathcal{N}}
\newcommand{\oO}{\mathcal{O}}
\newcommand{\pP}{\mathcal{P}}
\newcommand{\qQ}{\mathcal{Q}}
\newcommand{\rR}{\mathcal{R}}
\newcommand{\sS}{\mathcal{S}}

\newcommand{\xX}{\mathcal{X}}
\newcommand{\yY}{\mathcal{Y}}
\newcommand{\zZ}{\mathcal{Z}}

\makeatletter 
\newcommand{\cov}{{\operator@font cov}}
\newcommand{\var}{{\operator@font var}}
\newcommand{\corr}{{\operator@font corr}}
\newcommand{\supp}{{\operator@font supp}}
\newcommand{\diam}{{\operator@font diam}}
\newcommand{\Av}{{\operator@font Av}}
\newcommand{\trig}{{\operator@font trig}}
\newcommand{\Enh}{{\operator@font Enh}}
\newcommand{\EEnh}{\overline {\operator@font Enh}}
\newcommand{\Com}{{\operator@font Com}}
\makeatother

\newcommand{\lfl}{\left\lfloor }  
\newcommand{\rfl}{\right\rfloor} 

\newcommand{\E}{\mathbf{E}}

\newcommand{\N}{\mathbf{N}}
\newcommand{\R}{\mathbf{R}}
\newcommand{\T}{\mathbf{T}}

\newcommand{\Z}{\mathbf{Z}}

\renewcommand{\k}{\mathbf{k}}

\newcommand{\hxi}{\widehat{\xi}}

\newcommand{\eps}{\varepsilon}
\newcommand{\Ups}{\Upsilon}

\colorlet{symbols}{blue!90!black}
\colorlet{testcolor}{green!60!black}

\def\${|\!|\!|}

\usetikzlibrary{shapes.misc}
\usetikzlibrary{shapes.symbols}
\usetikzlibrary{snakes}
\usetikzlibrary{decorations}
\usetikzlibrary{decorations.markings}

\def\drawx{\draw[-,solid] (-3pt,-3pt) -- (3pt,3pt);\draw[-,solid] (-3pt,3pt) -- (3pt,-3pt);}
\tikzset{
	root/.style={circle,fill=testcolor,inner sep=0pt, minimum size=2mm},
	dot/.style={circle,fill=black,inner sep=0pt, minimum size=1mm},
	edot/.style={circle,fill=black,inner sep=0pt, minimum size=1mm},
	odot/.style={circle,draw=black,inner sep=0pt, minimum size=1mm},
	var/.style={circle,fill=black!10,draw=black,inner sep=0pt, minimum size=
	2mm},
    svar/.style={circle,fill=black!10,draw=black,inner sep=0pt, minimum size=
	1.5mm},
    noise0/.style={rectangle,draw=symbols,fill=white,inner sep=0pt, minimum size=2mm},
    noise1/.style={circle,draw=symbols,fill=white,inner sep=0pt, minimum size=2mm},
    noise2/.style={circle,draw=symbols,fill=black!30!,inner sep=0pt, minimum size=2mm},
    noise3/.style={circle,draw=symbols,fill=symbols,inner sep=0pt, minimum size=2mm},
    noise0s/.style={rectangle,draw=symbols,fill=white,inner sep=0pt, minimum size=1.5mm},
    noise1s/.style={circle,draw=symbols,fill=white,inner sep=0pt, minimum size=1.5mm},
    noise2s/.style={circle,draw=symbols,fill=black!30!,inner sep=0pt, minimum size=1.5mm},
    noise3s/.style={circle,draw=symbols,fill=symbols,inner sep=0pt, minimum size=1.5mm},
	dotred/.style={circle,fill=symbols!50,inner sep=0pt, minimum size=2mm},
	generic/.style={semithick,shorten >=1pt,shorten <=1pt},
	ageneric/.style={semithick},
	dist/.style={ultra thick,draw=testcolor,shorten >=1pt,shorten <=1pt},
	testfcn/.style={ultra thick,testcolor,shorten >=1pt,shorten <=1pt,<-},
	testfcnx/.style={ultra thick,testcolor,shorten >=1pt,shorten <=1pt,<-,
		postaction={decorate,decoration={markings,mark=at position 0.6 with {\drawx}}}},
	kepsilon/.style={semithick,shorten >=1pt,shorten <=1pt,densely dashed,->},
	kprimex/.style={semithick,shorten >=1pt,shorten <=1pt,densely dashed,->,
		postaction={decorate,decoration={markings,mark=at position 0.4 with {\drawx}}}},
	kernel/.style={semithick,shorten >=1pt,shorten <=1pt,->},
	akernel/.style={semithick,->},
	multx/.style={shorten >=1pt,shorten <=1pt,
		postaction={decorate,decoration={markings,mark=at position 0.5 with {\drawx}}}},
	kernelx/.style={semithick,shorten >=1pt,shorten <=1pt,->,
		postaction={decorate,decoration={markings,mark=at position 0.4 with {\drawx}}}},
	kernel1/.style={->,semithick,shorten >=1pt,shorten <=1pt,postaction={decorate,decoration={markings,mark=at position 0.45 with {\draw[-] (0,-0.1) -- (0,0.1);}}}},
	kernel2/.style={->,semithick,shorten >=1pt,shorten <=1pt,postaction={decorate,decoration={markings,mark=at position 0.45 with {\draw[-] (0.05,-0.1) -- (0.05,0.1);\draw[-] (-0.05,-0.1) -- (-0.05,0.1);}}}},
	kernelBig/.style={semithick,shorten >=1pt,shorten <=1pt,decorate, decoration={zigzag,amplitude=1.5pt,segment length = 3pt,pre length=2pt,post length=2pt}},
	gepsilon/.style={dotted,semithick,shorten >=1pt,shorten <=1pt},
	renorm/.style={shape=circle,fill=white,inner sep=1pt},
	labl/.style={shape=rectangle,fill=white,inner sep=1pt},
	xi/.style={circle,fill=symbols!10,draw=symbols,inner sep=0pt,minimum size=1.2mm},
	xix/.style={crosscircle,fill=symbols!10,draw=symbols,inner sep=0pt,minimum size=1.2mm},
	xib/.style={circle,fill=symbols!10,draw=symbols,inner sep=0pt,minimum size=1.6mm},
	xibx/.style={crosscircle,fill=symbols!10,draw=symbols,inner sep=0pt,minimum size=1.6mm},
	not/.style={circle,fill=symbols,draw=symbols,inner sep=0pt,minimum size=0.5mm},
	>=stealth,
  	highlight/.style={line width=7pt,blue,draw opacity=0.2,line cap=round,line join=round},
  	cover/.style={line width=7pt,blue,line cap=round,line join=round},
	smalldot/.style={circle,fill=symbols,draw=symbols, solid,inner sep=0pt,minimum size=0.5mm},
	}

\makeatletter
\def\DeclareSymbol#1#2#3{\expandafter\gdef\csname MH@symb@#1\endcsname{\tikz[baseline=#2,scale=0.15,draw=symbols]{#3}}\expandafter\gdef\csname MH@symb@#1s\endcsname{\scalebox{0.5}{\tikz[baseline=#2,scale=0.15,draw=symbols]{#3}}}}
\def\<#1>{\csname MH@symb@#1\endcsname}
\makeatother

\DeclareSymbol{Xi22}{0.5}{\draw (0,0) node[xi] {} -- (-1,1) node[not] {} -- (0,2) node[xi] {}; }

\DeclareSymbol{0'}{0}{\node at (0,0.7) [noise0] {}; }
\DeclareSymbol{1'}{0}{\node at (0,0.7) [noise1] {}; }
\DeclareSymbol{2'}{0}{\node at (0,0.7) [noise2] {}; }
\DeclareSymbol{3'}{0}{\node at (0,0.7) [noise3] {}; }

\DeclareSymbol{3'1'}{0}{\draw (-1,2) node[noise3] {}  -- (0.5,-0.2) node[noise1] {};}
\DeclareSymbol{2'2'}{0}{\draw (-1,2) node[noise2] {}  -- (0.5,-0.2) node[noise2] {};}
\DeclareSymbol{3'2'}{0}{\draw (-1,2) node[noise3] {}  -- (0.5,-0.2) node[noise2] {};}

\DeclareSymbol{0'0}{0}{\draw (0,-0.5) node[not] {} -- (0,1.7) node[noise0s] {};}
\DeclareSymbol{1'0}{0}{\draw (0,-0.5) node[not] {} -- (0,1.7) node[noise1s] {};}
\DeclareSymbol{2'0}{0}{\draw (0,-0.5) node[not] {} -- (0,1.7) node[noise2] {};}
\DeclareSymbol{3'0}{0}{\draw (0,-0.5) node[not] {} -- (0,1.7) node[noise3] {};}
\DeclareSymbol{1'1'}{0}{\draw (-1,1.5) node[noise1] {}  -- (0.5,0) node[noise1] {};}
\DeclareSymbol{2'1'}{0}{\draw (-1,2) node[noise2] {}  -- (0.5,0) node[noise1] {};}
\DeclareSymbol{2'1'0}{0}{\draw (1.5,2) node[noise2] {}  -- (0,0.5) node[noise1] {} -- (1.5,-1) node[not] {};}
\DeclareSymbol{2'2'0}{-2}{\draw (-1,1) node[noise2] {}  -- (0,-0.5) node[smalldot] {} -- (1,1) node[noise2] {};}
\DeclareSymbol{2'1'1'}{-1}{\draw (0,1.4) node[noise2] {}  -- (1.5,0.5) node[noise1] {} -- (0,-0.4) node[noise1] {};}
\DeclareSymbol{2'0'}{0}{\draw (-1,1.5) node[noise2] {}  -- (0.5,0) node[noise0] {};}
\DeclareSymbol{2'2'0'}{0}{\draw (-1,1.5) node[noise2] {}  -- (0,0) node[noise0] {} -- (1,1.5) node[noise2] {};}

\DeclareSymbol{1}{0}{\draw[white] (-.4,0) -- (.4,0); \draw (0,-0.2)  -- (0,1.6) node[dot] {};}
\DeclareSymbol{2}{0}{\draw (-0.5,1.2) node[dot] {} -- (0,0) -- (0.5,1.2) node[dot] {};}
\DeclareSymbol{3}{0}{\draw (0,0) -- (0,1.2) node[dot] {}; \draw (-.7,1) node[dot] {} -- (0,0) -- (.7,1) node[dot] {};}
\DeclareSymbol{30}{-3}{\draw (0,0) -- (0,-1); \draw (0,0) -- (0,1.2) node[dot] {}; \draw (-.7,1) node[dot] {} -- (0,0) -- (.7,1) node[dot] {};}
\DeclareSymbol{31}{-3}{\draw (0,0) -- (0,-1) -- (1,0) node[dot] {}; \draw (0,0) -- (0,1.2) node[dot] {}; \draw (-.7,1) node[dot] {} -- (0,0) -- (.7,1) node[dot] {};}
\DeclareSymbol{32}{-3}{\draw (0,0) -- (0,-1) -- (1,0) node[dot] {}; \draw (0,0) -- (0,-1) -- (-1,0) node[dot] {}; \draw (0,0) -- (0,1.2) node[dot] {}; \draw (-.7,1) node[dot] {} -- (0,0) -- (.7,1) node[dot] {};}
\DeclareSymbol{20}{-3}{\draw (0,0) -- (0,-1);\draw (-.7,1) node[dot] {} -- (0,0) -- (.7,1) node[dot] {};}
\DeclareSymbol{22}{-3}{\draw (0,0.3) -- (0,-1) -- (1,0) node[dot] {}; \draw (0,0.3) -- (0,-1) -- (-1,0) node[dot] {};\draw (-.7,1) node[dot] {} -- (0,0.3) -- (.7,1) node[dot] {};}
\DeclareSymbol{31p}{-3}{\draw (0,0) -- (0,-1) -- (1,0) node[dot] {}; \draw (0,0) -- (0,1.2) node[dot] {}; \draw (-.7,1) node[dot] {} -- (0,0) -- (.7,1) node[dot] {}; \draw (0,-1) ; }
\DeclareSymbol{32p}{-3}{\draw (0,0) -- (0,-1) -- (1,0) node[dot] {}; \draw (0,0) -- (0,-1) -- (-1,0) node[dot] {}; \draw (0,0) -- (0,1.2) node[dot] {}; \draw (-.7,1) node[dot] {} -- (0,0) -- (.7,1) node[dot] {}; \draw (0,-1) ;}
\DeclareSymbol{22p}{-3}{\draw (0,0.3) -- (0,-1) -- (1,0) node[dot] {}; \draw (0,0.3) -- (0,-1) -- (-1,0) node[dot] {};\draw (-.7,1) node[dot] {} -- (0,0.3) -- (.7,1) node[dot] {}; \draw (0,-1) ;}

\setlist[itemize]{topsep=3pt,itemsep=1.5pt,parsep=0pt}

\def\scal#1{\langle#1\rangle}
\def\cent#1{\mathopen{{\langle\kern-0.3em\rangle}}#1\mathclose{{\langle\kern-0.3em\rangle}}}

\def\d{\partial}

\begin{document}

\title{Weak universality of $\Phi^4_3$: polynomial potential and general smoothing mechanism}
\author{Dirk Erhard$^1$ and Weijun Xu$^2$}
\institute{Universidade Federal da Bahia, Brazil.\; \email{erharddirk@gmail.com}
	\and University of Oxford, UK.\; \email{weijunx@gmail.com}}

\maketitle

\begin{abstract}
	We consider a class of stochastic reaction-diffusion equations on the three dimensional torus. The non-linearities are odd polynomials in the weakly non-linear regime, and the smoothing mechanisms are very general higher order perturbations of the Laplacian. The randomness is the space-time white noise without regularisation. We show that these processes converge to the dynamical $\Phi^4_3 (\lambda)$ model, where the coupling constant $\lambda$ has an explicit expression involving non-trivial interactions between all details of the smoothing mechanism and the non-linearity. 
\end{abstract}

\setcounter{tocdepth}{2}
\microtypesetup{protrusion=false}
\tableofcontents
\microtypesetup{protrusion=true}

\def\k{\mathbf{k}}

\section{Introduction}
\label{sec:intro}

\subsection{Statement of the main result}

The dynamical $\Phi^4_3(\lambda)$ model ($\lambda>0$) is the equation formally given by
\begin{equation} \label{eq:phi43}
\d_t \phi = \Delta \phi - \lambda \phi^3 + \xi\;, 
\end{equation}
where $\xi$ is the space-time white noise on the three dimensional torus $\T^3 = (\R / \Z)^{3}$. 
There are two principle reasons that sparked interest in the equation. Firstly, the formal equilibrium measure of the dynamics~\eqref{eq:phi43} is the measure on Schwartz distributions associated to Bosonic Euclidean quantum field theory. The construction of this measure was one of the main achievements of constructive field
theory in the seventies; see for instance the articles~\cite{EO,Feldman,FO,GJ,Glimm} and references therein.

Secondly, the solution to~\eqref{eq:phi43} is expected to describe the 3D Ising model with Glauber dynamics and Kac interactions near critical temperature (see~\cite{GLP99}). The one dimensional version of this result was shown in~\cite{BPRS93}. The two dimensional situation requires a renormalisation to the equation. It was shown in~\cite{Ising2d} that the 2D dynamical Kac-Ising model does rescale to~\eqref{eq:phi43} in $2$ dimensions, and the renormalisation constant has a beautiful interpretation as a shift of the temperature from its mean field value. The 3D case is expected to be much more involved. 

The problem with~\eqref{eq:phi43} for $d\geq 2$ is that the equation is not well posed. Indeed, the roughness of the noise $\xi$ forces the solutions to~\eqref{eq:phi43} to be distrubitions rather than functions. Therefore, the cubic term lacks any interpretation. The $2$D case was resolved by Da Prato and Debussche~\cite{DPD03} using a first order expansion around the solution to the corresponding linear equation. This type of ``global expansion'' breaks down for dimension three, and the problem stayed open until Hairer in his breakthrough paper~\cite{rs_theory} developed the theory of regularity structures allowing expanding the solution around each space-time point systematically. The local well-posedness of the $3$D problem comes as an application. The theory has now been developed into a blackbox machinary that allows to solve essentially all subcritical SPDEs automatically (\cite{rs_algebraic, rs_analytic, rs_equation}), including the ``$4-\delta$'' dimensional case. The theory of para-controlled distributions developed in \cite{para_control_theory} also allows to tackle a large class of singular SPDEs, and the well-posedness of \eqref{eq:phi43} in $3$D was also shown in~\cite{phi43_CC} using paracontrolled distributions. Kupiainen also developed renormalisation group arguments in~\cite{phi43_Antti}. The local well-posedness of~\eqref{eq:phi43} can be loosely stated as follows. 

\begin{thm}\label{thm:standard}
	Given a smooth approximation $\xi_\eps$ of the space-time white noise $\xi$, there exists a sequence of constants $C_\eps \rightarrow +\infty$ as $\eps\to 0$ such that the solution $\phi_\eps$ to the regularised equation
	\begin{equation*}
	\d_t \phi_\eps = \Delta \phi_\eps - \lambda \phi_\eps^3 + \xi_\eps + C_\eps \phi_\eps
	\end{equation*}
	converges in probability as $\eps \rightarrow +\infty$. $C_\eps$ has the form $C_\eps = \frac{c_1}{\eps} + c_2 \log \eps +\oO(1)$. The family of limits parametrised by the $\oO(1)$ quantity is the family of solutions to $\Phi^4_3(\lambda)$. 
\end{thm}

We will give a characterisation of one particular element in the limiting family in Appendix~\ref{app:pde}. In the case $\lambda>0$, global well-posedness and quantitative bounds on large scale behaviour of the solution for \eqref{eq:phi43} and more singular equations of the same form have been established in \cite{Phi43global, Phi43_local_bounds, Phi4_sub_local_bounds}. 

In this article, we study the $\eps \rightarrow 0$ limit of the family of processes $\Phi_\eps$ satisfying the equation
\begin{equation} \label{eq:main_eq_intro}
\d_t \Phi_\eps = \lL_\eps \Phi_\eps - \eps^{-\frac{3}{2}} V'(\sqrt{\eps} \Phi_\eps) + \xi + C_\eps \Phi_\eps\;, \quad (t,x) \in \R^{+} \times \T^3\;, 
\end{equation}
where $\xi$ is the space-time white noise on $\R \times \T^3$, $V$ is an even polynomial of degree at least $4$, and the differential operator is of the form $\lL_\eps = - \eps^{-2} \qQ(i \eps \nabla)$ in the sense that its Fourier transform is given by $\widehat{\lL_\eps}(k) = - \eps^{-2} \qQ(2 \pi \eps k)$. Here $\qQ$ is a function satisfying Assumption~\ref{as:Q} below. Finally, $C_\eps$ is a constant depending on $\eps$ and its precise form will be determined in the sequel.

\begin{assumption} \label{as:Q}
	$\qQ$ is radially symmetric, and its radial version (also denoted by $\qQ$) $\qQ: \R^{+} \rightarrow \R$ has five continuous derivatives and satisfies the following: 
	\begin{enumerate}
		\item $Q(0)=0 $ and $\frac{1}{2} \qQ''(0) = 1$.  In particular, for every $\Lambda >0$ there exists $C>0$ such that
		\begin{equation} \label{eq:Q_origin}
		|\qQ(z) - z^2| \leq C z^{4}\;, \qquad \forall z \in [0,\Lambda]\,.
		\end{equation}
		
		\item $\qQ(z) > 0$ for all $z > 0$. 
		
		\item There exists $c>0$ and $\eta>0$ such that
		\begin{equation} \label{eq:Q_growth}
		\qQ(z) > c z^{3+\eta}\;, \qquad \forall z \geq 1\;.
		\end{equation}
		
		\item For every $\delta \in (0,1)$, there exists $C_\delta>0$ such that
		\begin{equation} \label{eq:Q_derivative_growth}
		\max_{0 \leq n \leq 5} |z^n \qQ^{(n)}(z)| \leq C_\delta |\qQ(z)|^{1+\delta}\;, \quad \forall z \geq 1\,.
		\end{equation}
	\end{enumerate}
\end{assumption}

Note that the first three conditions together imply $\qQ(z) \geq c z^{3+\eta}$ for all $z$. Our main theorem is the following. 

\begin{thm} \label{th:main}
	Suppose $\qQ$ satisfies Assumption~\ref{as:Q}, and $V$ is an even polynomial of degree $2n$ with $n \geq 2$, and let $\Phi_\eps$ be the solution to \eqref{eq:main_eq_intro}. Let
	\begin{equation} \label{eq:variance_whole_space}
	\sigma^2 := \frac{1}{2} \int_{\R^3} \frac{1}{\qQ(2 \pi |\theta|)} {\rm d} \theta\;. 
	\end{equation}
	Suppose also that the initial conditions $\{\Phi_\eps(0,\cdot)\}$ converges as $\eps \rightarrow 0$ in a suitable space to some $\Phi(0,\cdot)$\footnote{See Remarks~\ref{rmk:notion_convergence} and~\ref{rmk:initial} for precise descriptions and detailed discussions.}. Then, there exist $c_1$, $c_2$ depending on $\qQ$ as well as on $V$ such that for $C_\eps = \frac{c_1}{\eps} + c_2 \log \eps + \oO(1)$, the processes $\{\Phi_\eps\}_{\eps >0}$ converge as $\eps \to 0$ to the $\Phi^4_3(\lambda)$ family of solutions with initial data $\Phi(0,\cdot)$ with
	\begin{equation} \label{eq:lambda_value}
	\lambda = \frac{1}{6} (V^{(4)} * \mu)(0)\;, 
	\end{equation}
	where $\mu \sim \nN(0,\sigma^2)$ with $\sigma^2$ given as above. 
\end{thm}

\begin{rmk} \label{rmk:notion_convergence}
	The precise notion of convergence is as follows. The solutions $\Phi_\eps$ to \eqref{eq:main_eq_intro} and $\Phi$ to \eqref{eq:phi43} can be decomposed as
	\begin{equation} \label{eq:sol_decomp}
	\Phi_\eps = \<1>_\eps - \lambda \<3'0>_\eps + v_\eps + w_\eps\;, \qquad \Phi = \<1> - \lambda \<30> + v + w\;,
	\end{equation}
	where $\<1>_\eps$ and $\<1>$ are the stationary solutions to the linearised equations in \eqref{eq:free_field}, and $\<3'0>_\eps$ and $\<30>$ are stochastic objects specified in Section~\ref{sec:stochastic_objects}. $(v_\eps, w_\eps)$ and $(v,w)$ are remainders that satisfy the systems \eqref{eq:fixed_pt_eps} and \eqref{eq:fixed_pt_0}. Here we have replaced $\lL_\eps$ and $\Delta$ by $\lL_\eps-1$ and $\Delta-1$. The precise statement of Theorem~\ref{th:main} is that if there exists a (possibly random) distribution $\Phi(0,\cdot)$ such that the initial data $\Phi_\eps(0,\cdot) - \<1>_\eps(0,\cdot)$ converges to $\Phi(0,\cdot) - \<1>(0,\cdot)$ in $\bB^\kappa$ in probability, where $\bB^\kappa$ is the Besov space further specified in the appendix, then there is a random time $T$ such that $\|(v_\eps, w_\eps) - (v,w)\|_{\yY_{T,\eps}} \rightarrow 0$ in probability where $\yY_{T,\eps}$ is given as in Definition~\ref{de:space_solution_remainder}. Together with the convergence of $\<1>_\eps$ to $\<1>$ and $\<3'0>_\eps$ to $\<30>$ this justifies our notion of convergence\footnote{$\Phi_\eps$ converges to $\Phi$ means that each component in~\eqref{eq:sol_decomp} converges in their corresponding space}. Furthermore, if $\lambda$ in \eqref{eq:lambda_value} is positive, then the time $T$ can be taken to be any (deterministic) positive number. 
\end{rmk}

\begin{proof} [Proof of Theorem~\ref{th:main}]
	We give a quick proof of the main theorem assuming Theorems~\ref{th:fixed_pt_convergence} and~\ref{th:main_conv_stochastic_obj}. According to the decomposition \eqref{eq:sol_decomp} and the derivation in Section~\ref{sec:derivation_pde}, $(v_\eps, w_\eps)$ and $(v,w)$ satisfy the PDE systems \eqref{eq:fixed_pt_eps} and \eqref{eq:fixed_pt_0} with initial data
	\begin{equation*}
	\begin{split}
	v_\eps(0,\cdot) + w_\eps(0,\cdot) &= \Phi_\eps(0,\cdot) - \<1>_\eps(0,\cdot) + \lambda \<3'0>_\eps(0,\cdot)\;,\\
	v(0,\cdot) + w(0,\cdot) &= \Phi(0,\cdot) - \<1>(0,\cdot) + \lambda \<30>(0,\cdot)\;.
	\end{split}
	\end{equation*}
	By Theorem~\ref{th:main_conv_stochastic_obj}, arbitrarily high moments of $\<3'0>_\eps - \<30>$ converge in $\bB^{\frac{1}{2}-\kappa}$. Combined with the assumption on the convergence of the initial condition, we deduce that the sum $v_\eps(0,\cdot)+ w_\eps(0,\cdot)$ converges to $v(0,\cdot)+ w(0,\cdot)$ in $\bB^\kappa$. Hence, we can allocate the initial data of $(v_\eps, w_\eps)$ and $(v, w)$ such that the pair $(v_\eps(0,\cdot), w_\eps(0,\cdot))$ 
	converges in $\bB^{\kappa} \times \bB^{\kappa}$. Again by the convergence of the stochastic objects in Theorem~\ref{th:main_conv_stochastic_obj}, we see that the assumptions of Theorem~\ref{th:fixed_pt_convergence} are satisfied. Hence, $\|(v_\eps, w_\eps) - (v,w)\|_{\yY_T} \rightarrow 0$, which in turn implies the convergence of $\Phi_\eps$ to $\Phi$ in the sense of Remark~\ref{rmk:notion_convergence}. 
\end{proof}

\begin{rmk}
	The expressions \eqref{eq:variance_whole_space} and \eqref{eq:lambda_value} suggest that all details of $\qQ$ (and hence $\lL_\eps$) contribute to the limiting coupling constant $\lambda$. A typical function $\qQ$ satisfying Assumption~\ref{as:Q} is the even polynomial
	\begin{equation*}
	\qQ(z) = \sum_{j=1}^{q} \nu_j z^{2j}
	\end{equation*}
	with $q \geq 2$, $\nu_1 = 1$ and such that $\qQ(z)>0$ for all $z \neq 0$. In this case, the differential operator $\lL_\eps$ takes the form
	\begin{equation*}
	\lL_\eps = - \frac{1}{\eps^2} \qQ(i \eps \nabla) = \sum_{j=1}^{q} (-1)^{j-1} \nu_j \eps^{2(j-1)} \Delta^{j}. 
	\end{equation*}
	We see that all higher powers of the Laplacian vanish as $\eps \rightarrow 0$, but their coefficients $\nu_j$ still contribute to the coupling constant $\lambda$ of the limiting equation (even though the smoothing operator of the limiting equation is just the Laplacian with coefficient $\nu_1 = 1$). The same is true for $V'$, where all its coefficients also contribute to $\lambda$. 
\end{rmk}

\begin{eg}
	In the simplest nontrivial case where $\lL_\eps = \Delta - \nu \eps^2 \Delta^2$ for some $\nu > 0$, and $V$ consists of the sixth power only, the equation for $\Phi_\eps$ is
	\begin{equation*}
	\d_t \Phi_\eps = (\Delta - \nu \eps^2 \Delta^2) \Phi_\eps - a \eps \Phi_\eps^5 + \xi + C_\eps \Phi_\eps\;, \qquad \text{on}\; \; \R^+ \times \T^3\;. 
	\end{equation*}
	In this situation, with the proper choice of $C_\eps$, we have that $\Phi_\eps$ converges to $\Phi^4_3(\lambda)$ with
	\begin{equation*}
	\lambda = \frac{5a}{4 \pi^2} \int_{\R^3} \frac{1}{|\theta|^2 (1 + 4 \pi^2 \nu |\theta|^2)} {\rm d} \theta. 
	\end{equation*}
	We see that $\lambda$ depends on both $\nu$ and $a$ (even though $V'$ itself does not have a cubic term), and that $\lambda$ would not be defined if $\nu \leq 0$. 
\end{eg}

\begin{rmk}
	We now comment on the assumptions on $\qQ$. 
	\begin{enumerate}
		\item The first assumption ensures that $\lL_\eps$ approximates the Laplacian at low frequencies.  
		
		\item The positivity assumption says that $\lL_\eps$ should be ``smoothing'' at all scales. It is almost necessary in the sense that if $\qQ(z_0) < 0$ for some $z_0 \neq 0$, then even the deterministic linear evolution $e^{t \lL_\eps} f$ does not converge to $e^{t\Delta} f$ in $L^2$ unless one imposes very restrictive assumptions on the decay of $\widehat{f}$. 
		
		\item Since we consider the equation with the space-time white noise (not its regularised version), $\lL_\eps$ should have a sufficient smoothing effect in order for $\Phi_\eps$ to make sense even for fixed $\eps$. This requires $\qQ$ to grow sufficiently fast at infinity so that $\frac{1}{\qQ}$ is integrable (as can be seen from the expression \eqref{eq:variance_whole_space}). 
		
		On the other hand, if in \eqref{eq:main_eq_intro}, the noise $\xi$ is replaced by its Fourier truncated version $\xi_\eps$ as $\widehat{\xi_\eps}(t,k) := \rho(\eps |k|) \; \widehat{\xi}(t,k)$ with some nice cutoff function $\rho$, then Theorem~\ref{th:main} holds with $\lambda$ in the same form as \eqref{eq:lambda_value}, but the variance $\sigma^2$ of the Gaussian is given by
		\begin{equation*}
		\sigma^2 = \frac{1}{2} \int_{\R^3} \frac{\rho^2 (|\theta|)}{\qQ(2 \pi |\theta|)} {\rm d} \theta\;. 
		\end{equation*}
		In this case with the presence of the approximation to the noise, the growth assumption on $\qQ$ can be removed since $\rho^2$ is integrable at large scales. 
		
		\item Finally, the assumption~\ref{eq:Q_derivative_growth} ensures that the perturbed heat kernel $e^{t \lL_\eps}$ has the right smoothing properties with a $\delta$-loss at $t=0$ (see Lemma~\ref{le:heat_regularisation}). We do not know if this $\delta$-loss indeed happens or can be removed by improving our estimates.
		Note that even though one requires \eqref{eq:Q_derivative_growth} to hold for every $\delta>0$, this already includes a large class of functions with no restriction on their growth. What this condition prevents is the situation that the derivative oscillates much faster than the function grows. On the other hand, if we require \eqref{eq:Q_derivative_growth} to hold for $\delta=0$ with a finite proportionality constant, then this would restrict to functions bounded by polynomial growth. 
		We finally note that this condition is only used in the PDE part of the proof, and is not needed for the convergence of the stochastic objects.
	\end{enumerate}
\end{rmk}

\subsection{Motivation and related works}

\subsubsection{Hairer-Quastel universality}

Our work is not the first one in which the non-linearity of a singular SPDEs is generalised.
To the best of our knowledge, motivated by the weak universality conjecture of the KPZ equation, the first work in that context is by Hairer and Quastel~\cite{HQ}, who considered an approximation to the KPZ equation of the form
\begin{equation}\label{eq:HQ}
\partial_t h_\eps = \partial_x^2 h_\eps +\eps^{-1}F(\sqrt{\eps}\partial_x h) +\xi_\eps - C_\eps
\end{equation}
for some even polynomial $F$. They observed that as $\eps$ tends to zero one recovers the usual KPZ equation. Yet, the crux is that powers larger than two of $F$ do not simply vanish but they produce additional quadratic non-linearities. They therefore alter the coupling constant in front of the non-linearity which is a mean to regulate the strength of the asymmetry inherent to the equation. The Hairer-Quastel result has been extended to non-Gaussian noise (\cite{KPZCLT}) and general non-linearities (\cite{KPZ_general}). See also \cite{HQ_stationary} for similar results when the processes are at stationarity. 

As for the $\Phi_3^4$-equation, a similar universality result has been established in~\cite{Phi4_poly} where the approximations~\eqref{eq:main_eq_intro} with $\lL_\eps = \Delta$ and mollified noise $\xi_\eps$ was studied. It has been again observed that the non-linearity collapses into a cubic non-linearity with an altered coupling constant $\lambda$. These results have been extended to more general non-linearities and noises (\cite{Phi4_non_Gaussian, Phi4_general, Phi4_general_whole}), and the same phenomena was observed. 

It was therefore a curiosity to investigate the effect of higher smoothing mechanisms. Do they simply vanish as $\eps$ converges to zero, do they alter the strength of the final smoothing by producing a constant in front of the Laplacian, or is the effect different? Indeed, it was already expected in \cite[Section~1.3]{HQ} that these smoothing mechanisms should contribute to the limiting equation, and such effects was shown in \cite{Singular_perturb_SHE} in a simpler situation. Theorem~\ref{th:main} shows in the current case that the coupling constant is a function of $\qQ$ and therefore confirms that the last guess is the correct one, i.e., higher powers of the Laplacian produce additional non-linearities.

\subsubsection{Derivation from microscopic phase coexistence model}

Consider the microscopic process $\phi^{(\eps)}$ defined by
\begin{equation*}
\d_t \phi^{(\eps)} = \lL^{(\eps)} \phi^{(\eps)} - \eps V_{\theta}'(\phi^{(\eps)}) + \xi^{(\eps)}\;, \quad \text{on} \; \R^{+} \times [0,1/\eps]^{3}\;
\end{equation*}
with periodic boundary condition. Here, $\xi^{(\eps)}$ is the space-time white noise on $(\T/\eps)^{3}$, $\{V_{\theta}\}$ is a family of even polynomial potentials parametrised by $\theta$, and $\lL^{(\eps)} = -\qQ(i \nabla)$ is a differential operator on functions on $(\T/\eps)^{3}$ such that $\widehat{\lL^{(\eps)}}(\eta) = - \qQ(2 \pi \eta)$ for $\eta \in (\eps \Z)^{3}$. $V_\theta$ can be viewed as a perturbation of $V=V_0$ by a small parameter $\theta$. 

Consider the macroscopic process $\Phi_\eps$ defined by
\begin{equation*}
\Phi_{\eps}(t,x) := \eps^{-\frac{1}{2}} \phi^{(\eps)}(t/\eps^2, x/\eps)\;. 
\end{equation*}
Then $\Phi_\eps$ satisfies the equation
\begin{equation*}
\d_t \Phi_\eps = \lL_\eps \Phi_\eps - \eps^{-\frac{3}{2}} V_{\theta}'(\sqrt{\eps} \Phi_\eps) + \xi\;, \quad (t,x) \in \R^{+} \times \T^3
\end{equation*}
where $\lL_\eps = - \eps^{-2} \qQ(i \eps \nabla)$ and $\xi(t,x) = \eps^{-\frac{5}{2}} \xi^{(\eps)}(t/\eps^{2}, x/\eps)$ is the space-time white noise on $\T^3$. 

Let $\mu$ be the centered Gaussian measure on $\R$ with variance $\sigma^2$ given in \eqref{eq:variance_whole_space}. Define the averaged potential $\scal{V_\theta}$ to be
\begin{equation*}
\scal{V_\theta}(x) := \int_{\R} V_\theta (x+y) {\rm d}y\;.
\end{equation*}
Now, if $\scal{V}: (\theta, u) \mapsto V_{\theta}(u)$ satisfies the pitchfork bifurcation at the origin in the sense that
\begin{equation*}
\frac{\d^4 \scal{V}}{\d u^4}(0,0) > 0\;, \quad \frac{\d^2 \scal{V}}{\d u^2} = 0\;, \quad \frac{\d^3 \scal{V}}{\d \theta \d x^2}(0,0) < 0\;,
\end{equation*}
then the choice of $c_1$ in Theorem~\ref{th:main} is necessarily $0$, and for small $\theta$, the non-linearity behaves like
\begin{equation*}
- \eps^{-\frac{3}{2}} V_\theta'(\sqrt{\eps} \Phi_\eps) \approx - \eps^{-\frac{3}{2}} V'(\sqrt{\eps} \Phi_\eps) + \frac{c \theta}{|\log \eps|} \Phi_\eps\;.
\end{equation*}
With proper choice of $\theta(\eps) = \oO(\eps \log \eps)$, we are then in the form of \eqref{eq:main_eq_intro}. This says that when $\{V_\theta\}$ is "critical" in that it satisfies the pitchfork bifurcation, then the equation \eqref{eq:main_eq_intro} can be derived as the macroscopic process for phase coexistence model near criticality (with a small shift of $\theta(\eps) = \oO(\eps \log \eps)$ from its critical value).

\subsection{Structure of the article}

According to the argument given after Theorem~\ref{th:main}, the proof of the main theorem will be complete if we prove Theorems~\ref{th:fixed_pt_convergence} and~\ref{th:main_conv_stochastic_obj}. In Section~\ref{sec:outline}, we give an outline leading to the derivation of the PDE system, and in particular explains why $\lambda$ takes the value in \eqref{eq:lambda_value}. In Section~\ref{sec:PDE}, we prove Theorem~\ref{th:fixed_pt_convergence}, establishing the well-posedness and stability of the PDE. Section~\ref{sec:main_convergence} is devoted to the proof of Theorem~\ref{th:main_conv_stochastic_obj}, the convergence of all relevant stochastic objects. In the appendix, we give necessary backgrounds on Besov spaces, paraproducts and a brief description of the stochastic objects that arise from the standard dynamical $\Phi^4_3$ model.

\subsection{Notations}
\label{sec:notation}

We write
\begin{equation} \label{eq:notation_k}
\langle k\rangle_\eps =\sqrt{1+\eps^{-2}\qQ(2\pi \eps |k|)} \qquad \text{and} \qquad \langle k \rangle = \sqrt{1+4\pi^2|k|^2}\;.
\end{equation}
The above definition also works for $\eps=0$, and we have $\scal{k}_0 = \scal{k}$ in that situation. 

For $\alpha \in \R$, we write $\|\cdot\|_{\alpha}$ for $\|\cdot\|_{\bB^\alpha}$, the Besov norm for functions on the torus defined in \eqref{eq:norm_Besov}. We refer to Appendix~\ref{sec:Besov} for its precise definition and properties. Also, for $T>0$, $\theta \in (0,1)$ and $\alpha \in \R$, we use $\cC_{T}^{\alpha}$ and $\cC_{T}^{\theta,\alpha}$ to denote the following norms on space-time functions:
\begin{equation} \label{eq:notation_norm_spacetime}
\|f\|_{\cC_{T}^{\alpha}} := \sup_{t \in [0,T]} \|f(t)\|_{\alpha}\;, \qquad \|f\|_{\cC_T^{\theta,\alpha}} := \sup_{0 \leq < s < t \leq T} \frac{\|f(t)-f(s)\|_{\alpha}}{|t-s|^{\theta}}\;. 
\end{equation}
We also let $\iI_\eps$ and $\iI$ denote operators on space-time functions such that
\begin{equation} \label{eq:notation_op_I}
(\iI_\eps f)(t, \cdot) = \int_{0}^{t} e^{(t-r)(\lL_\eps-1)} f(r, \cdot) {\rm d}r\;, \quad (\iI f)(t,\cdot) = \int_{0}^{t} e^{(t-r)(\Delta-1)} f(r,\cdot) {\rm d}r\;. 
\end{equation}
We further use the notations
\begin{equation} \label{eq:notation_comm_heat}
\big[ e^{t(\lL_\eps-1)}, \prec \big] \big( f, g \big)= e^{t(\lL_\eps -1)} \big( f \prec g \big) - f \prec \big( e^{t(\lL_\eps -1)}g \big),
\end{equation}
as well as
\begin{equation} \label{eq:notation_comm_I}
\big[ \iI_\eps, \prec\big](f,g) = \iI_\eps (f\prec g)- f\prec \iI_\eps (g).
\end{equation}
Here, $\prec$ denotes the paraproduct introduced in Appendix~\ref{sec:Besov}. Moreover, properties that will be used in the sequel of the above norms and operators are provided therein as well.

We use $\kappa>0$ to denote a fixed constant that is sufficiently small which ensures that all arguments go through.

\subsection*{Acknowledgements}

D. E. gratefully acknowledges financial support from the National Council for Scientific and Technological Development - CNPq via a Universal grant 409259/2018-7 and a Bolsa de Produtividade 303520/2019-1. W.X. gratefully acknowledges financial support from the Engineering and Physical Sciences Research Council through the fellowship EP/N021568/1. 

Part of the work was completed when both authors visited The Hausdorff Research Institute for Mathematics during the junior trimester programme ``Randomness, PDEs and Nonlinear Fluctuations". We thank the hospitality and the financial support of HIM.

\section{Setting up the proof}
\label{sec:outline}

In this section, we give a heuristic explanation on why one expects Theorem~\ref{th:main} to be true. We then introduce proper setups and frameworks under which one can prove it rigorously. Since the operator $\lL_\eps$ takes a form that is convenient to work in Fourier space, we adapt the theory of paracontrolled distributions introduced in~\cite{para_control_theory} and its application to the $\Phi^4_3$ equation in~\cite{phi43_CC}. It will be interesting to work through all relevant bounds of the convolution kernel given by $(\d_t - \lL_\eps)^{-1}$ in the real space, and then to apply the general framework in \cite{rs_theory,rs_analytic, general_discrete} to prove convergence. 

For technical simplicity, we consider the equation
\begin{equation} \label{eq:main_Phi}
\d_t \Phi_\eps = (\lL_\eps - 1) \Phi_\eps - \eps^{-\frac{3}{2}}V'(\sqrt{\eps} \Phi_\eps) + \xi + C_\eps \Phi_\eps\;. 
\end{equation}
Compared to \eqref{eq:main_eq_intro}, we have subtracted the constant $1$ from $\lL_\eps$ to make the $0$-th Fourier mode stationary. This does not change the equation since one can add it back by changing $C_\eps$. The particular choice of $C_\eps$ which makes $\Phi_\eps$ converge to the limiting $\phi$ characterised in Appendix~\ref{app:pde} will be specified in Section~\ref{sec:derivation_pde} below. 

Let $\<1>_\eps$ and $\<1>$ denote the space-time stationary solutions to the linearised equations
\begin{equation} \label{eq:free_field}
\d_t \<1>_\eps = (\lL_\eps -1) \<1>_\eps + \xi\; \qquad \text{and} \qquad \d_t \<1> = (\Delta - 1) \<1> + \xi
\end{equation}
respectively. $\<1>_\eps$ is the key object from which all subsequent quantities are constructed from.

\subsection{Heuristic explanation of Theorem~\ref{th:main}}

We first explain why one expects Theorem~\ref{th:main} to be true, and in particular why $\lambda$, the coupling constant of the limiting equation, takes the form in \eqref{eq:lambda_value}, involving nontrivial contributions both from all higher order smoothings beyond the Laplacian, and from all higher order powers in $V'$ beyond the cubic term. 

If $\Phi_\eps$ solves \eqref{eq:main_eq_intro}, then the remainder $\zZ_\eps = \Phi_\eps - \<1>_\eps$ satisfies
\begin{equation*}
\d_t \zZ_\eps = (\lL_\eps-1) \zZ_\eps - \eps^{-\frac{3}{2}} V'\big( \sqrt{\eps} \<1>_\eps + \sqrt{\eps} \zZ_\eps \big) + C_\eps (\<1>_\eps + \zZ_\eps)\;.
\end{equation*}
The quantity $\sqrt{\eps} \<1>_\eps$ is asymptotically distributed as $\nN(0, \sigma^2)$ for $\sigma^2$ as in~\eqref{eq:variance_whole_space} (see Proposition~\ref{pr:ff_correlation}), and in analogy with the standard $\Phi^4_3$ equation, we expect $\zZ_\eps$ to be uniformly bounded. Hence, Taylor expanding $V'$ near the quantity $\sqrt{\eps} \<1>_\eps$ yields
\begin{equation} \label{eq:V_expansion}
\begin{split}
\eps^{-\frac{3}{2}} V'\big( \sqrt{\eps} \<1>_\eps + &\sqrt{\eps} \zZ_\eps \big)  = \eps^{-\frac{3}{2}} V'(\sqrt{\eps} \<1>_\eps) + \eps^{-1} V''(\sqrt{\eps} \<1>_\eps) \cdot \zZ_\eps\\
&+ \frac{1}{2\sqrt{\eps}} V^{(3)}(\sqrt{\eps} \<1>_\eps) \cdot \zZ_\eps^2 + \frac{1}{6} V^{(4)}(\sqrt{\eps} \<1>_\eps) \cdot \zZ_\eps^3 + \oO(\eps^{\frac{1}{2}-})\;.
\end{split}
\end{equation}
Since $\<1>_\eps$ is stationary Gaussian with an explicit variance, one can perform a chaos expansion to show that there exists a large constant $C_\eps^{(1)}$ such that for $\lambda$ given in \eqref{eq:lambda_value}, the quantities
\begin{equation} \label{eq:V_asymptotic_quantities}
\eps^{-\frac{3}{2}} V'(\sqrt{\eps} \<1>_\eps) - C_\eps^{(1)} \<1>_\eps, \; \; \eps^{-1} V''(\sqrt{\eps} \<1>_\eps) - C_\eps^{(1)}, \; \; \frac{1}{2 \sqrt{\eps}} V^{(3)}(\sqrt{\eps} \<1>_\eps), \; \; \frac{1}{6} V^{(4)}(\sqrt{\eps} \<1>_\eps)
\end{equation}
behave like $\lambda \<1>_\eps^{\diamond 3}$, $3 \lambda \<1>_\eps^{\diamond 2}$, $3 \lambda \<1>_\eps$ and $\lambda$ respectively. Here $\<1>_\eps^{\diamond j}$ denotes the $j$-th Wick power of $\<1>_\eps$. Plugging them back into the equation for $\zZ_\eps$, we get
\begin{equation*}
\d_t \zZ_\eps = (\lL_\eps - 1) \zZ_\eps - \lambda (\<1>_\eps + \zZ_\eps)^{\diamond 3} + (C_\eps - C_\eps^{(1)}) (\<1>_\eps + \zZ_\eps) + \rR_\eps\;, 
\end{equation*}
where the Wick product is with respect to the Gaussian structure of $\<1>_\eps$, and $\rR_\eps$ is an error term which vanishes in a proper sense as $\eps \rightarrow 0$. Now the above equation for $\zZ_\eps$ is almost the same as the remainder equation for the standard $\Phi^4_3 (\lambda)$ model (in particular, with $\lambda$ multiplying the cubic ``Wick" term), except that the Laplacian is replaced by $\lL_\eps$ and that there is an error term $\rR_\eps$. Hence, it is reasonable to expect that $\Phi_\eps$ convergences to $\Phi^4_3 (\lambda)$. 

We see that $\lambda$ arises as part of the coefficients of certain Wick powers in the expansion of the quantities in \eqref{eq:V_asymptotic_quantities}. Since $\sqrt{\eps} \<1>_\eps$ is asymptotically non-degenerate, these coefficients come as combined effects of all terms in $V'$ (except the linear one) and the variance of $\sqrt{\eps} \<1>_\eps$. The latter depends on higher order smoothing effects in $\lL_\eps$. Together they give the expression \eqref{eq:lambda_value}.


\subsection{Definition of the stochastic objects}
\label{sec:stochastic_objects}

Since part of the construction and estimates we are relying on have been already carried out and derived in~\cite{phi43_CC} and in~\cite{Phi43_pedestrians}, we only sketch some of the arguments and refer to the articles just alluded to for more details. 

Recall from \eqref{eq:lambda_value} that $\lambda = \frac{1}{6} \E [V^{(4)}\big( \nN(0,\sigma^2) \big)]$. In view of the quantities appearing in the expansion \eqref{eq:V_expansion}, it is natural to introduce processes $\<0'>_\eps$, $\<1'>_\eps$, $\<2'>_\eps$ and $\<3'>_\eps$ by setting
\begin{equation} \label{eq:noises_V}
\begin{split}
\<0'>_\eps &:= \frac{1}{6\lambda} V^{(4)}(\sqrt{\eps} \<1>_\eps)\;, \qquad \qquad \<1'>_\eps := \frac{1}{6\lambda \sqrt{\eps}} V^{(3)}(\sqrt{\eps} \<1>_\eps)\;,\\
\<2'>_\eps &:= \frac{1}{3 \lambda \eps} V''(\sqrt{\eps} \<1>_\eps) - C_{\eps}^{(1)}\;, \quad \phantom{1} \<3'>_\eps := \frac{1}{\lambda \eps^{3/2}} V'(\sqrt{\eps} \<1>_\eps) - 3 C_{\eps}^{(1)} \<1>_\eps\;.\\
\end{split}
\end{equation}
Here, the constant $C_{1}^{(\eps)}$ is given by
\begin{equation} \label{eq:C1}
C_{\eps}^{(1)} = \frac{1}{3 \lambda \eps} \E V''(\sqrt{\eps} \<1>_\eps)\;. 
\end{equation}
The reason for dividing by multiples of $\lambda$ is to normalise, so that $\<0'>_\eps$, $\<1'>_\eps$ and $\<2'>_\eps$ should converge to the constant $1$, the free field $\<1>$ and its Wick square $\<1>^{\diamond 2}$, while $\<3'>_\eps$ should behave like $\<1>_\eps^{\diamond 3}$, which converges only after further convolution with the heat kernel. 

We now introduce two new processes $\<2'0>_\eps$ and $\<3'0>_\eps$ by setting
\begin{equation} \label{eq:noises_integration}
\<2'0>_\eps(t) = \int_{-\infty}^{t} e^{(t-r)(\lL_\eps-1)} \<2'>_\eps(r) {\rm d}r\;, \qquad \<3'0>_\eps(t) = \int_{-\infty}^{t} e^{(t-r)(\lL_\eps-1)} \<3'>_\eps(r) {\rm d}r\;.
\end{equation}
Note that these are different from $\iI_\eps(\<2'>_\eps)$ and $\iI_\eps(\<3'>_\eps)$ since the integration in time starts from $-\infty$. Therefore they are stationary in both space and time. Finally, we define $\<2'2'>_\eps$, $\<3'1'>_\eps$ and $\<3'2'>_\eps$ by
\begin{equation} \label{eq:noises_2nd}
\begin{split}
\<2'2'>_\eps &:= \<2'0>_\eps \circ \<2'>_\eps - C_\eps^{(2)}\;, \qquad \<3'1'>_\eps := \<3'0>_\eps \circ \<1'>_\eps - C_{\eps}^{(3)}\;,\\
\<3'2'>_\eps &:= \<3'0>_\eps \circ \<2'>_\eps - (3 C_\eps^{(2)} + 2 C_\eps^{(3)}) \<1>_\eps\;,
\end{split}
\end{equation}
where $C_\eps^{(2)}$ and $C_\eps^{(3)}$ are given by
\begin{equation} \label{eq:C2C3}
C_{\eps}^{(2)} = \E \big[ \; \<2'0>_\eps \circ \<2'>_\eps \; \big]\;, \quad\text{and}\quad C_\eps^{(3)} = \E \big[ \; \<3'0>_\eps \circ \<1'>_\eps \; \big]\;.
\end{equation}
The $C_{\eps}^{(j)}$'s above do not depend on $(t,x)$ since all the processes defined above are space-time stationary. One can see from the expression of \eqref{eq:C1} that $C_\eps^{(1)}$ diverges at order $\eps^{-1}$. We will see from Proposition~\ref{pr:C2C3_expression} below that $C_\eps^{(2)}$ diverges logarithmically while $C_\eps^{(3)}$ is uniformly bounded in $\eps$. 

The above definition of the stochastic objects will ensure that
\begin{equation*}
\big(\;\<0'>_\eps\;, \<1'>_\eps\;, \<2'>_\eps\;, \<3'0>_\eps\;, \<3'1'>_\eps\;, \<2'2'>_\eps\;, \<3'2'>_\eps \big) \rightarrow \big( \; 1\;, \<1>\;, \<2>\;, \<30>\;, \<31>\;, \<22>\;, \<32> \big)
\end{equation*}
as $\eps\to 0$ in a suitable topology. The symbols on the right hand side above are the stochastic objects that appear in the standard $\Phi^4_3$ model, which are described in Appendix~\ref{app:stochastic}. For convenience, we have summarised our processes, their corresponding limiting objects and their Besov regularities in Table~\ref{table:noise}.

{\small
	\begin{table}
		\centering
		\renewcommand{\arraystretch}{1.5}
		\begin{tabular}{cccccccc}
			\toprule
			$\tau_\eps:$ & $\<0'>_\eps$ & $\<1'>_\eps$ & $\<2'>_\eps$ & $\<3'0>_\eps$ & $\<3'1'>_\eps$ & $\<2'2'>_\eps$ & $\<3'2'>_\eps$ 
			\\
			\midrule
			$\Phi^4_3:$ & $1$ & $\<1>$ & $\<2>$ & $\<30>$ & $\<31>$ & $\<22>$ & $\<32>$
			\\
			\midrule
			Besov reg.:. & $-\kappa$ & $-\frac{1}{2} - \kappa $ & $-1 - \kappa$ & $\frac{1}{2} - \kappa$ & $-\kappa$ & $- \kappa$ & $-\frac{1}{2}-\kappa$
			\\
			\bottomrule
		\end{tabular}
		\bigskip
		\caption{List of processes, their limits and their respective regularities.}
		\label{table:noise}
	\end{table}
}

\subsection{Formal derivation of the (system of) PDEs for the remainder}
\label{sec:derivation_pde}

We now start to formally derive a system of equations for the remainder. The solution theory for this perturbed equation is essentially the same as that of $\Phi^4_3$. We give a brief description below for the sake of completeness. We follow the formulation in \cite{Phi43global}. We choose $C_\eps$ to be
\begin{equation} \label{eq:C_eps_split}
C_{\eps} = 3 \lambda C_{\eps}^{(1)} - 9 \lambda^2 C_{\eps}^{(2)} - 6 \lambda^2 C_{\eps}^{(3)}\;, 
\end{equation}
where $C_\eps^{(1)}$, $C_\eps^{(2)}$ and $C_\eps^{(3)}$ are given in \eqref{eq:C1} and \eqref{eq:C2C3}. In view of \eqref{eq:V_expansion} and that $\eps^{-\frac{3}{2}} V'(\sqrt{\eps} \<1>_\eps)$ behaves like $\lambda \<3'>_\eps$, it is natural to add $\lambda \<3'0>_\eps$ to the remainder $\zZ_\eps$ and consider
\begin{equation*}
u_\eps := \zZ_\eps + \lambda \<3'0>_\eps = \Phi_\eps - \<1>_\eps + \lambda \<3'0>_\eps\;.
\end{equation*}
Using the definition of the processes in \eqref{eq:noises_V} and the choice of $C_\eps$ in \eqref{eq:C_eps_split}, we see that $u_\eps$ satisfies the equation
\begin{equation*}
\begin{split}
\d_t u_\eps = &(\lL_\eps - 1) u_\eps - 3 \lambda \; \<2'>_\eps \; (u_\eps - \lambda \<3'0>_\eps) - 3 \lambda \; \<1'>_\eps \; (u_\eps - \lambda \<3'0>_\eps)^{2} - \lambda \; \<0'>_\eps \; (u_\eps - \lambda \<3'0>_\eps)^{3}\\
&\phantom{11}- \eps^{-\frac{3}{2}} V'\big(\sqrt{\eps} \<1>_\eps; \sqrt{\eps} (u_\eps - \lambda \<3'0>_\eps) \big) - \big(9 \lambda^2 C_\eps^{(2)} + 6 \lambda^2 C_{\eps}^{(3)} \big) (\<1>_\eps + u_\eps - \lambda \<3'0>_\eps)\;,
\end{split}
\end{equation*}
where
\begin{equation} \label{eq:V_remainder}
V'(x;y) := V'(x+y) - \sum_{j=0}^{3} \frac{1}{j!} V^{(j+1)}(x) \cdot y^{j}\;
\end{equation}
denotes the Taylor remainder. 

Because of the term $\<2'>_\eps \; (u_\eps - \lambda \<3'0>_\eps)$ on the right hand side, the best regularity one can hope for $u_\eps$ (uniformly in $\eps$) is $\cC^{1-}$, which does not allow us to close the loop since both $\<2'>_\eps \; (u_\eps - \lambda \<3'0>_\eps)$ and $\<1'>_\eps \; (u_\eps - \lambda \<3'0>_\eps)^2$ involve products between terms that are below the threshold of analytic well-posedness (in the limit as $\eps \rightarrow 0$). We first decompose these two products into paraproducts, and combine with part of the renormalisation to get
\begin{equation} \label{eq:u_eps}
\begin{split}
\d_t u_\eps = &(\lL_\eps - 1) u_\eps - 3 \lambda (u_\eps - \lambda \<3'0>_\eps) \prec \<2'>_\eps - 3 \lambda (u_\eps - \lambda \<3'0>_\eps) \succ \<2'>_\eps - 3 \lambda u_\eps \circ \<2'>_\eps\\
&- \sum_{j=0}^{3} \widetilde{F}_j u_\eps^j - \eps^{-\frac{3}{2}} V'\big( \sqrt{\eps} \<1>_\eps; \sqrt{\eps} (u_\eps - \lambda \<3'0>_\eps) \big) - 9 \lambda^2 C_\eps^{(2)} (u_\eps - \lambda \<3'0>_\eps)\;.
\end{split}
\end{equation}
where the coefficients $\widetilde{F}_j$ are given by
\begin{equation} \label{eq:coeff_F_tilde}
\begin{split}
&\widetilde{F}_3 = - \lambda \; \<0'>_\eps\;, \qquad \qquad \widetilde{F}_2 = 3 \lambda^2 \; \<0'>_\eps \<3'0>_\eps - 3 \lambda \<1'>_\eps\;,\\
&\widetilde{F}_1 = - 3 \lambda^3 \; \<0'>_\eps \big( \<3'0>_\eps \big)^{2} + 6 \lambda^2 \Big[ \<3'0>_\eps\prec \<1'>_\eps + \<3'0>_\eps \succ \<1'>_\eps + \big( \underbrace{\<3'0>_\eps \circ \<1'>_\eps - C_\eps^{(3)}}_{\<3'1'>_\eps} \big) \Big]\;,\\
&\widetilde{F}_0 = \lambda^4 \; \<0'>_\eps \big(\<3'0>_\eps \big)^{3} - 3 \lambda^3 \Big[ \big( \<3'0>_\eps \big)^{2} \prec \<1'>_\eps + \big( \<3'0>_\eps \big)^{2} \succ \<1'>_\eps + \big( \<3'0>_\eps \circ \<3'0>_\eps \big) \circ \<1'>_\eps\\
&+ 2 \big( \underbrace{\<3'0>_\eps \circ \<1'>_\eps - C_\eps^{(3)}}_{\<3'1'>_\eps} \big) \<3'0>_\eps + 2 \Com \big( \<3'0>_\eps; \<3'0>_\eps; \<1'>_\eps \big) \Big] + 3 \lambda^2 \big( \underbrace{\<3'0>_\eps \circ \<2'>_\eps - (3 C_\eps^{(2)} + 2 C_\eps^{(3)} ) \<1>_\eps}_{\<3'2'>_\eps} \big)\;, 
\end{split}
\end{equation}
and the commutator operator $\Com$ is given in Proposition~\ref{pr:commutator_estimate}. We see that if all the stochastic objects live in their corresponding regularity spaces (as in Table~\ref{table:noise}) and that if $u_\eps \in \cC^{1-}$, then all of the above terms would be well defined except the resonance product $u_\eps \circ \<2'>_\eps$. 

For this term we need to employ the structures of $u_\eps$ inherited by the fact that it solves the equation \eqref{eq:u_eps}, and combine it with the remaining renormalisation terms to give a meaningful expression of this reasonance product. To employ such structures given by the equation, we split $u_\eps$ into $v_\eps + w_\eps$ where $v_\eps$ satisfies
\begin{equation} \label{eq:v_eps}
\d_t v_\eps = (\lL_\eps - 1) v_\eps - 3 \lambda (v_\eps + w_\eps - \lambda \<3'0>_\eps) \prec \<2'>_\eps.
\end{equation}
From the equation for $v_\eps$ and the regularities of the stochastic objects, it is natural to expect that $v_\eps$ inherits the regularity of $u_\eps$ while $w_\eps$ is almost in $\bB^{\frac{3}{2}-}$ and hence $\<2'>_\eps \circ w_\eps$ can be defined uniformly in $\eps$. To treat the only problematic term $\<2'>_\eps \circ v_\eps$, we use the equation for $v_\eps$ so that
\begin{equation*}
v_\eps (t) = e^{t(\lL_\eps - 1)} v_\eps(0) - 3 \lambda (u_\eps - \lambda \<3'0>_\eps) \prec \iI_\eps(\<2'>_\eps) - 3 \lambda [\iI_\eps, \prec](u_\eps - \lambda \<3'0>_\eps, \<2'>_\eps)\;, 
\end{equation*}
where we recall the operator $\iI_\eps$ and commutator $[\iI_\eps, \prec]$ from \eqref{eq:notation_op_I} and \eqref{eq:notation_comm_I}.Plugging this expression into $\<2'>_\eps \circ v_\eps$ and combine it with the remaining renormalisation, we get
\begin{equation} \label{eq:v_resonance}
\begin{split}
- 3 \lambda \<2'>_\eps &\circ v_\eps - 9 \lambda^2 C_\eps^{(2)} (u_\eps - \lambda \<3'0>_\eps) = - 3 \lambda \<2'>_\eps \circ e^{t(\lL_\eps-1)} v_\eps(0)\\
&+ 9 \lambda^2 \Big[ \<2'>_\eps \circ [\iI_\eps, \prec](u_\eps - \lambda \<3'0>_\eps, \<2'>_\eps) + \Com \big(u_\eps - \lambda \<3'0>_\eps; \iI_\eps(\<2'>_\eps); \<2'>_\eps \big)\\
&+ \big( \underbrace{\<2'0>_\eps \circ \<2'>_\eps - C_\eps^{(2)}}_{\<2'2'>_\eps} - \; \<2'>_\eps \circ e^{t(\lL_\eps-1)} \<2'0>_\eps(0) \big) (u_\eps - \lambda \<3'0>_\eps) \Big]\;, 
\end{split}
\end{equation}
where we have used $(\iI_\eps \<2'>_\eps)(t) = \<2'0>_\eps (t) - e^{t(\lL_\eps-1)} \<2'0>_\eps(0)$, and all processes unless indicated otherwise are evaluated at time $t$. If $u_\eps$ has a positive H\"older-in-time regularity, then all the above terms will be well defined. Hence, combining \eqref{eq:u_eps} and \eqref{eq:v_resonance}, we derive the system of $(v_\eps, w_\eps)$ as
\begin{equation} \label{eq:pde_system_formal}
\begin{cases}
&\d_t v_\eps = (\lL_\eps - 1) v_\eps - 3 \lambda (v_\eps + w_\eps - \lambda \<3'0>_\eps) \prec \<2'>_\eps\\
&\d_t w_\eps = (\lL_\eps - 1) w_\eps - 3 \lambda \big( e^{t(\lL_\eps-1)} v_\eps(0) + w_\eps \big) \circ \<2'>_\eps + G_\eps(v_\eps+w_\eps)
\end{cases}
\end{equation}
where
\begin{equation} \label{eq:G_formal}
\begin{split}
G_\eps(u) = & \sum_{j=0}^{3} F_j u^{j} - 3 \lambda (u - \lambda \<3'0>_\eps) \succ \<2'>_\eps - \eps^{-\frac{3}{2}} V'\big(\sqrt{\eps} \<1>_\eps; \sqrt{\eps} (u - \lambda \<3'0>_\eps) \big)\\
&+ 9 \lambda^2 \Big[ \Com \big( u - \lambda \<3'0>_\eps; \; \iI_\eps (\<2'>_\eps); \; \<2'>_\eps \big) + \<2'>_\eps \circ [\iI_\eps, \prec](u - \lambda \<3'0>_\eps, \<2'>_\eps)\\
&\phantom{111}- \big( \<2'>_\eps \circ e^{t(\lL_\eps-1)} \<2'0>_\eps(0) \big) \cdot (u - \lambda \<3'0>_\eps) \Big]\;, 
\end{split}
\end{equation}
and the coefficients $F_j$ are given by
\begin{equation} \label{eq:coeff_F_formal}
F_3 = \widetilde{F}_3\;, \quad F_2 = \widetilde{F}_2\;, \quad F_1 = \widetilde{F}_1 + 9 \lambda^2 \; \<2'2'>_\eps\;, \quad F_0 = \widetilde{F}_0 - 9 \lambda^3 \; \<2'2'>_\eps \cdot \<3'0>_\eps\;,
\end{equation}
with $\widetilde{F}_j$ given in \eqref{eq:coeff_F_tilde}.

\begin{rmk}
	In addition to the operator $\lL_\eps$ and the remainder term $\eps^{-\frac{3}{2}} V'(\sqrt{\eps} \cdot; \sqrt{\eps} \cdot)$, the system \eqref{eq:pde_system_formal} is also different from that in \cite{Phi43global, phi43_CC} in that our last term in $G_\eps$ (the one involving $e^{t(\lL_\eps-1)}\<2'0>_\eps(0)$) is extra. The reason for the appearance of this additional term is that our definition of $\<2'0>_\eps$ is slightly different -- the integration of heat kernel starts from $-\infty$ rather than $0$, and hence both $\<2'0>_\eps$ and $\<2'2'>_\eps$ are stationary. 
	
	If we define $\<2'0>_\eps$ to be the same as $\iI_\eps (\<2'>_\eps)$ so that $\<2'0>_\eps(0) = 0$, then there would be no such term, but the trade-off is that the process $\<2'2'>_\eps$ would not be stationary in time. 
\end{rmk}

\section{Solution theory}
\label{sec:PDE}

In this section, we prove pathwise well-posedness and stability properties of the system \eqref{eq:pde_system_formal}. It is then natural to introduce spaces that encode the pathwise analytic properties of the relevant stochastic processes. In the rest of this section, all functions/distributions in relevant spaces are treated deterministically, and the only information used are their relevant norms.

\subsection{The fixed point equation}\label{sec:fixed_point}

We start by introducing the relevant spaces for the external functions/distribution as well as the space in which we are going to construct the solution pair $(v_\eps, w_\eps)$. 

For every $T>0$, let $\xX_T$ be the space of a collection of continuous evolutions in certain Besov spaces (to be specified in Definition~\ref{de:norm_enhanced_noise}) up to time $T$. We denote a generic element in $\xX_T$ by
\begin{equation} \label{eq:enhanced_noise}
\Upsilon = \big(\; \<0'>\;, \<1'>\;, \<2'>\;, \<3'0>\;, \<3'1'>\;, \<2'2'>\;, \<3'2'> \; \big)\;. 
\end{equation}

\begin{defn} \label{de:norm_enhanced_noise}
We define the norm $\|\cdot\|_{\xX_T}$ by
\begin{equation} \label{eq:norm_enhanced_noise}
\|\Upsilon\|_{\xX_T} := \sum_{\tau} \sup_{t \in [0,T]} \|\tau(t)\|_{|\tau|} + \sup_{0 \leq s < t \leq T} \frac{\|\<3'0>(t) - \<3'0>(s)\|_{\frac{1}{4}-\kappa}}{|t-s|^{\frac{1}{8}}}\;, 
\end{equation}
where the sum is taken over all the seven components in $\Ups$, $|\tau|$ is the homogeneity of the component $\tau$ as specified in Table~\ref{table:noise}, and $\|\cdot\|_{|\tau|}$ is the Besov norm as defined in \eqref{eq:norm_Besov}. We also write $\xX = \xX_1$ for $T=1$. 
\end{defn}

The symbols in \eqref{eq:enhanced_noise} are abstract placeholders for generic elements in $\xX_T$. We do not assume any relationship between them and the processes introduced in Section~\ref{sec:stochastic_objects}. 

Let $\{\psi_\eps\}$ and $\{h_\eps\}$ be families of space-time functions indexed by $\eps \in (0,1)$. Let $h$ be another space-time function. We will specify relevant analytic bounds for these functions later but we mention already now that they play the roles of $\sqrt{\eps} \<1>_\eps$, $e^{t(\lL_\eps-1)} \<2'0>_\eps(0)$ and its limit $e^{t(\Delta-1)} \<20>(0)$ respectively. In view of \eqref{eq:G_formal}, \eqref{eq:coeff_F_formal} and \eqref{eq:coeff_F_tilde}, for every $\lambda \in \R$, $\eps \in (0,1)$ and every $\Upsilon \in \xX_T$, we define a map $G_{\eps}(\lambda,\Upsilon,\cdot)$ on the set of sufficiently regular space-time functions by
\begin{equation} \label{eq:G_eps}
\begin{split}
G_{\eps}(&\lambda,\Upsilon,u) := \sum_{j=0}^{3} F_j(\lambda,\Upsilon) u^{j} - 3 \lambda (u-\lambda \<3'0>) \succ \<2'> - \eps^{-\frac{3}{2}} V'\big(\sqrt{\eps} \psi_\eps; \sqrt{\eps} (u - \lambda \<3'0>) \big)\\
&+ 9 \lambda^2 \Big[ \Com \big(u-\lambda\<3'0>; \iI_\eps(\<2'> \big); \<2'>) + \<2'> \circ [\iI_\eps, \prec](u-\lambda \<3'0>, \<2'>) - (\<2'> \circ h_\eps) \cdot (u - \lambda \<3'0>) \Big]\;,
\end{split}
\end{equation}
where the coefficients $F_j$ are given by
\begin{equation} \label{eq:F_j_coefficients}
\begin{split}
&F_3(\lambda,\Upsilon) = - \lambda \; \<0'>\;, \qquad F_2(\lambda,\Upsilon) = 3 \lambda^2 \; \<0'> \cdot \<3'0> - 3 \lambda \; \<1'>\;,\\
&F_1(\lambda,\Upsilon) = - 3 \lambda^3 \; \<0'> \cdot (\<3'0>)^{2}  + 6 \lambda^2 \Big( \<3'0> \prec \<1'> + \<3'0> \succ \<1'> + \<3'1'> \Big) + 9 \lambda^2 \; \<2'2'>\;\\
&F_0(\lambda,\Upsilon) = \lambda^4 \; \<0'> \cdot (\<3'0>)^{3} - 3 \lambda^3 \Big[ (\<3'0>)^2 \prec \<1'> + (\<3'0>)^2 \succ \<1'> + (\<3'0> \circ \<3'0>) \circ \<1'>\\
&\phantom{111111111}+ 2 \; \<3'1'> \cdot \<3'0> + 2 \; \Com (\<3'0>; \<3'0>; \<1'>) \Big] + 3 \lambda^2 \; \<3'2'> - 9 \lambda^3 \; \<2'2'> \cdot \<3'0>\;. 
\end{split}
\end{equation}
By Bony's estimate and the definition of $\|\cdot\|_{\xX_T}$, we see that the coefficients $F_j$'s are all well defined, and satisfy the bounds
\begin{equation*}
\begin{split}
\|F_3(\Upsilon)\|_{\cC_{T}^{-\kappa}}\phantom{11} &\lesssim \|\Upsilon\|_{\xX_T}\;, \qquad \qquad \qquad \|F_2(\Upsilon)\|_{\cC_{T}^{-\frac{1}{2}-\kappa}} \lesssim \|\Upsilon\|_{\xX_T} (1 + \|\Upsilon\|_{\xX_T})\;,\\
\|F_1(\Upsilon)\|_{\cC_{T}^{-\frac{1}{2}-\kappa}} &\lesssim \|\Upsilon\|_{\xX_T} (1+\|\Upsilon\|_{\xX_T}^{2})\;, \quad \|F_0(\Upsilon)\|_{\cC_{T}^{-\frac{1}{2}-\kappa}} \lesssim \|\Upsilon\|_{\xX_T} (1 + \|\Upsilon\|_{\xX_T}^{3})\;, 
\end{split}
\end{equation*}
where the dependence of $\lambda$ are hidden in the proportionality constants. For $\eps=0$, we define $G_0$ as
\begin{equation} \label{eq:G_0}
\begin{split}
G_0(\lambda,\Ups,u) := &\sum_{j=0}^{3} F_{j}(\lambda,\Upsilon) u^j - 3 \lambda (u- \lambda \<3'0>) \succ \<2'> + 9 \lambda^2 \Big[ \Com \big(u-\lambda\<3'0>; \iI(\<2'>); \<2'> \big)\\
&+ \<2'> \circ [\iI, \prec](u-\lambda \<3'0>, \<2'>) - (\<2'> \circ h) \cdot (u - \lambda \<3'0>) \Big]\;,
\end{split}
\end{equation}
and consider the systems of equations for $(v,w)$ given by
\begin{equation} \label{eq:fixed_pt_0}
\begin{split}
&v(t) = e^{t (\Delta-1)} v(0) - 3 \lambda \int_{0}^{t} e^{(t-r)(\Delta-1)} \Big[ \big( v(r) + w(r) - \lambda \<3'0>(r) \big) \prec \<2'>(r) \Big] {\rm d}r\;,\\
&w(t) = e^{t(\Delta-1)} w(0) - 3 \lambda \int_{0}^{t} e^{(t-r)(\Delta-1)} \Big[ \big( e^{r(\Delta-1)}v(0) + w(r) \big) \circ \<2'>(r) \Big] {\rm d}r\\
&\phantom{11111}+ \int_{0}^{t} e^{(t-r)(\Delta-1)} G_0 \big(\lambda, \Ups(r), v(r) + w(r) \big) {\rm d}r\;.
\end{split}
\end{equation}
This is the natural candidate for the limiting equation. We now specify the space in which we are seeking the solutions. Since the linear evolution allows a singularity at $t=0$ (even when measured as a map between the same Besov spaces), we set up $\eps$-dependent spaces to encode this possible singularity. Recall that $V'$ has degree $2n-1$. 

\begin{defn} \label{de:space_solution_remainder}
We fix $\delta_0 \in (0, \frac{\kappa}{n})$. Define the norms $\|\cdot\|_{\yY_{T,\eps}^{(1)}}$ and $\|\cdot\|_{\yY_{T,\eps}^{(2)}}$ on the space of space-time functions up to time $T$ by
\begin{equation} \label{eq:space_solution_remainder_sep_eps}
\begin{split}
\|v\|_{\yY_{T,\eps}^{(1)}} := &\sup_{t \in [0,\eps^2]} \big( (\sqrt{t} / \eps )^{\delta_0} \|v(t)\|_{\kappa} \big) + \sup_{t \in [\eps^2, T]} \|v(t)\|_{\kappa}\\
&+ \sup_{t \in [0,T]} \big( t^{\frac{2}{3}} \|v(t)\|_{1-2\kappa} \big) + \sup_{0 \leq s < t \leq T} s^{\frac{1}{4}} \frac{\|v(t)-v(s)\|_{\kappa}}{|t-s|^{\frac{1}{8}}}\;,\\
\|w\|_{\yY_{T,\eps}^{(2)}} := &\sup_{t \in [0,\eps^2]} \big( (\sqrt{t} / \eps )^{\delta_0} \|w(t)\|_{\kappa} \big) + \sup_{t \in [\eps^2, T]} \|w(t)\|_{\kappa}\\
&+ \sup_{t \in [0,T]} \big( t^{\frac{2}{3}} \|w(t)\|_{1+2\kappa} \big) + \sup_{0 \leq s < t \leq T} s^{\frac{1}{4}} \frac{\|w(t)-w(s)\|_{\kappa}}{|t-s|^{\frac{1}{8}}}\;.
\end{split}
\end{equation}
For $\eps=0$, define the norms $\|\cdot\|_{\yY_T^{(1)}}$ and $\|\cdot\|_{\yY_T^{(2)}}$ by
\begin{equation} \label{eq:space_solution_remainder_sep}
\begin{split}
\|v\|_{\yY_T^{(1)}} &:= \sup_{t \in [0,T]} \Big( \|v(t)\|_{\kappa} + t^{\frac{2}{3}} \|v(t)\|_{1-2\kappa} \Big) + \sup_{0 \leq s < t \leq T} s^{\frac{1}{4}} \frac{\|v(t)-v(s)\|_{\kappa}}{|t-s|^{\frac{1}{8}}}\;,\\
\|w\|_{\yY_T^{(2)}} &:= \sup_{t \in [0,T]} \Big( \|w(t)\|_{\kappa} + t^{\frac{2}{3}} \|w(t)\|_{1+2\kappa} \Big) + \sup_{0 \leq s < t \leq T} s^{\frac{1}{4}} \frac{\|w(t)-w(s)\|_{\kappa}}{|t-s|^{\frac{1}{8}}}\;.
\end{split}
\end{equation}
We define the norm on the space $\yY_{T,\eps}$ and $\yY_T$ of pairs of space-time functions by
\begin{equation} \label{eq:space_solution_remainder}
\| (v, w) \|_{\yY_{T,\eps}} := \|v\|_{\yY_{T,\eps}^{(1)}} + \|w\|_{\yY_{T,\eps}^{(2)}}\;, \qquad \|(v,w)\|_{\yY_T} := \|v\|_{\yY_T^{(1)}} + \|w\|_{\yY_T^{(2)}}. 
\end{equation}
The only difference between $\yY_{T,\eps}^{(1)}$ and $\yY_{T,\eps}^{(2)}$ is that the spatial regularity (at fixed time) for the former is $1-2\kappa$ while it is $1+2\kappa$ for the latter, and the same is true for $\yY_T^{(1)}$ and $\yY_T^{(2)}$. 
\end{defn}

In the sequel, we will write $\yY_{T,\eps}$ for $\eps \in [0,1]$, with $\eps=0$ corresponding to the space $\yY_T$. The following is the main statement on the existence and convergence of the solutions $(v_\eps,w_\eps)$.

\begin{thm}
	\label{th:fixed_pt_convergence}
	Let $\{\psi_\eps\}_{\eps \in (0,1]}$ and $\{h_\eps\}_{\eps \in [0,1]}$ be families of space-time functions such that
	\begin{equation*}
	\sup_{\eps \in (0,1]} \sup_{(t,x) \in [0,1] \times \T^3} \eps^{\frac{1}{2}+\kappa} |\psi_\eps(t,x)| < +\infty\;, \quad \sup_{\eps \in (0,1)} \sup_{t \in [0,1]} \Big( t^{\frac{1}{4}} \|h_\eps(t)\|_{1+2\kappa} \Big) < +\infty\;.
	\end{equation*}
	Recall the definition of the spaces $\xX_T$ and $\yY_T$ in \eqref{eq:norm_enhanced_noise} and \eqref{eq:space_solution_remainder}. Consider the fixed point problem
	\begin{equation} \label{eq:fixed_pt_eps}
	\begin{split}
	&v_\eps(t) = e^{t (\lL_\eps-1)} v_\eps(0) - 3 \lambda_\eps \int_{0}^{t} e^{(t-r)(\lL_\eps-1)} \Big[ \big( v_\eps(r) + w_\eps(r) - \lambda_\eps \<3'0>_\eps(r) \big) \prec \<2'>_\eps(r) \Big] {\rm d}r\;,\\
	&w_\eps(t) = e^{t(\lL_\eps-1)} w_\eps(0) - 3 \lambda_\eps \int_{0}^{t} e^{(t-r)(\lL_\eps-1)} \Big[  \<2'>_\eps(r) \circ \big( e^{r(\lL_\eps-1)}v_\eps(0) + w_\eps(r) \big) \Big] {\rm d}r\\
	&\phantom{11111}+ \int_{0}^{t} e^{(t-r)(\lL_\eps-1)} G_\eps \big(\lambda_\eps, \Ups_\eps(r), v_\eps(r) + w_\eps(r) \big) {\rm d}r\;, 
	\end{split}
	\end{equation}
	where $G_\eps$ for $\eps > 0$ and $\eps=0$ are given in \eqref{eq:G_eps} and \eqref{eq:G_0} respectively. Then for every $\lambda_\eps \in \R$, $\Ups_\eps \in \xX$ and $\big( v_\eps(0), w_\eps(0) \big) \in \bB^\kappa \times \bB^\kappa$, there exists $T_\eps \leq 1$ such that the fixed point problem \eqref{eq:fixed_pt_eps} has a unique solution $(v_\eps, w_\eps) \in \yY_{T_\eps}$. Furthermore, if $\lambda_\eps$, $\Ups_\eps$ and $\big( v_\eps(0), w_\eps(0) \big)$ are uniformly bounded in their respective spaces, the local existence time $T_\eps$ can be taken uniform in $\eps \in [0,1]$. 
	
	Fix arbitrary $\lambda \in \R$, $\Ups \in \xX$ and $\big( v(0), w(0) \big) \in \xX$. Let $(v,w) \in \yY_T$ denote the unique solution to \eqref{eq:fixed_pt_eps} with $\eps=0$ and with the above inputs. Suppose $\lambda_\eps \rightarrow \lambda$, $\Ups_\eps \rightarrow \Ups$ in $\xX$, $\big( v_\eps(0), w_\eps(0) \big) \rightarrow \big( v(0), w(0) \big)$ in $\bB^\kappa \times \bB^\kappa$, and $\sup_{t \in [0,1]} \big( t^{\frac{1}{4}} \|h_\eps(t) - h(t)\|_{1+2\kappa} \big) \rightarrow 0$. Then there exists $\eps_0>0$ such that for every $\eps \in (0,\eps_0)$, the solution $(v_\eps, w_\eps) \in \yY_T$ to \eqref{eq:fixed_pt_eps} can be defined up to the same time $T$. Furthermore, we have $\|(v_\eps,w_\eps) - (v,w)\|_{\yY_{T,\eps}} \rightarrow 0$ as $\eps \rightarrow 0$. 
\end{thm}

\begin{rmk} \label{rmk:initial}
	The solution $\Phi_\eps$ to \eqref{eq:main_Phi} can be written as $\Phi_\eps = \<1>_\eps - \lambda \<3'0>_\eps + v_\eps + w_\eps$. Hence, the requirement on the initial condition in Theorem~\ref{th:main} is that $\Phi_\eps(0,\cdot) - \<1>_\eps(0,\cdot)$ converges in $\bB^\kappa$ to $\Phi(0,\cdot) - \<1>(0,\cdot)$ for some function $\Phi(0,\cdot)$. Then, we have that $\Phi_\eps$ converges to the solution of the dynamical $\Phi^4_3(\lambda)$-equation with initial data $\Phi(0,\cdot)$. 
	
	Note that this puts restriction on the local behaviour of $\Phi_\eps(0,\cdot)$. Ideally we would like condition on the convergence of $\Phi_\eps(0,\cdot)$ itself without giving reference to $\<1>_\eps$. Then in order to extend local solution to longer time intervals, one necessarily needs to be able to treat initial data below $\bB^{-\frac{1}{2}-}$. But this will create a problem of non-integrable singularity (even for fixed $\eps$) for higher powers in $V'$ if the smoothing effect of $\lL_\eps$ is not strong enough. One way to circumvent this without putting further assumption on $\lL_\eps$ or $V$ is to set up a weighted space, separating small and large scale behaviours. We choose to put restrictions on the initial condition to avoid technical complications. 
\end{rmk}

\begin{rmk}
	The existence of solutions to \eqref{eq:fixed_pt_eps} has been proven in \cite[Theorem~2.1]{Phi43global} and \cite[Theorem~3.1]{phi43_CC}, at least for $\eps=0$. The existence of solutions for $\eps>0$ can be proven with essentially the same arguments, except that the operator $\Delta$ is replaced by $\lL_\eps$ (which satisfies all the necessary bounds), that the choice of exponents are slightly different, and that there are two additional terms: the small remainder $\eps^{-\frac{3}{2}} V'(\sqrt{\eps} \cdot; \sqrt{\eps} \cdot)$, and the one involving $h_\eps$ (or $h$). The setup here follows that in \cite{Phi43global}. In the proof below, we only give details for the small remainder term as well as terms involving commutators. The convergence of the solutions as $\eps \rightarrow 0$ employs additional bounds involving the difference of the heat semi-groups $e^{t(\lL_\eps-1)} - e^{t(\Delta-1)}$ which are provided in the appendix. 
\end{rmk}

\subsection{Some preliminary bounds}

We give some preliminary bounds that are needed in the proof of Theorem~\ref{th:fixed_pt_convergence}. We write $\yY_{T,\eps}$ for $\eps \in [0,1]$, with $\eps=0$ corresponding to $\yY_T$. Also recall that $\|f\|_{\cC_r^{\alpha}} = \sup_{r' \in [0,r]} \|f(r')\|_{\alpha}$. 

\begin{lem} \label{le:fixed_pt_commutator_heat}
	We have the bounds
	\begin{equation*}
	\big\| [\iI_\eps, \prec](u, \<2'>)(r) \big\|_{1+2\kappa} \lesssim r^{-\frac{1}{4}} \|\<2'>\|_{\cC_r^{-1-\kappa}} \|u\|_{\yY_{r,\eps}^{(1)}}\;, 
	\end{equation*}
	and
	\begin{equation*}
	\big\| [\iI_\eps, \prec](\<3'0>, \<2'>)(r) \big\|_{1+2\kappa} \lesssim \Big( r^{\frac{1}{4}-2\kappa} \|\<3'0>\|_{\cC_r^{\frac{1}{2}-\kappa}} + r^{\frac{1}{8}-\frac{3\kappa}{2}} \|\<3'0>\|_{\cC^{\frac{1}{8}}([0,r], \cC^{\frac{1}{4}-\kappa}(\T^3))} \Big) \|\<2'>\|_{\cC_r^{-1-\kappa}}\;.
	\end{equation*}
	Both are uniform in $\eps \in [0,1]$ and $r \in (0,1)$. As a consequence, we have
	\begin{equation*}
	\big\| \<2'>(r) \circ [\iI_\eps, \prec](u - \lambda \<3'0>, \<2'>)(r) \big\|_{\kappa} \lesssim r^{-\frac{1}{4}} \|\Ups\|_{\xX_r}^{2} \big( \|u\|_{\yY_{r,\eps}^{(1)}} + \|\Ups\|_{\xX_r} \big)\;.
	\end{equation*}
\end{lem}
\begin{proof}
	It suffices to prove the first two bounds. The third follows from the first two and the estimate of the resonance product in Proposition~\ref{pr:Bony_estimates}. 
	
	For the first, we have
	\begin{equation} \label{eq:commutator_split}
	\begin{split}
	[\iI_\eps, \prec](u,\<2'>)(r) = &\int_{0}^{r} \big[ e^{(r-r')(\lL_\eps-1)}, \prec \big] \big( u(r'), \<2'>(r') \big) {\rm d}r'\\
	&+ \int_{0}^{r} \big( u(r') - u(r) \big) \prec \Big( e^{(r-r')(\lL_\eps-1)} \<2'>(r') \Big) {\rm d}r'\;.
	\end{split}
	\end{equation}
	By the commutator estimate for the heat kernel \eqref{eq:heat_commutator} and the definition of $\yY^{(1)}$ in \eqref{eq:space_solution_remainder_sep_eps}, we can control the first term on the right hand side by
	\begin{equation*}
	\begin{split}
	&\phantom{111}\int_{0}^{r} \big\| \big[ e^{(r-r')(\lL_\eps-1)}, \prec \big] \big( u(r'), \<2'>(r') \big)  \big\|_{1+2\kappa} {\rm d}r'\\
	&\lesssim \int_{0}^{r} (r-r')^{-\frac{1+5\kappa+\delta_0}{2}} \|u(r')\|_{1-2\kappa} \|\<2'>(r')\|_{-1-\kappa} {\rm d}r'\\
	&\lesssim \bigg( \int_{0}^{r} (r-r')^{-\frac{1+5\kappa+\delta_0}{2}} (r')^{-\frac{2}{3}} {\rm d}r' \bigg) \|\<2'>\|_{\cC_r^{-1-\kappa}} \|u\|_{\yY_{r}^{(1)}}\\
	&\lesssim r^{-\frac{1}{6}-\frac{5\kappa+\delta_0}{2}} \|\<2'>\|_{\cC_r^{-1-\kappa}} \|u\|_{\yY_{r,\eps}^{(1)}}\;.
	\end{split}
	\end{equation*}
	As for the second term on the right hand side of \eqref{eq:commutator_split}, using Proposition~\ref{pr:Bony_estimates} and Lemma~\ref{le:heat_regularisation}, we have
	\begin{equation*}
	\begin{split}
	&\phantom{111}\int_{0}^{r} \Big\| \big( u(r') - u(r) \big) \prec \Big( e^{(r-r')(\lL_\eps-1)} \<2'>(r') \Big) \Big\|_{1+2\kappa} {\rm d}r'\\
	&\lesssim \int_{0}^{r} \|u(r) - u(r')\|_{\kappa} (r-r')^{-1-\frac{3\kappa+\delta_0}{2}} \|\<2'>(r')\|_{-1-\kappa} {\rm d}r'\\
	&\lesssim \bigg(\int_{0}^{r} (r-r')^{-\frac{7}{8}-\frac{3\kappa+\delta_0}{2}} (r')^{-\frac{1}{4}} {\rm d}r' \bigg) \|\<2'>\|_{\cC_r^{-1-\kappa}} \|u\|_{\yY_{r}^{(1)}}\\
	&\lesssim r^{-\frac{1}{8}-\frac{3\kappa+\delta_0}{2}} \|\<2'>\|_{\cC_r^{-1-\kappa}} \|u\|_{\yY_{r,\eps}^{(1)}}\;. 
	\end{split}
	\end{equation*}
	Since $\delta_0 < \frac{\kappa}{n}$ and $\kappa$ is sufficiently small, we can enlarge both bounds to $r^{-\frac{1}{4}}$. This completes the proof of the first bound. 
	
	The proof for the second one is the same -- one splits the quantity into two sums as above, and uses the $\cC_r^{\frac{1}{2}-\kappa}$ and $\cC_{r}^{\frac{1}{8}, \frac{1}{4}-\kappa}$ norms of $\<3'0>$ respectively. Finally, one combines the two bounds with Proposition~\ref{pr:Bony_estimates} and the fact that $r \leq 1$ to conclude the lemma. 
\end{proof}

\begin{rmk}
	The commutator estimate for $[\iI_\eps, \prec]$ (and for $[\iI_\eps-\iI, \prec]$) is the only place that requires H\"older-in-time continuity of the process $\<3'0>$ and the solution $(v,w)$. 
\end{rmk}

We recall at this point that $V$ is an even polynomial of degree $2n$. 

\begin{lem} \label{le:remainder}
	We have
	\begin{equation} \label{eq:remainder_bd}
	\eps^{-\frac{3}{2}} \| V'\big( \sqrt{\eps} \psi_\eps; \sqrt{\eps} f \big)\|_{L^{\infty}(\T^3)} \lesssim \eps^{\frac{1}{4}} \|f\| \Big( 1 + \eps^{\frac{1}{2}+\kappa} \|\psi_\eps\| + \|f\| \Big)^{2n-2}\;, 
	\end{equation}
	and 
	\begin{equation} \label{eq:remainder_contraction}
	\begin{split}
	&\phantom{111}\eps^{-\frac{3}{2}} \big\| V'\big(\sqrt{\eps} \psi_\eps; \sqrt{\eps} f \big) - V'\big(\sqrt{\eps} \psi_\eps; \sqrt{\eps} \bar{f} \big) \big\|_{L^{\infty}(\T^3)}\\
	&\lesssim \eps^{\frac{1}{4}} \|f-\bar{f}\| \Big( 1 + \eps^{\frac{1}{2}+\kappa} \|\psi_\eps\| + \|f\| + \|\bar{f}\| \Big)^{2n-2}\;, 
	\end{split}
	\end{equation}
	where all the norms on the right hand sides of the two bounds above are $L^{\infty}(\T^3)$-norms, and all functions are evaluated at a fixed time. Both bounds are uniform in $\eps \in (0,1)$ and in $\psi_\eps,f,\bar{f}\in L^{\infty}(\T^3)$, and the proportionality constants are also independent of the actual time. 
\end{lem}
\begin{proof}
	By the mean value theorem, there exists a $g$ with $0 \leq |g| \leq |f|$ such that
	\begin{equation*}
	V'(\sqrt{\eps} \psi_\eps; \sqrt{\eps} f) = \frac{1}{24} V^{(5)}(\sqrt{\eps} \psi_\eps + \sqrt{\eps} g) \cdot (\sqrt{\eps} f)^{4}\;. 
	\end{equation*}
	Since $V^{(5)}$ has polynomial growth or order at most $2n-5$, we have
	\begin{equation*}
	|V'(\sqrt{\eps} \psi_\eps; \sqrt{\eps}f)| \lesssim \eps^2 |f|^{4} \big( 1 + \sqrt{\eps} |\psi_\eps| + \sqrt{\eps} |g| \big)^{2n-5}\;. 
	\end{equation*}
	Using $|g| \leq |f|$ and that we can choose $\kappa$ sufficiently small ($\kappa < \frac{1}{8n}$ would be sufficient), we get
	\begin{equation*}
	\eps^{-\frac{3}{2}} |V'(\sqrt{\eps} \psi_\eps; \sqrt{\eps} f)| \lesssim \eps^{\frac{1}{4}} |f| \big( 1 + \eps^{\frac{1}{2}+\kappa} |\psi_\eps| + |f| \big)^{2n-2}\;. 
	\end{equation*}
	Taking $L^\infty$-norm in $[0,T] \times \T^3$ yields the desired bound \eqref{eq:remainder_bd}. As for \eqref{eq:remainder_contraction}, we notice the identity
	\begin{equation*}
	\begin{split}
	&\eps^{-\frac{3}{2}} \Big( V'\big(\sqrt{\eps} \psi_\eps; \sqrt{\eps} f\big) - V'\big( \sqrt{\eps} \psi_\eps; \sqrt{\eps} \bar{f} \big) \Big) = \eps^{-\frac{3}{2}} V'\big( \sqrt{\eps} (\psi_\eps + \bar{f}); \sqrt{\eps} (f-\bar{f}) \big)\\
	&+ \eps^{-\frac{3}{2}} \sum_{\ell=1}^{3} \frac{\big(\sqrt{\eps} (f-\bar{f})\big)^{\ell}}{\ell!} \Big[ V^{(\ell+1)}\big( \sqrt{\eps} \psi_\eps + \sqrt{\eps} \bar{f} \big) - \sum_{j=0}^{3-\ell} \frac{V^{(\ell+1+j)}\big(\sqrt{\eps} \psi_\eps \big)}{j!}  \cdot (\sqrt{\eps} \bar{f})^{j}  \Big]
	\end{split}\;, 
	\end{equation*}
	and the desired bound follows from the same argument as above. 
\end{proof}


\subsection{Proof of Theorem~\ref{th:fixed_pt_convergence}}

The statement (and hence the proof) consists of two parts: the existence of solutions $(v_\eps,w_\eps)$ for each $\eps \in [0,1]$, and the convergence of $(v_\eps,w_\eps)$ to $(v,w)$ as $\eps \rightarrow 0$. 

\begin{flushleft}
	\textit{Part 1. }
\end{flushleft}

Fix $\lambda \in \R$, $\Ups \in \xX$, $\psi_\eps \in L^{\infty}\big( [0,T] \times \T^3 \big)$, $h_\eps \in \cC((0,1]; \bB^{1+2\kappa})$ and $\big(v_\eps(0), w_\eps(0) \big) \in \bB^\kappa \times \bB^\kappa$ with
\begin{equation*}
\begin{split}
&\eps^{\frac{1}{2}+\kappa} \|\psi_\eps\|_{L^{\infty}([0,1] \times \T^3)} + \sup_{t \in [0,1]} \big( t^{\frac{1}{4}} \|h_\eps(t)\|_{1+2\kappa} \big) + \|\Ups\|_{\xX} \leq \kK\;,\\
&\phantom{1} \qquad \text{and} \qquad \|v_\eps(0)\|_{\kappa} + \|w_\eps(0)\|_{\kappa} \leq \mM\;. 
\end{split}
\end{equation*}
For $T, \eps \in [0,1]$, define the mild solution map $\Gamma_{T,\eps} = (\Gamma_{T,\eps}^{(1)}, \Gamma_{T,\eps}^{(2)})$ by
\begin{equation} \label{eq:fixed_pt_map}
\begin{split}
\Gamma_{T,\eps}^{(1)}(v,w)(t) = &e^{t(\lL_\eps-1)} v_\eps(0) - 3 \lambda \int_{0}^{t} e^{(t-r)(\lL_\eps-1)} \Big( \big( v(r) + w(r) - \lambda \<3'0>(r) \big) \prec \<2'>(r)  \Big) {\rm d}r\;,\\
\Gamma_{T,\eps}^{(2)}(v,w)(t) = &e^{t(\lL_\eps-1)} w_\eps(0) - 3 \lambda \int_{0}^{t} e^{(t-r)(\lL_\eps-1)} \Big( \<2'>(r) \circ \big( e^{r(\lL_\eps-1)} v_\eps(0) + w(r) \big) \Big) {\rm d}r\\
&+ \int_{0}^{t} e^{(t-r)(\lL_\eps-1)} G_\eps \big( \lambda, \Ups(r), v(r) + w(r) \big) {\rm d}r\;, 
\end{split}
\end{equation}
where we recall the symbols in the generic element $\Ups \in \xX$ from \eqref{eq:enhanced_noise}, and the expression of $G_\eps$ in \eqref{eq:G_eps} and \eqref{eq:G_0} for $\eps > 0$ and $\eps=0$ respectively. We need to show that for suitable $T$ and $\rR$ (depending only on $\lambda$, $\kK$ and $\mM$ but independent of $\eps$), $\Gamma_{T,\eps}$ is a contraction map from the ball in $\yY_T$ with radius $\rR$ into itself. 

\begin{flushleft}
	\textit{Step 1.}
\end{flushleft}

We first check $\Gamma_{T,\eps}$ maps the ball in $\yY_{T,\eps}$ with radius $\rR$ (centered at the origin) into itself. For notational simplicity, we write $u=v+w$. 

We give details for three terms appearing in the definition of $\Gamma_{T,\eps}^{(2)}$: the initial data term $e^{t(\lL_\eps-1)} w_\eps(0)$, the term involving the commutator $[\iI_\eps, \prec]$ and the remainder term $\eps^{-\frac{3}{2}} V'(\sqrt{\eps}\cdot; \sqrt{\eps}\cdot)$. The bounds for the other terms (including those appearing in the definition of $\Gamma_{T,\eps}^{(1)}$) can be obtained in the same way. 

For the initial data term, we have by Lemma~\ref{le:heat_regularisation}
\begin{equation*}
\big\|e^{t(\lL_\eps-1)} w_\eps(0) \big\|_{\gamma} \lesssim t^{-\frac{\gamma-\kappa}{2}} (1 + \eps^{\delta_0} t^{-\frac{\delta_0}{2}}) \|w_\eps(0)\|_{\kappa} \lesssim t^{-\frac{\gamma-\kappa}{2}} (1 + \eps^{\delta_0} t^{-\frac{\delta_0}{2}}) \mM\;.
\end{equation*}
Taking $\gamma = \kappa$ and $1+2\kappa$ respectively gives the corresponding bounds in spatial regularity. As for the term with temporal H\"older regularity, using the continuity estimate for the perturbed heat semigroups in Lemma~\ref{le:heat_continuity} (with $\theta=\frac{1}{4}$ and $\gamma=\alpha=\kappa$), we have
\begin{equation*}
\big\| \big(e^{t(\lL_\eps-1)} - e^{s(\lL_\eps-1)}\big)w_\eps(0) \big\|_{\kappa} \lesssim (t-s)^{\frac{1}{8}} s^{-\frac{1}{8}-\frac{\delta_0}{2}} \|w_\eps(0)\|_{\kappa}\;.
\end{equation*}
This shows that for the term with the initial data, we have
\begin{equation} \label{eq:map_GammaT2_initial_bd}
\big\| e^{t(\lL_\eps-1)} w_\eps(0) \big\|_{\yY_{T,\eps}^{(2)}} \lesssim \mM\;.
\end{equation}
We now turn to $G_\eps$, focusing on the commutator term $[\iI_\eps, \prec]$ and the remainder $\eps^{-\frac{3}{2}} V'(\sqrt{\eps}\cdot; \sqrt{\eps} \cdot)$. Recall that we write $u=v+w$. For the commutator term, we also write $f_\eps = \<2'> \circ [\iI_\eps, \prec](u - \lambda \<3'0>,\<2'>)$ for simplicity. By Lemma~\ref{le:fixed_pt_commutator_heat}, we have
\begin{equation*}
\|f_\eps(r)\|_{\kappa} \lesssim r^{-\frac{1}{4}} \kK^2 \big( \|v+w\|_{\yY_{T,\eps}^{(1)}} + \kK \big) \lesssim r^{-\frac{1}{4}} \kK^2 (\rR + \kK)\;, 
\end{equation*}
where the last inequality comes from  $\|w\|_{\yY_{T,\eps}^{(1)}} \leq \|w\|_{\yY_{T,\eps}^{(2)}}$ by definition of the norms. Hence, we have
\begin{equation*}
\Big\| \int_{0}^{t} e^{(t-r)(\lL_\eps-1)} f_\eps(r) {\rm d}r \Big\|_{\gamma} \lesssim \Big( \int_{0}^{t} (t-r)^{-\frac{\gamma-\kappa+\delta_0}{2}} r^{-\frac{1}{4}} {\rm d}r \Big) \kK^2 (\kK+\rR) \lesssim t^{\frac{3}{4}-\frac{\gamma-\kappa+\delta_0}{2}} \kK^2 (\kK+\rR)\;. 
\end{equation*}
Taking $\gamma = \kappa$ and $1+2\kappa$ respectively gives the bounds for spatial regularity with a factor bounded by $T^{\frac{1}{8}}$ provided $\kappa$ and $\delta_0$ are small enough. As for the temporal regularity, we write the difference between times $s$ and $t$ as
\begin{equation*}
\int_{0}^{s} \big( e^{(t-r)(\lL_\eps-1)} - e^{(s-r)(\lL_\eps-1)} \big) f_\eps(r) {\rm d}r + \int_{s}^{t} e^{(t-r)(\lL_\eps-1)} f_\eps(r) {\rm d}r\;, 
\end{equation*}
and we want to control the $\bB^{\kappa}$ norm of these two quantities. For the first term, using Lemma~\ref{le:heat_continuity} with $\theta = \frac{1}{4}$ and $\gamma=\alpha=\kappa$ and the bound on $\|f_\eps(r)\|_{\kappa}$, we have
\begin{equation*}
\begin{split}
\Big\| \int_{0}^{s} \big( e^{(t-r)(\lL_\eps-1)} - e^{(s-r)(\lL_\eps-1)} \big) f_\eps(r) {\rm d}r \Big\|_{\kappa} &\lesssim (t-s)^{\frac{1}{8}} \Big( \int_{0}^{s} (s-r)^{-\frac{1}{8}-\frac{\delta_0}{2}} r^{-\frac{1}{4}} {\rm d}r \Big) \kK^2 (\kK+\rR)\\
&\lesssim s^{-\frac{1}{4}} (t-s)^{\frac{1}{8}} s^{\frac{7}{8}-\frac{\delta_0}{2}} \kK^2 (\kK + \rR)\;.
\end{split}
\end{equation*}
For the second one, by Lemma~\ref{le:heat_regularisation}, we can control its $\bB^\kappa$-norm by
\begin{equation*}
\begin{split}
\int_{s}^{t} (t-r)^{-\frac{\delta_0}{2}} \|f_\eps(r)\|_{\kappa} {\rm d}r &\lesssim \Big( \int_{s}^{t} (t-r)^{-\frac{\delta_0}{2}} r^{-\frac{1}{4}} {\rm d}r \Big) \kK^2 (\kK+\rR)\\ &\lesssim s^{-\frac{1}{4}} (t-s)^{\frac{1}{8}} t^{\frac{7}{8}-\frac{\delta_0}{2}} \kK^2 (\kK+\rR)\;.
\end{split}
\end{equation*}
Hence, we obtain
\begin{equation} \label{eq:map_GammaT2_commutator}
\Big\| \int_{0}^{t} e^{(t-r)(\lL_\eps-1)} \Big( \<2'>(r) \circ [\iI_\eps,\prec](u,\<2'>)(r) \Big) {\rm d}r \Big\|_{\yY_{T,\eps}^{(2)}} \lesssim T^{\frac{1}{8}} \kK^2 (\kK+\rR)\;. 
\end{equation}
We now turn to the remainder term $V'$. For simplicity, we write
\begin{equation*}
F_\eps := \eps^{-\frac{3}{2}} V'\big( \sqrt{\eps} \psi_\eps; \sqrt{\eps} (u - \lambda \<3'0>) \big)\;. 
\end{equation*}
Lemma~\ref{le:remainder} implies
\begin{equation*}
\|F_\eps(r)\|_{L^{\infty}(\T^3)} \lesssim \eps^{\frac{1}{4}} r^{-\frac{(2n-1)\delta_0}{2}} (1 + \kK + \rR)^{2n-1}. 
\end{equation*}
We then have
\begin{equation*}
\Big\| \int_{0}^{t} e^{(t-r)(\lL_\eps-1)} F_\eps(r) {\rm d}r \Big\|_{\gamma} \lesssim \int_{0}^{t} (t-r)^{-\frac{\gamma+\delta_0}{2}} \|F_\eps(r)\|_{L^{\infty}(\T^3)} {\rm d}r \lesssim \eps^{\frac{1}{4}} t^{1-\frac{\gamma+2n\delta_0}{2}} (1+\kK+\rR)^{2n-1}\;,
\end{equation*}
where the first inequality follows from Lemma~\ref{le:heat_regularisation} and that $\|\cdot\|_{0} \lesssim \|\cdot\|_{L^{\infty}}$, which is the content of Lemma~\ref{le:embedding} and the second one is valid for $\gamma \in (0,2)$. Again, taking $\gamma = \kappa$ and $1+2\kappa$ gives the respective bounds for the two different spatial regularities. 

As for the time difference, similar as before, we write
\begin{equation*}
\begin{split}
&\phantom{111}\int_{0}^{t} e^{(t-r)(\lL_\eps-1)} F_\eps(r) {\rm d}r - \int_{0}^{s} e^{(s-r)(\lL_\eps-1)} F_\eps(r) {\rm d}r\\
&= \int_{0}^{s} \big( e^{(t-r)(\lL_\eps-1)} - e^{(s-r)(\lL_\eps-1)} \big) F_\eps(r) {\rm d}r + \int_{s}^{t} e^{(t-r)(\lL_\eps-1)} F_\eps(r) {\rm d}r\;. 
\end{split}
\end{equation*}
For the first term, using Lemma~\ref{le:heat_continuity} with $\theta = \frac{1}{4}$, $\gamma=\kappa$, $\alpha=0$ and that $\|\cdot\|_{0} \lesssim \|\cdot\|_{L^\infty}$, we can control it by
\begin{equation*}
\begin{split}
\Big\| \int_{0}^{s} \big( e^{(t-r)(\lL_\eps-1)} - e^{(s-r)(\lL_\eps-1)} \big) F_\eps(r) {\rm d}r \Big\|_{\kappa} &\lesssim (t-s)^{\frac{1}{8}} \int_{0}^{s} (s-r)^{-\frac{1}{8}-\frac{\kappa+\delta_0}{2}} \|F_\eps(r)\|_{L^\infty(\T^3)} {\rm d}r\\
&\lesssim \eps^{\frac{1}{4}}  s^{-\frac{1}{4}} (t-s)^{\frac{1}{8}} s^{\frac{9}{8}-\frac{\kappa}{2}-n\delta_0} (1+\kK+\rR)^{2n-1}\;. 
\end{split}
\end{equation*}
For the second one, we have
\begin{equation*}
\begin{split}
\int_{s}^{t} \big\| e^{(t-r)(\lL_\eps-1)} F_\eps(r) \big\|_{\kappa} {\rm d}r &\lesssim \int_{s}^{t} (t-r)^{-\frac{\kappa+\delta_0}{2}} \|F_\eps(r)\|_{L^\infty(\T^3)} {\rm d}r\\
&\lesssim \eps^{\frac{1}{4}} s^{-\frac{1}{4}} (t-s)^{1-\frac{\kappa+\delta_0}{2}} t^{\frac{1}{4}-\frac{(2n-1)\delta_0}{2}} (1+\kK+\rR)^{2n-1}\;. 
\end{split}
\end{equation*}
Hence, for $\kappa$ (and $\delta_0$) small enough we obtain
\begin{equation*}
\eps^{-\frac{3}{2}} \Big\| \int_{0}^{t} e^{(t-r)(\lL_\eps-1)} V'\big( \sqrt{\eps} \psi_\eps(r); \sqrt{\eps} (u(r) - \lambda \<3'0>(r)) \big) {\rm d}r \Big\|_{\yY_{T,\eps}^{(2)}} \lesssim \eps^{\frac{1}{4}} T^{\frac{1}{6}} (1+\kK+\rR)^{2n-1}\;. 
\end{equation*}
The bounds for $w \circ \<2'>$ and the nonlinear terms $F_j u^j$ are treated in detail in \cite{Phi43global}, and we omit the details here. Overall, if $\|(v,w)\|_{\yY_{T,\eps}} \leq \rR$, we deduce that
\begin{equation*}
\|\Gamma_{T,\eps}(v,w)\|_{\yY_{T,\eps}} \leq C \Big( \mM + T^{\theta} (1 + \mM + \kK + \rR)^{2n+1} \Big)
\end{equation*}
for some $\theta>0$ universal, and the constant $C$ depends on $\lambda$ only. Hence, if we take $\rR > 2 C \mM$ and $T>0$ sufficiently small such that
\begin{equation*}
C \Big( \mM + T^{\theta} (1 + \mM + \kK + \rR)^{2n+1} \Big) \leq \frac{\rR}{2}\;, 
\end{equation*}
we see $\Gamma_{T,\eps}$ maps the ball in $\yY_{T,\eps}$ with radius $\rR$ into the smaller ball in $\yY_T$ with radius $\frac{\rR}{2}$. Also, the local existence time $T$ is independent of $\eps$ since both $\theta$ and $C$ are. 

\begin{flushleft}
	\textit{Step 2. }
\end{flushleft}

We now show that $\Gamma_{T,\eps}$ is a contraction map between the above mentioned space for sufficiently small $T$ (independent of $\eps$). We need to get a bound for $\big\| \Gamma_{T,\eps}(v,w) - \Gamma_{T,\eps}(\bar{v},\bar{w}) \big\|_{\yY_{T,\eps}}$. Most of the terms in $\Gamma_{T,\eps}$ are ``constant'' or linear, in which case the quantities either completely cancel out, or the same bound as in Step 1 holds by replacing $\rR$ with $\|(v,w) - (\bar{v}-\bar{w})\|_{\yY_{T,\eps}}$. The only nonlinear terms are $F_j u^j$ for $j=2,3$ and the remainder $\eps^{-\frac{3}{2}} V'(\sqrt{\eps} \cdot; \sqrt{\eps} \cdot)$. In these two cases, one replaces one factor of $\rR$ by $\|(v,w) - (\bar{v}-\bar{w})\|_{\yY_{T,\eps}}$ (see Lemma~\ref{le:remainder} for example). Hence, we obtain the bound
\begin{equation*}
\big\| \Gamma_{T,\eps}(v,w) - \Gamma_{T,\eps}(\bar{v},\bar{w}) \big\|_{\yY_{T,\eps}} \lesssim T^{\theta} \|(v,w) - (\bar{v},\bar{w})\|_{\yY_{T,\eps}} \big( 1 + \kK + \rR \big)^{2n}\;
\end{equation*}
for some $\theta>0$ independent of $\eps$. One then deduces that for every $\eps \in [0,1]$, one can choose $T_\eps>0$ sufficiently small so that $\Gamma_{T_\eps,\eps}$ gives a contraction in a bounded ball in $\yY_{T_\eps}$. In addition, the local existence time $T_\eps$ is uniform in $\eps$ (bounded away from $0$) since the proportionality constant in the above bound is. 

\begin{flushleft}
	\textit{Part 2. }
\end{flushleft}

We now turn to the second part of the theorem, namely $\|(v_\eps, w_\eps) - (v,w)\|_{\yY_{T,\eps}}$ converging to $0$, where $T>0$ is the time up to which $(v,w)$ is defined. By the previous part, we know that there exists $0<S<T$ such that for all sufficiently small $\eps$, the solution $(v_\eps, w_\eps)$ to \eqref{eq:fixed_pt_eps} is defined in $\yY_{S,\eps}$. 

The difference $(v_\eps,w_\eps) - (v,w)$ is a linear combination of the terms of the form
\begin{equation*}
\int_{0}^{t} e^{(t-r)(\lL_\eps-1)} f_\eps(r) {\rm d}r - \int_{0}^{t} e^{(t-r)(\Delta-1)} f(r) {\rm d}r\;, 
\end{equation*}
where $f_\eps$ and $f$ come from the right hand sides of the equations, and can also depend on $(v_\eps,w_\eps)$ and $(v,w)$. The difference can be split into
\begin{equation} \label{eq:pde_convergence_split}
\int_{0}^{t} e^{(t-r)(\lL_\eps-1)} \big( f_\eps(r) - f(r) \big) {\rm d}r + \int_{0}^{t} \Big( e^{(t-r)(\lL_\eps-1)} - e^{(t-r)(\Delta-1)} \Big) f(r) {\rm d}r\;. 
\end{equation}
By invoking the above bounds in the fixed point map as well as the bound for the difference of the kernels $e^{(t-r)(\lL_\eps-1)} - e^{(t-r)(\Delta-1)}$ (see the bound \eqref{eq:heat_regularisation_difference} in Lemma~\ref{le:heat_regularisation}), we can obtain the bound
\begin{equation} \label{eq:pde_convergence_bound}
\begin{split}
\big\| (&v_\eps, w_\eps) - (v,w) \big\|_{\yY_{S,\eps}} \lesssim (S+\eps)^{\theta} \big\| (v_\eps,w_\eps) - (v,w) \big\|_{\yY_{S,\eps}}\\
&+ \|\Ups_\eps - \Ups\|_{\xX} +  |\lambda_\eps - \lambda| + \|v_\eps(0)-v(0)\|_{\kappa} + \|w_\eps(0)-w(0)\|_{\kappa} + (\eps T)^{\theta}\\
&+\big\| \big( e^{t(\lL_\eps-1)} - e^{t(\Delta-1)}  \big) v(0) \big\|_{\cC_{S,\eps}^\kappa} + \big\| \big( e^{t(\lL_\eps-1)} - e^{t(\Delta-1)} \big) w(0) \big\|_{\cC_{S,\eps}^\kappa}\;, 
\end{split}
\end{equation}
for some $\theta>0$, where we have written $\cC_{S,\eps}^{\kappa}$ as a shorthand for continuous evolution in $\bB^{\kappa}$ such that the part $t \in [0,\eps^2]$ is weighted by $(\sqrt{t}/\eps)^{\delta_0}$ and the rest is taken in the supremum norm in $t \in [\eps^2, S]$. The proportionality constant above is independent of $\eps$ (but depends on the size $\mM$ of the limiting solution $(v,w)$ and the size $\kK$ of the external inputs). Here, the first five terms on the right hand side of \eqref{eq:pde_convergence_bound} come from the estimates of the first term in \eqref{eq:pde_convergence_split} (with $f_\eps$ being $G_\eps$ and polynomials of solutions), and the last three terms come from the estimates for the second term in \eqref{eq:pde_convergence_split}. We do not have a positive power of $\eps$ in the last two since we are measuring those quantities in the same space as the initial data, but they still vanish as $\eps \rightarrow 0$ (Lemma\ref{le:heat_continuity_same_space}). 

If $S>0$ is sufficiently small (still uniform in $\eps \leq \eps_0$ for some fixed small $\eps_0$), we can absorb the first term on the right hand side of \eqref{eq:pde_convergence_split} into its left hand side to obtain
\begin{equation} \label{eq:pde_iterate}
\begin{split}
\big\| (v_\eps,w_\eps) - &(v,w) \big\|_{\yY_{S,\eps}} \lesssim \|\Ups_\eps - \Ups\|_{\xX} + |\lambda_\eps - \lambda| + \|v_\eps(0) - v(0)\|_{\kappa} + \|w_\eps(0) - w(0)\|_{\kappa}\\
&+ (\eps T)^{\theta} + \big\| \big( e^{t(\lL_\eps-1)} - e^{t(\Delta-1)}  \big) v(0) \big\|_{\cC_{S,\eps}^\kappa} + \big\| \big( e^{t(\lL_\eps-1)} - e^{t(\Delta-1)} \big) w(0) \big\|_{\cC_{S,\eps}^\kappa}\;.
\end{split}
\end{equation}
All the terms on the right hand side above vanish as $\eps \rightarrow 0$ (the first four are by assumption, while the last two follow from Lemma~\ref{le:heat_continuity_same_space}). In particular, we will have $S>\eps^2$ and that
\begin{equation*}
\|v_\eps(S)\|_{\kappa} + \|w_\eps(S)\|_{\kappa} \leq \sup_{t \in [0,T]} \Big( \|v(t)\|_{\kappa} + \|w(t)\|_{\kappa} \Big) + 1 =: \mM\;
\end{equation*}
if $\eps$ is sufficiently small. This enables us to iterate the bound \eqref{eq:pde_iterate} up to time $T$ and thus complete the proof of the theorem.

\section{Convergence of the stochastic objects}
\label{sec:main_convergence}

In this section, we show that the assumptions in Theorem~\ref{th:fixed_pt_convergence} on the convergences of the external inputs ($\Ups_\eps$, $\psi_\eps$ and $h_\eps$) are indeed true if they are for the stochastic objects as defined in Section~\ref{sec:stochastic_objects}. Note that in this section we are only using the first three conditions in Assumption~\ref{as:Q}, and no control on the derivative of $\qQ$ is assumed. The main statement is Theorem~\ref{th:main_conv_stochastic_obj} below.

\subsection{The main convergence theorem and a convergence criterion}
\label{sec:main_convergence_statement}

For $\eps \in (0,1)$, we define
\begin{equation*} 
\Ups_\eps := \big(\; \<0'>_\eps, \<1'>_\eps, \<2'>_\eps, \<3'0>_\eps, \<3'1'>_\eps, \<2'2'>_\eps, \<3'2'>_\eps \;\big)\;, 
\end{equation*}
where components of $\Ups_\eps$ are the stochastic objects defined in \eqref{eq:noises_V}, \eqref{eq:noises_integration} and \eqref{eq:noises_2nd}. Also let
\begin{equation} \label{eq:Ups_standard}
\Ups := \big(\; 1, \; \<1>, \; \<2>, \; \<30>, \; \<31>, \; \<22>, \; \<32> \;\big)\;, 
\end{equation}
where the components (except $1$) are the standard $\Phi^4_3$ stochastic objects described in Appendix~\ref{app:stochastic}. $\Ups_\eps$ and $\Ups$ are defined on the same probability space (that is, constructed from the same space-time white noise $\xi$). We emphasize that in this section these symbols do mean concrete stochastic processes rather than abstract placeholders representing generic distributions as in Section~\ref{sec:fixed_point}.


We also fix a sufficiently small $\kappa>0$, and recall the norm $\xX$ from \eqref{eq:norm_enhanced_noise}. The main convergence theorem for these stochastic objects is the following. 

\begin{thm} \label{th:main_conv_stochastic_obj}
	There exists $\delta>0$ (depending on $\kappa$) such that for every $p \geq 1$, we have the bound
	\begin{equation*}
	\E \|\Ups_\eps\|_{\xX}^{p} \lesssim_p 1\;, \qquad \lim_{\eps\to 0}\E \|\Ups_\eps - \Ups\|_{\xX}^{p} =0\;.
	\end{equation*}
	We also have
	\begin{equation*}
	\eps^{\kappa p} \E \|\<1>_\eps\|_{L^{\infty}([0,T] \times \T^3)}^{p} \lesssim_p\; \eps^{\frac{\delta p}{2}}\;, \, \E \sup_{t \in [0,1]} \Big( t^{\frac{1}{4}} \|e^{t(\lL_\eps-1)} \<2'0>_\eps(0) - e^{t(\Delta-1)} \<20>(0) \|_{1+2\kappa} \Big)^{p} \lesssim_p \eps^{\delta p}\;.
	\end{equation*}
	The proportionality constants are independent of $\eps$. 
\end{thm}

We provide a criterion for the convergence of the (stationary) stochastic objects. The following proposition is the same as \cite[Proposition~3.6]{Phi43_pedestrians}. 

\begin{prop} \label{pr:criterion_convergence}
	Let $N \in \N$ and let $\{\tau_\eps\}_{\eps \in (0,1)}, \tau: \R^{+} \rightarrow \sS'(\T^d)$ be a family of random processes which are in the first $N$ Wiener chaos and which are also stationary in space. Let $\widehat{\tau_\eps}(t,\cdot)$ and $\widehat{\tau}(t,\cdot)$ denote their Fourier coefficients. If
	\begin{equation} \label{eq:criterion_space_reg}
	\E |\widehat{\tau_\eps}(t,k)|^{2} \lesssim \scal{k}^{-d-2\alpha}\;, \quad \lim_{\eps\to 0} \sup_{k\in\Z^3}\Big(\scal{k}^{d+2\alpha}\E |\widehat{\tau_\eps}(t,k) - \widehat{\tau}(t,k)|^{2}\Big)=0\;, 
	\end{equation}
	where both bounds are uniform in $\eps$ and $k$, then for every $\beta < \alpha$ and every $p \geq 1$, we have
	\begin{equation} \label{eq:convergence_space_reg}
	\E \sup_{t \in [0,1]} \|\tau(t,\cdot)\|_{\beta}^{p} < +\infty\;, \qquad \lim_{\eps\to 0}\E \sup_{t \in [0,1]} \|\tau_\eps(t,\cdot) - \tau(t,\cdot)\|_{\beta}^{p} =0\;.
	\end{equation}
	If in addition to \eqref{eq:criterion_space_reg}, we also have the bounds
	\begin{equation} \label{eq:criterion_process}
	\begin{split}
	&\E |\widehat{\tau_\eps}(t,k) - \widehat{\tau_\eps}(s,k)|^{2} \lesssim |t-s|^{\theta} \scal{k}^{-d-2\alpha+2\theta}\;,\\
	&\E \big| \big( \widehat{\tau_\eps}(t,k) - \widehat{\tau_\eps}(s,k) \big) - \big( \widehat{\tau}(t,k) - \widehat{\tau}(s,k) \big)  \big|^{2} \lesssim \eps^{\delta} |t-s|^{\theta} \scal{k}^{-d-2\alpha + 2 \theta}\;, 
	\end{split}
	\end{equation}
	for some $\theta \in (0,1)$, uniformly in $0 \leq s, t \leq 1$ and $k \in \Z^d$, then for every $\beta < \alpha-\theta$ and every $p \geq 1$, we have
	\begin{equation} \label{eq:convergence_process}
	\E \sup_{0 \leq s, t \leq 1} \frac{\|\tau(t)-\tau(s)\|_{\beta}^{p}}{|t-s|^{\frac{\theta p}{2}}} <+\infty\;, \quad \E \sup_{0 \leq s,t \leq 1} \frac{\big\| \big(\tau_\eps(t) - \tau_\eps(s)\big) - \big( \tau(t) - \tau(s) \big) \big\|_{\beta}^{p}}{|t-s|^{\frac{\theta p}{2}}} \lesssim \eps^{\frac{\delta p}{2}}\;.
	\end{equation}
	The proportionality constants depend on $\beta$ and $p$. 
\end{prop}
\begin{proof}
This is \cite[Proposition~3.6]{Phi43_pedestrians}. 
\end{proof}
\begin{remark}
We will apply the above result to the collection of symbols in $\Upsilon_\eps$ and $\Upsilon$. It turns out that for all symbols but for $\<3'2'>_\eps$ we are even showing more, namely there is a rate of convergence (a positive power of $\eps$). If we further use the fourth item in Assumption~\ref{as:Q} then one would also obtain a rate of convergence for $\<3'2'>_\eps$.
\end{remark}

\subsection{Notations and observations}
\label{sec:Notations}

We let $\widetilde{\iI_\eps}$ be the operator defined via
\begin{equation} \label{eq:convolution_stationary}
(\widetilde{\iI_\eps} f)(t) := \int_{-\infty}^{t} e^{(t-r)(\lL_\eps-1)} f(r) {\rm d}r\;.
\end{equation}
Note that $\widetilde{\iI_\eps}$ is different from $\iI_\eps$ since the integration in time starts from $-\infty$. 

For every $k \in \Z^3$ and $N \in \N$, we use the notation
\begin{equation*}
\pP(N,k) = \big\{ (\ell_1, \dots, \ell_N) \in (\Z^{3})^{N}: \ell_1 + \cdots + \ell_{N} = k \big\}\;. 
\end{equation*}
Recall the rescaled smooth cutoff functions $\xX_j$ from Appendix~\ref{sec:Besov}. We write with a slight abuse of notation
\begin{equation} \label{eq:resonance_set}
\sum_{\stackrel{\ell + \widetilde{\ell} = k}{\ell \sim \widetilde{\ell}}} \widehat{f}(\ell) \widehat{g}(\widetilde{\ell}) := \sum_{\ell + \widetilde{\ell} = k} \Big( \widehat{f}(\ell) \widehat{g}(\widetilde{\ell}) \sum_{|i-j| \leq 1}  \xX_i(\ell) \xX_j(\widetilde{\ell}) \Big) = \widehat{f \circ g}(k)\;,
\end{equation}
as well as
\begin{equation} \label{eq:resonance_set_complement}
\sum_{\stackrel{\ell + \widetilde{\ell} = k}{\ell \nsim \widetilde{\ell}}} \widehat{f}(\ell) \widehat{g}(\widetilde{\ell}) := \sum_{\ell + \widetilde{\ell} = k} \bigg( \widehat{f}(\ell) \widehat{g}(\widetilde{\ell}) \Big( 1 - \sum_{|i-j| \leq 1}  \xX_i(\ell) \xX_j(\widetilde{\ell}) \Big) \bigg)\;.
\end{equation}
Hence, the "set" $\{\ell \sim \widetilde{\ell}\}$ can be viewed as complement of $\aA$ where
\begin{equation*}
\aA=\Big\{(\ell, \widetilde{\ell}) \in \Z^3\times \Z^3:\, \Big(|\ell|> \frac83 \text{ or } |\widetilde{\ell}|> \frac83 \Big) \text{ and } \frac{|\ell|}{|\widetilde{\ell}|} \notin \Big[\frac{9}{64},\frac{64}{9}\Big]\Big\}\;. 
\end{equation*}
Note that the left hand sides of \eqref{eq:resonance_set} and \eqref{eq:resonance_set_complement} are not exactly the sum over the sets $\{\ell \sim \widetilde{\ell}\}$ or $\{\ell \nsim \widetilde{\ell}\}$, but rather weighted by the cutoff function $\chi$. But in what follows, we can regard them as the real sum over these sets without affecting the statements.

\begin{remark} \label{rem:resonance_restriction}
The above notation will come in handy in Section~\ref{sec:32} when we need to bound the resonance product between two stochastic objects. We note at this point the following consequences. For $k \in \Z^3$:
\begin{itemize}
	\item[(1)] The set $\{\ell \in \Z^3: \ell+k \sim -\ell\}$ is contained in $\big\{\ell \in \Z^3: \scal{\ell} \geq \frac{\scal{k}}{10}\big\}$; 
	
	\item[(2)] The set $\{\ell \in \Z^3: \ell+l \nsim - \ell\}$ is contained in $\big\{\ell \in \Z^3: \scal{\ell} \leq 10 \scal{k} \big\}$. 
\end{itemize}
Hence, whenever one deals with sums over $\{\ell+k \sim - \ell\}$ or $\{\ell+k \nsim -\ell\}$, they can be controlled by the sum over the larger sets as described above. 

\end{remark}

\subsection{Preliminary lemmas}

We start with the correlation of the free field.

\begin{prop} \label{pr:ff_correlation}
	Recall that $\<1>_\eps$ and $\<1>$ are the stationary solutions to \eqref{eq:free_field}. Also recall the notations $\scal{k}_\eps$ and $\scal{k}$ introduced in~\eqref{eq:notation_k}. Then we have
	\begin{equation}\label{eq:ff_correlation}
	\E \left( \widehat{\<1>_\eps}(s,k) \widehat{\<1>_\eps}(t,\ell) \right) = \delta_{k,-\ell} \cdot \frac{e^{-|t-s|\scal{k}_\eps^2}}{2\scal{k}_\eps^2}\;. 
	\end{equation}
	The bound holds for all $\eps \geq 0$, where the $\eps=0$ case corresponds to $\widehat{\<1>}$ and $\scal{k}$. As a consequence, for every $\theta \in [0,1]$, we have
	\begin{equation} \label{eq:ff_time}
	\E |\widehat{\<1>_\eps}(s,k) - \widehat{\<1>_\eps}(t,k)|^{2} \lesssim |t-s|^{\theta} \cdot \frac{1}{\scal{k}_\eps^{2-2\theta}}
	\end{equation}
	uniformly over $s, t \in [0,1]$. We also have that
	\begin{equation} \label{eq:ff_variance}
	\E [|\<1>_\eps (t,x)|^2] = \frac{1}{2} \sum_{k \in \Z^3} \frac{1}{\scal{k}_\eps^2} = \frac{\sigma^2}{\eps} + o_\eps(1)\;, 
	\end{equation}
	where we recall that $\sigma^2$ is given in \eqref{eq:variance_whole_space}. 
\end{prop}
\begin{proof}
	We first derive the correlation function \eqref{eq:ff_correlation}. We have the formula
	\begin{equation*}
	\widehat{\<1>_\eps}(t,k) = \int_{-\infty}^{t} e^{-\scal{k}_\eps^2 (t-r)} \hxi(r,k) {\rm d}r\;, 
	\end{equation*}
	where $\{\widehat{\xi}(\cdot,k)\}$ are independent complex white noises (in time) except that $\widehat{\xi}(\cdot,k) = \overline{\widehat{\xi}(\cdot,-k)}$. Hence it satisfies $\E\big( \widehat{\xi}(r,k) \widehat{\xi}(r',\ell)\big) = \delta_{k,-\ell} \delta(r-r')$. The correlation relation \eqref{eq:ff_correlation} then follows. Note that it also works for $\eps = 0$ with $\<1>_\eps$ and $\scal{k}_\eps$ replaced by $\<1>$ and $\scal{k}$. The bound \eqref{eq:ff_time} is a direct consequence of the correlation relation in \eqref{eq:ff_correlation}. 
	
	As for \eqref{eq:ff_variance}, using~\eqref{eq:ff_correlation} and Parseval's equality, we see that
	\begin{equation}
	\E [|\<1>_\eps (t,x)|^2] = \sum_{k\in\Z^3} \E[|\widehat{\<1>_\eps}(t,k)|^2]
	= \sum_{k\in\Z^3}\frac{\eps^2}{2(\eps^2 + \qQ(2\pi \eps |k|) }.
	\end{equation}
	Making the change of variables $\theta = \eps k$ and a Riemann sum approximation we see that the latter term is asymptotically equal to
	\begin{equation}
	\frac{1}{2\eps} \int_{\R^3}\frac{1}{\qQ(2\pi|\theta|)}\, d\theta
	\end{equation}
	which yields \eqref{eq:ff_variance}. 
\end{proof}

\begin{rmk}
	As a consequence of the above proof we see that the random field $\sqrt{\eps} \<1>_\eps$ is stationary and has distribution $\nN(0,\sigma_\eps^2)$ at every space-time point, where
	\begin{equation} \label{eq:variance_ff_scale_eps}
	\sigma_\eps^2 = \frac{\eps}{2} \sum_{k \in \Z^3} \frac{1}{\scal{k}_\eps^2} = \frac{\eps^3}{2} \sum_{k \in \Z^3} \frac{1}{\eps^2 + \qQ(2 \pi \eps |k|)}\;, 
	\end{equation}
	which converges to $\sigma^2$ as $\eps \rightarrow 0$. 
\end{rmk}

\begin{lem}\label{le:estimatekeps}
	Recall \eqref{eq:Q_growth} which states that $\qQ(z) \gtrsim |z|^{3+\eta}$ for $|z|\geq 1$. For any $k, \ell_1,\ldots, \ell_{n}\in\Z^3$ and $\alpha \in [-\eta,1]$ we have the estimates
\begin{equation}
\sum_{j=1}^{n} \scal{\ell_j}_\eps^2 \gtrsim \prod_{j=1}^{n} \scal{\ell_j}_\eps^{\frac{2}{n}}\;,\quad \text{and} \quad \scal{k}_\eps\gtrsim \eps^{\frac{1-\alpha}{2}}\scal{k}^{\frac{3-\alpha}{2}}\;.
\end{equation}
\end{lem}
\begin{proof}
	The first inequality follows from the estimate 
	\begin{equation}
	\sum_{j=1}^{n} \scal{\ell_j}_\eps^2 \gtrsim \max_{j\in\{1,2,\ldots, n\}} \scal{\ell_j}_\eps^{2},
	\end{equation}
	whereas the second one follows from the assumptions \eqref{eq:Q_origin} and \eqref{eq:Q_growth}.
\end{proof}

\begin{lem} \label{le:convolution}
	For every $n \in \N$ and $\alpha \in (0, \frac{3}{n})$, we have the bound
	\begin{equation*}
	\sum \prod_{j=1}^{n} \frac{1}{\scal{\ell_j}^{3-\alpha}} \lesssim \frac{1}{\scal{k}^{3-n\alpha}}\;, 
	\end{equation*}
	where the sum is taken over $\{\ell_1, \dots, \ell_n\} \in (\Z^3)^{n}$ such that $\ell_1 + \cdots + \ell_n = k \in \Z^3$. The proportionality constant is independent of $k$. 
\end{lem}
\begin{proof}
	The case $n=2$ is the usual convolution estimate. The case $n \geq 3$ can be obtained by induction. 
\end{proof}

\begin{lem} \label{le:convolution_resonance}
	For $\beta, \gamma \in \R$ satisfying $\beta + \gamma > 3$, we have the bound
	\begin{equation*}
	\sum_{\stackrel{\ell_1+ \ell_2 = k}{\ell_1 \sim \ell_2}} \frac{1}{\scal{\ell_1}^{\beta}} \cdot \frac{1}{\scal{\ell_2}^{\gamma}} \lesssim \frac{1}{\scal{k}^{\beta+\gamma-3}}\;, 
	\end{equation*}
	uniformly over $k \in \Z^3$. 
\end{lem}
\begin{proof}
	This is the content of \cite[Lemma~4.2]{Phi43_pedestrians}. 
\end{proof}

Before we formulate the next proposition, we recall that for $n \in \N$, the $n$-th Hermite polynomial with parameter $\nu \geq 0$, denoted by $H_n(\cdot,\nu)$, is defined via
\begin{equation}
H_n(x;\nu) = (-1)^n e^{x^2/2\nu} \nu^n \frac{\partial^n}{\partial x^n} \big( e^{-x^2/2\nu} \big). 
\end{equation}
These polynomials satisfy the relation $\nu^{n/2} H_n(\tfrac{x}{\sqrt{\nu}};1) = H_n(x;\nu)$. Finally, we note that if $X\sim\nN(0,\nu)$, then for any nice function $f:\R \to\R$, we have the identity
\begin{equs}\label{eq:hermiteexpansion}
f(X) = \sum_{k \geq 0} c_k H_k(X;\nu),\quad\text{with}\quad
c_k = \frac{\E[f^{(k)}(X)]}{k!}.
\end{equs} 
To derive the formula for the $c_k$'s, one uses the orthogonality of the Hermite polynomials, Lemma 1.1.1 in~\cite{Nua06} (note however at this point that our normalisation differs from~\cite{Nua06}), and $k$-times integration by parts.

\begin{prop} \label{pr:chaos_expan_processes}
	Recall the definition of $\<1'>_\eps$, $\<2'>_\eps$ and $\<3'>_\eps$ from \eqref{eq:noises_V}. For each $m \in \N^{+}$ and $\eps>0$, define $a_{m}^{(\eps)}$ by
	\begin{equation} \label{eq:chaos_expan_coeff}
	a_{m}^{(\eps)} := \frac{\E V^{(2m+2)}(\sqrt{\eps} \<1>_\eps)}{6 \lambda \cdot (2m-1)!}\;.
	\end{equation}
	We have the chaos expansion
	\begin{equation} \label{eq:chaos_expan_processes}
	\begin{split}
	\<1'>_\eps &= \sum_{m=1}^{n-1} a_{m}^{(\eps)} \cdot \eps^{m-1} \<1>_\eps^{\diamond (2m-1)}\;, \qquad \<2'>_\eps = \sum_{m=1}^{n-1} \frac{a_m^{(\eps)}}{m} \cdot \eps^{m-1} \<1>_\eps^{\diamond (2m)}\;\\
	\<3'>_\eps &= \sum_{m=1}^{n-1} \frac{3 a_m^{(\eps)}}{m(2m+1)} \cdot \eps^{m-1} \<1>_\eps^{\diamond (2m+1)}\;.
	\end{split}
	\end{equation}
	Here, we use the shorthand $\<1>_\eps^{\diamond \ell}= H_\ell(\<1>_\eps, \tfrac{\sigma_\eps^2}{\eps})$, where $\sigma_\eps^2$ denotes the variance of the stationary Gaussian process $\sqrt{\eps} \, \<1>_\eps$.
\end{prop}
\begin{proof}
    For $V^{(3)}$, $V''$ and $V'$, we have by~\eqref{eq:hermiteexpansion} the chaos expansions
	\begin{equation*}
	\begin{split}
	V^{(3)}(\sqrt{\eps} \<1>_\eps) &= \sum_{m=1}^{n-1} \frac{\E V^{(2m+2)}(\sqrt{\eps} \<1>_\eps)}{(2m-1)!} \cdot (\sqrt{\eps} \<1>_\eps)^{\diamond (2m-1)}\;,\\
	V''(\sqrt{\eps} \<1>_\eps) &= \sum_{m=0}^{n-1} \frac{\E V^{(2m+2)}(\sqrt{\eps} \<1>_\eps)}{(2m)!} \cdot (\sqrt{\eps} \<1>_\eps)^{\diamond (2m)}\;,\\
	V'(\sqrt{\eps} \<1>_\eps) &= \sum_{m=0}^{n-1} \frac{\E V^{(2m+2)}(\sqrt{\eps} \<1>_\eps)}{(2m+1)!} \cdot (\sqrt{\eps} \<1>_\eps)^{\diamond (2m+1)}\;, 
	\end{split}
	\end{equation*}
	where the Wick product is with respect to the stationary Gaussian $\sqrt{\eps} \<1>_\eps$ with variance $\sigma_\eps^2$ as given in \eqref{eq:variance_ff_scale_eps}. Now note that $(\sqrt{\eps} \<1>_\eps)^{\diamond \ell} = \eps^{\frac{\ell}{2}} \<1>_\eps^{\diamond \ell}$. The conclusion then follows from the Definition \eqref{eq:noises_V} and the expression for $C_{\eps}^{(1)}$ in \eqref{eq:C1}. 
\end{proof}

\begin{prop} \label{pr:C2C3_expression}
	The renormalisation constants $C_{\eps}^{(2)}$ and $C_{\eps}^{(3)}$ defined in \eqref{eq:C2C3} satisfy the expressions
	\begin{equation} \label{eq:C2C3_expression}
	\begin{split}
	C_{\eps}^{(2)} &= \sum_{m=1}^{n-1} \frac{(a_m^{(\eps)})^{2}}{m^2} \cdot \eps^{2m-2} \; \E \big[ \widetilde{\iI_\eps}(\<1>_\eps^{\diamond (2m)}) \circ \<1>_\eps^{\diamond (2m)} \big]\;,\\
	C_\eps^{(3)} &= \sum_{m=1}^{n-2} \frac{3 a_{m}^{(\eps)} a_{m+1}^{(\eps)}}{m (2m+1)} \cdot \eps^{2m-1} \; \E \big[ \widetilde{\iI_\eps}(\<1>_\eps^{\diamond (2m+1)}) \circ \<1>_\eps^{\diamond (2m+1)} \big]\;, 
	\end{split}
	\end{equation}
	where $a_{m}^{(\eps)}$ is as given in \eqref{eq:chaos_expan_coeff}, and $\widetilde{\iI_\eps}$ is as given in \eqref{eq:convolution_stationary}. 
\end{prop}
\begin{proof}
	This is a direct consequence of the definition of the constants in \eqref{eq:C2C3}, the expressions in \eqref{eq:chaos_expan_processes}, and the fact that the (resonance) product between two elements in different homogeneous chaos has expectation zero. 
	
	We can also see that $C_\eps^{(2)}$ diverges logarithmically, and that $C_\eps^{(3)}$ is uniformly bounded in $\eps$. In fact, the only term with logarithmic divergence in $C_\eps^{(2)}$ is the one with $m=1$ in the sum, and all other terms contributing to that sum are uniformly bounded in $\eps$ also, see~\eqref{eq:C2_formula}. 
\end{proof}

Combining Proposition~\ref{pr:chaos_expan_processes} with Proposition~\ref{pr:C2C3_expression} we obtain the following corollary.
\begin{cor}\label{cor:32_chaos_expan}
	For $\<3'2'>_\eps$ we have the chaos expansion
	\begin{equation*}\label{e:32_chaos_expan}
	\begin{split}
	\<3'2'>_\eps = &\sum_{m=1}^{n-1} \frac{3 (a_m^{(\eps)})^{2}}{m^2 (2m+1)} \cdot\tau_{\eps,m} + \sum_{m=1}^{n-2} \frac{3 a_{m}^{(\eps)} a_{m+1}^{(\eps)}}{m (m+1) (2m+1)} \cdot\sigma_{\eps,m}\\
	&+ \sum_{m,\ell} \frac{3 a_m^{(\eps)} a_{\ell}^{(\eps)}}{\ell m (2m+1)} \cdot\nu_{\eps,\ell,m}\;, 
	\end{split}
	\end{equation*}
	where the last sum is taken over integers $1 \leq m,\ell \leq n-1$ such that $m-\ell \geq 1$ or $\ell - m \geq 2$. 
	Here
	\begin{equation}\label{e:32_chaos_expan_coeff}
	\begin{split}
	&\tau_{\eps,m}=\eps^{2m-2} \Big[ \widetilde{\iI_\eps}(\<1>_\eps^{\diamond (2m+1)}) \circ \<1>_\eps^{\diamond (2m)} - (2m+1) \; \E \big[ \widetilde{\iI_\eps}(\<1>_\eps^{\diamond (2m)}) \circ \<1>_\eps^{\diamond (2m)} \big] \cdot \<1>_\eps \Big]\;,\\
	&\sigma_{\eps,m}=\eps^{2m-1} \Big[ \widetilde{\iI_\eps}(\<1>_\eps^{\diamond (2m+1)}) \circ \<1>_\eps^{\diamond (2m+2)} - (2m+2) \; \E \big[ \widetilde{\iI_\eps}^{\diamond (2m+1)}) \circ \<1>_\eps^{\diamond (2m+1)} \big] \cdot \<1>_\eps \Big]\;,\\
	&\nu_{\eps,\ell,m}= \eps^{m+\ell-2} \Big[ \widetilde{\iI_\eps}(\<1>_\eps^{\diamond (2m+1)}) \circ \<1>_\eps^{\diamond (2\ell)} \Big]\;.
	\end{split}
	\end{equation}
\end{cor}

\subsection{Proof of Theorem~\ref{th:main_conv_stochastic_obj}}

According to Proposition~\ref{pr:criterion_convergence} and the definition of $\xX$, we need to check the bounds \eqref{eq:criterion_space_reg} for all the seven components in $\Ups_\eps$ and $\Ups_\eps - \Ups$ (with the suitable $\alpha$ given in Table~\ref{table:noise}), and also to check the bounds \eqref{eq:criterion_process} for $\<3'0>_\eps$ (and also $\<3'0>_\eps - \<30>$). For simplicity of the presentation, we only provide details for three of the components $\<1'>_\eps$, $\<3'0>_\eps$ and $\<3'2'>_\eps$. The first one is the simplest, and gives a good illustration of the methods and techniques that are used. The third one is the most complicated and is quite subtle. The middle one is the only one that requires additional H\"older regularity in time. 

Also, for the latter two symbols alluded to above we provide only the first bound (uniform-in-$\eps$ bound) in both \eqref{eq:criterion_space_reg} and \eqref{eq:criterion_process}. The other one (convergence in $\eps$) can be obtained in a similar way. 

\subsubsection[Term Ladder]{The process \texorpdfstring{$\<1'>_\eps$}{Ladder}}

To show that $\E \sup_{t \in [0,1]} \| \<1'>_\eps - \<1>\|_{-\frac{1}{2}-\kappa}^{p} \rightarrow 0$ as $\eps \rightarrow 0$ according to Proposition~\ref{pr:criterion_convergence}, it suffices to show that there exists $\delta>0$ such that
\begin{equation*}
\sup_{\eps \in (0,1)} \E |\widehat{\<1'>}_\eps(t,k)|^{2} \lesssim \scal{k}^{-2+2\kappa-\delta}\;, \qquad \E |\widehat{\<1'>}_\eps(t,k) - \widehat{\<1>}(t,k)|^{2} \lesssim \eps^{\delta} \scal{k}^{-2+2\kappa-\delta}
\end{equation*}
for all sufficiently small $\delta>0$. By the chaos expansion in Proposition~\ref{pr:chaos_expan_processes} and the fact that the first coefficient  satisfies $a_{1}^{(\eps)} \rightarrow 1$ as $\eps \rightarrow 0$, the desired bound will follow from the following proposition. 

\begin{prop}
For all sufficiently small $\delta>0$, we have the bounds
\begin{equation} \label{eq:1_uniform_convergence}
\E |\widehat{\<1>}_\eps(t,k)|^{2} \lesssim \scal{k}^{-2}\;, \quad \E[|\widehat{\<1>}_\eps(t,k)- \widehat{\<1>}(t,k)|^2]\lesssim \eps^{2\delta} \langle k\rangle ^{-2+2\delta}\;,
\end{equation}
and
\begin{equation} \label{eq:1_high}
\eps^{2(m-1)} \E |\widehat{\<1>_\eps^{\diamond (2m-1)}}(t,k)|^{2} \lesssim \eps^{2\delta} \scal{k}^{-2+2\delta}\;, \qquad m \geq 2\;. 
\end{equation}
Both proportionality constants are uniform in $\eps \in (0,1)$, $k \in \Z^3$. 
\end{prop}
\begin{proof}
The first bound in \eqref{eq:1_uniform_convergence} follows directly from the formula \eqref{eq:ff_correlation}. As for the second one, note that 
\begin{equation}
\widehat{\<1>_\eps}(t,k)-\widehat{\<1>}(t,k) 
= \int_{-\infty}^{t} \left( e^{-\langle k\rangle_\eps^2 (t-r)}- e^{-\langle k\rangle^2 (t-r)} \right)\,  \hxi(r,k) {\rm d}r. 
\end{equation}
Hence, taking the second moment of the above expression gives
\begin{equation} \label{eq:1_convergence}
\begin{aligned}
\E |\widehat{\<1>_\eps}(t,k)-\widehat{\<1>}(t,k)|^2
&= \int_{-\infty}^t \left( e^{- (t-r) \langle k\rangle_\eps^2}- e^{- (t-r) \langle k\rangle^2} \right)^2\, {\rm d}r\\
&= \int_{-\infty}^{t} e^{- 2(t-r) \left(\scal{k}_\eps^2 \wedge \scal{k}^{2} \right)} \left( 1 - e^{-(t-r) |\scal{k}_\eps^2 - \scal{k}^2|} \right)^{2} {\rm d}r\;.
\end{aligned}
\end{equation}
If $\eps \scal{k} \geq 1$, then we bound the term in the parenthesis above by $1$. Integrating $r$ out and using $\scal{k} \lesssim \scal{k}_\eps$, we obtain
\begin{equation*}
\E |\widehat{\<1>_\eps}(t,k) - \widehat{\<1>}(t,k)|^{2} \lesssim \scal{k}^{-2} \lesssim \eps^{2\delta} \scal{k}^{-2+2\delta}\;, \quad \eps \scal{k} \geq 1\;.
\end{equation*}
If $\eps \scal{k} \leq 1$, we use~\eqref{eq:Q_origin} to get
\begin{equation*}
\left( 1 - e^{-(t-r) |\scal{k}_\eps^2 - \scal{k}^2|} \right)^{2} \lesssim (t-r)^{\delta} |\scal{k}_\eps^2 - \scal{k}^2|^{\delta} \lesssim (t-r)^{\delta} \eps^{2\delta} \scal{k}^{4\delta}\;.
\end{equation*}
Substituting it back into the right hand side of \eqref{eq:1_convergence} and integrating out $r$, we again get the bound $\eps^{2\delta} \scal{k}^{-2+2\delta}$ for a possibly slightly different $\delta$. This completes the second claim in \eqref{eq:1_uniform_convergence}. 

We now turn to \eqref{eq:1_high}. We have the identity
\begin{equation*}
\widehat{\<1>_\eps^{\diamond (2m-1)}}(t,k) = \sum_{\pP(2m-1,k)} \widehat{\<1>_\eps}(t,\ell_1) \diamond \cdots \diamond \widehat{\<1>_\eps}(t,\ell_{2m-1})\;,
\end{equation*}
where the notation $\pP(2m-1,k)$ was introduced at the beginning of Section~\ref{sec:Notations}. By Wick's formula and the correlation relation \eqref{eq:ff_correlation}, we have
\begin{equation*}
\E |\widehat{\<1>_\eps^{\diamond (2m-1)}}(t,k)|^{2} = (2m-1)! \sum_{\pP(2m-1,k)} \prod_{j=1}^{2m-1} \E |\widehat{\<1>_\eps}(t,\ell_j)|^{2} \lesssim \sum_{\pP(2m-1,k)} \prod_{j=1}^{2m-1} \frac{1}{\scal{\ell_j}_\eps^2}\;.
\end{equation*}
By Lemma~\ref{le:estimatekeps}, we have $\scal{\ell_j}_\eps^2 \gtrsim \eps^{1-\alpha} \scal{\ell_j}^{3-\alpha}$ for $\alpha \in (0,1)$. Substituting it into the above bound, we get
\begin{equation*}
\eps^{2(m-1)} \E |\widehat{\<1>_\eps^{\diamond (2m-1)}}(t,k)| \lesssim \eps^{2(m-1)-(2m-1)(1-\alpha)} \sum_{\pP(2m-1,k)} \prod_{j=1}^{2m-1} \frac{1}{\scal{\ell_j}^{3-\alpha}}\;.
\end{equation*}
Taking $\alpha = \frac{1+2\delta}{2m-1}$ for $\delta$ sufficiently small and applying Lemma~\ref{le:convolution}, we obtain the bound \eqref{eq:1_high}. Note that the requirement $\alpha \in (0,1)$ of the above step is satisfied for this choice of $\alpha$ only if $m \geq 2$. This completes the proof of the proposition. 
\end{proof}

\subsubsection[Term Ladder]{The process \texorpdfstring{$\<3'0>_\eps$}{Ladder}}

We now treat the term $\<3'0>_\eps$. According to the definition of the norm $\xX$ in \eqref{eq:norm_enhanced_noise}, we need to show the convergence to $\<30>$ both in $\cC([0,1], \bB^{\frac{1}{2}-\kappa}(\T^3))$ and in $\cC^{\frac{1}{8}}([0,1]; \bB^{\frac{1}{4}-\kappa}(\T^3))$. 

Taking the Fourier transform of $\<3'0>_\eps$ and using Proposition~\ref{pr:chaos_expan_processes}, we get
\begin{equation*}
\widehat{\<3'0>}_\eps(t,k) = \sum_{m=1}^{n-1} \frac{3 a_m^{(\eps)}}{m (2m+1)} \widehat{\tau_{\eps,m}}(t,k)\;, 
\end{equation*}
where
\begin{equation*}
\widehat{\tau_{\eps,m}}(t,k) = \eps^{m-1} \int_{-\infty}^{t} e^{-(t-r) \scal{k}_\eps^2} \cdot \widehat{\<1>_\eps^{\diamond (2m+1)}}(r,k) {\rm d}r\;.
\end{equation*}
By Proposition~\ref{pr:criterion_convergence} and the fact that $a_{1}^{(\eps)} \rightarrow 1$, the two desired convergences will follow from the following two propositions. 

\begin{prop} \label{pr:30_bound}
	For all sufficiently small $\delta$, we have the bounds
	\begin{equation} \label{eq:30_uniform_converge}
	\E |\widehat{\tau_{\eps,1}}(t,k)|^{2} \lesssim \frac{1}{\scal{k}^{4}}\;, \quad \E |\widehat{\tau_{\eps,1}}(t,k) - \widehat{\<30>}(t,k)|^{2} \lesssim \frac{\eps^{2\delta}}{\scal{k}^{4-2\delta}}\;,
	\end{equation}
	and
	\begin{equation} \label{eq:30_high}
	\E |\widehat{\tau_{\eps,m}}(t,k)|^{2} \lesssim \frac{\eps^{2\delta}}{\scal{k}^{4-2\delta}}\;, \quad m \geq 2\;.
	\end{equation}
	All proportionality constants are independent of $\eps$. 
\end{prop}
\begin{prop} \label{pr:30_bound_time}
	For all sufficiently small $\delta$, we have the bounds
	\begin{equation} \label{eq:30_uniform_converge_time}
	\begin{split}
	&\E |\widehat{\tau_{\eps,1}}(t,k) - \widehat{\tau_{\eps,1}}(s,k)|^{2} \lesssim (t-s)^{\frac{1}{4}} \scal{k}^{-\frac{7}{2}}\;,\\
	&\E |\widehat{\tau_{\eps,1}}(t,k) - \widehat{\tau_{\eps,1}}(s,k) - \big( \widehat{\<30>}(t,k) - \widehat{\<30>}(s,k) \big)|^{2} \lesssim \eps^{2\delta} (t-s)^{\frac{1}{4}} \scal{k}^{-\frac{7}{2}+2\delta}\;,
	\end{split}
	\end{equation}
	and
	\begin{equation} \label{eq:30_high_time}
	\E |\widehat{\tau_{\eps,m}}(t,k) - \widehat{\tau_{\eps,m}}(s,k)|^{2} \lesssim \eps^{2\delta} (t-s)^{\frac{1}{4}} \scal{k}^{-\frac{7}{2}-2\delta}\;. 
	\end{equation}
	The bounds are uniform in $\eps \in (0,1)$, $k \in \Z^3$ and $0 \leq s \leq t \leq 1$. 
\end{prop}
We first prove Proposition~\ref{pr:30_bound}.
\begin{proof} [Proof of Proposition~\ref{pr:30_bound}]
	We start with the expression
	\begin{equation*}
	\widehat{\tau_{\eps,m}}(t,k) = \eps^{m-1} \int_{-\infty}^{t} e^{-\scal{k}_{\eps}^{2}(t-s)} \sum_{\pP(2m+1,k)} \widehat{\<1>_\eps}(s,\ell_1) \diamond \cdots \diamond \widehat{\<1>_\eps}(s,\ell_{2m+1}) {\rm d}s\;. 
	\end{equation*}
	By Wick's formula and the correlation relation \eqref{eq:ff_correlation}, we have
	\begin{equation} 
	\begin{split}
	\E |\widehat{\tau_{\eps,m}}(t,k)|^{2} = &2 \times (2m+1)! \times \eps^{2(m-1)}\\
	&\times \int_{-\infty}^{t} \int_{-\infty}^{s} e^{-(2t-s-r) \scal{k}_{\eps}^{2}} \sum_{\pP(2m+1,k)} \prod_{j=1}^{2m+1} \frac{e^{-\scal{\ell_j}_{\eps}^{2}(s-r)}}{2 \scal{\ell_j}_\eps^2} {\rm d}r {\rm d}s\;.
	\end{split}
	\end{equation}
	Succesively integrating out $r$ and $s$, we obtain the bound
	\begin{equation} \label{eq:30_intermediate}
	\E |\widehat{\tau_{\eps,m}}(t,k)|^{2} \lesssim_m \frac{\eps^{2(m-1)}}{\scal{k}_\eps^2} \sum_{\pP(2m+1,k)} \bigg[ \Big( \prod_{j=1}^{2m+1} \scal{\ell_j}_\eps^2 \Big)^{-1} \Big( \scal{k}_\eps^{2} + \sum_{j=1}^{2m+1} \scal{\ell_j}_\eps^2 \Big)^{-1} \bigg]\;.
	\end{equation}
	For $m=1$, we have
	\begin{equation*}
	\E |\widehat{\tau_{\eps,1}}(t,k)|^{2} \lesssim \frac{1}{\scal{k}^{4-\delta}} \sum_{\ell_1 + \ell_2 + \ell_3 = k} \frac{1}{\scal{\ell_1}^{2+\delta} \scal{\ell_2}^2 \scal{\ell_3}^2} \lesssim \frac{1}{\scal{k}^4}\;, 
	\end{equation*}
	which is the first bound in \eqref{eq:30_uniform_converge}. The second bound in~\eqref{eq:30_uniform_converge} can be obtained in a similar way and we omit the details. 
	
	We now turn to the case $m \geq 2$. By Lemma~\ref{le:estimatekeps} for $\alpha\in (0,1)$ we have
	\begin{equation*}
	\Big( \prod_{j=1}^{2m+1} \scal{\ell_j}_\eps^2 \Big) \Big( \scal{k}_\eps^2 + \sum_{j=1}^{2m+1} \scal{\ell_j}_\eps^2 \Big) \gtrsim \prod_{j=1}^{2m+1} \scal{\ell_j}_\eps^{2 + \frac{2}{2m+1}} \gtrsim \prod_{j=1}^{2m+1} \big( \eps^{1-\alpha} \scal{\ell_j}^{3-\alpha} \big)^{1+\frac{1}{2m+1}}\;.
	\end{equation*}
    Substituting it back into \eqref{eq:30_intermediate}, we get
	\begin{equation*}
	\E |\widehat{\tau_{\eps,m}}(t,k)|^{2} \lesssim \frac{\eps^{(2m+2)\alpha -4}}{\scal{k}_\eps^2}\sum_{\pP(2m+1,k)} \bigg( \prod_{j=1}^{2m+1} \frac{1}{\scal{\ell_j}^{3-\frac{(2m+2)\alpha-3}{2m+1}}} \bigg)\;.
	\end{equation*}
	Now we take $\alpha = \frac{2+\delta}{m+1}$ for sufficiently small $\delta$ so that $(2m+2)\alpha-4 = 2\delta$. This $\alpha$ belongs to $(0,1)$ only when $m \geq 2$. Also with this choice of $\alpha$, the exponent
	\begin{equation*}
	\frac{(2m+2)\alpha-3}{2m+1} = \frac{1+2\delta}{2m+1} < \frac{3}{2m+1}
	\end{equation*}
	satisfies the hypothesis of Lemma~\ref{le:convolution}. A direct application of that lemma gives the bound \eqref{eq:30_high}, and concluding the proof of the lemma.  
\end{proof}

We now come to the proof of Proposition~\ref{pr:30_bound_time}.
\begin{proof}
	Again, we prove the uniform boundedness only, that is, the first bound in \eqref{eq:30_uniform_converge_time} and \eqref{eq:30_high_time}. For $m \geq 1$, and $t \geq s$ we have
	\begin{equation} \label{eq:30_time_split}
	\begin{split}
	\widehat{\tau_{\eps,m}}(t,k) - \widehat{\tau_{\eps,m}}(s,k) = &\eps^{m-1} \int_{-\infty}^{s} \big( e^{-(t-r)\scal{k}_\eps^2} - e^{-(s-r)\scal{k}_\eps^2} \big) \; \widehat{\<1>_\eps^{\diamond (2m+1)}}(r,k) {\rm d}r\\
	&+ \eps^{m-1} \int_{s}^{t} e^{-(t-r) \scal{k}_\eps^2} \; \widehat{\<1>_\eps^{\diamond (2m+1)}}(r,k) {\rm d}r\;.
	\end{split}
	\end{equation}
	For the first term on the right hand side above, since
	\begin{equation*}
	\E \Big( \widehat{\<1>_\eps^{\diamond (2m+1)}}(r_1,k) \widehat{\<1>_\eps^{\diamond (2m+1)}}(r_2,-k) \Big) \lesssim \sum_{\pP(2m+1,k)} \prod_{j=1}^{2m+1} \frac{e^{-|r_1 - r_2| \scal{\ell_j}_\eps^2}}{\scal{\ell_j}_\eps^2}\;, 
	\end{equation*}
	taking the second moment of that term and successivly integrating out $r_1$ and $r_2$, we see that for every $\theta \in [0,1]$, we have the bound
	\begin{equation} \label{eq:30_time_intermediate}
	\begin{split}
	&\phantom{111}\E \Big| \eps^{m-1} \int_{-\infty}^{s} \big( e^{-(t-r) \scal{k}_\eps^2} - e^{-(s-r) \scal{k}_\eps^2} \big) \widehat{\<1>_\eps^{\diamond (2m+1)}}(r,k) {\rm d}r \Big|^{2}\\
	&\lesssim \eps^{2(m-1)} (t-s)^{\theta} \scal{k}_{\eps}^{-2+2\theta} \sum_{\pP(2m+1,k)} \bigg[ \frac{1}{\scal{k}_\eps^2 + \sum_{j=1}^{2m+1} \scal{\ell_j}_\eps^2}  \cdot \prod_{j=1}^{2m+1} \frac{1}{\scal{\ell_j}_\eps^2} \bigg]\;.
	\end{split}
	\end{equation}
	The second moment of the second term on the right hand side of \eqref{eq:30_time_split} also satisfies \eqref{eq:30_time_intermediate}. Hence, we arrive in the same situation of \eqref{eq:30_intermediate}. Taking $\theta=\frac{1}{4}$ gives the desired bounds. 
\end{proof}

\subsubsection[Term Ladder]{The process \texorpdfstring{$\<3'2'>_\eps$}{Ladder}}\label{sec:32}

We now come to the term $\<3'2'>_\eps$. According to the definition of the norm on $\xX$ in~\eqref{eq:norm_enhanced_noise} and Proposition~\ref{pr:criterion_convergence}, we need to show that
\begin{equation}\label{e:32_bound}
\sup_{k \in \Z^3} \Big( \scal{k}^{2-2\kappa} \E |\widehat{\<3'2'>_\eps}(t,k) - \widehat{\<32>}(t,k)|^{2} \Big) \rightarrow 0
\end{equation}
as $\eps \rightarrow 0$. According to the decomposition of $\<3'2'>_\eps$ in Corollary~\ref{cor:32_chaos_expan}, it suffices to consider the Fourier tranforms of $\tau_{\eps,m}$, $\sigma_{\eps,m}$ and $\nu_{\eps,\ell,m}$ defined therein. We have the following proposition. 


\begin{prop} \label{pr:32_Fourier_bound}
	For $\tau_{\eps,m}$, $\sigma_{\eps,m}$ and $\nu_{\eps,\ell,m}$ given in Corollary~\ref{cor:32_chaos_expan}, we have the following bounds. There exists a universal constant $C_0$ and $\delta>0$ such that for every $\Lambda > 0$, there exists $C(\Lambda)$ such that
	\begin{equation} \label{eq:32_Fourier_bound}
	\begin{split}
	\E |\widehat{\tau_{\eps,1}}(t,k) - \widehat{\<32>}(t,k)|^{2} &\leq \Big( C(\Lambda) \eps^{2\delta} + \frac{C_0}{\Lambda} \Big) \scal{k}^{-2+2\kappa}\;,\\
	\E |\widehat{\tau_{\eps,m}}(t,k)|^{2} &\leq \Big( C(\Lambda) \eps^{2\delta} + \frac{C_0}{\Lambda} \Big) \scal{k}^{-2+2\kappa}\;, \quad m \geq 2\;.
	\end{split}
	\end{equation}
	$\widehat{\sigma_{\eps,m}}$ and $\widehat{\nu_{\eps,\ell,m}}$ satisfy the second bound above for all $\ell$ and $m$. 
\end{prop}

The above proposition implies that $\tau_{\eps,1}$ converges to $\<32>$ in the desired space, while all other components of $\<3'2'>_\eps$ vanish. Since $a_{1}^{(\eps)} \rightarrow 1$ as $\eps \rightarrow 0$, this will imply the convergence of $\<3'2'>_\eps$ to $\<32>$. It turns out that the analysis for $\widehat{\tau_{\eps,m}}$ is the hardest, so we will focus on the bound \eqref{eq:32_Fourier_bound} for $\widehat{\tau_{\eps,m}}$ only. 

\bigskip

\noindent

Following~\cite[page 24]{Phi43_pedestrians} we see that the Fourier transform of $\tau_{\eps,m}$ is
\begin{equation} \label{eq:32_Fourier_expression}
\begin{split}
\widehat{\tau_{\eps,m}}(t,k) = &\eps^{2m-2} \Bigg[\sum \bigg[ \bigg( \int_{-\infty}^{t} e^{-(t-r) \scal{\sum_{j=1}^{2m+1} \ell_j}_\eps^2} \Big( \bdiamond_{j=1}^{2m+1} \widehat{\<1>_\eps}(r,\ell_j) \Big)\, {\rm d}r \bigg)\\
&\cdot \Big(\bdiamond_{j=1}^{2m} \widehat{\<1>_\eps}(t,\widetilde{\ell}_j) \Big) \bigg] - (2m+1) \; \E \big[ \widetilde{\iI_\eps}(\<1>_\eps^{\diamond (2m)}) \circ \<1>_\eps^{\diamond (2m)} \big] \cdot \widehat{\<1>_\eps}(t,k) \Bigg]\;,
\end{split}
\end{equation}
where the sum in the first line is taken over $(\ell_{1}, \dots, \ell_{2m+1}, \widetilde{\ell}_1, \dots, \widetilde{\ell}_{2m}) \in \pP(4m+1,k)$ with the further restriction that $\ell_{1} + \cdots + \ell_{2m+1} \sim \widetilde{\ell}_1 + \cdots \widetilde{\ell}_{2m}$\footnote{See Section~\ref{sec:Notations} for the precise meaning of this notation.}.

Note that it has homogeneous Wiener chaos components of orders $1, 3, \dots, 4m+1$. We will deal separately with the different chaos components of $\widehat{\tau_{\eps,m}}$, and we will denote the chaos component of order $i\in\{1,3,\ldots, 2m+1\}$ by $\widehat{\tau_{\eps,m}}^{(i)}$. 
\medskip

\noindent
\textbf{Analysis of $\widehat{\tau_{\eps,m}}^{(1)}$:}
\medskip

For this term, we will frequently consider the sum over $2m$ parameters $(\ell_1, \dots, \ell_{2m})$ in some domain in $(\Z^3)^{2m}$. The sum of these parameters will be frequently appearing, and hence for convenience, we always write
\begin{equation*}
\ell = \ell_1 + \cdots + \ell_{2m}\;.
\end{equation*}
Let $G_{\eps,m}$ be the integration kernel whose Fourier coefficients are given by
\begin{equation} \label{eq:kernel_Gm}
\widehat{G_{\eps,m}}(t-r,k) = \frac{(2m+1)!}{2^{2m}} \sum_{\ell+k \sim -\ell} \frac{e^{-(t-r) \left(\scal{\ell+k}_\eps^2 + \sum_{j=1}^{2m} \scal{\ell_j}_\eps^2 \right)}}{\prod_{j=1}^{2m} \scal{\ell_j}_\eps^2}\;.
\end{equation}
Here, the sum is taken over all $2m$-tuples $(\ell_1, \dots, \ell_{2m})$ such that $\ell +k \sim -\ell$. This notation is well defined according to Section~\ref{sec:Notations} since we always have $\ell+k + (-\ell) = k$. 

Now we have the expression
\begin{equation*}
\begin{split}
\E \big[ \widetilde{\iI_\eps}(\<1>_\eps^{\diamond (2m)}) \circ \<1>_\eps^{\diamond (2m)} \big] &= \frac{1}{2m+1} \int_{-\infty}^{t} \widehat{G_{\eps,m}}(t-r,0) {\rm d}r\\
&= \frac{(2m)!}{2^{2m}} \sum_{\ell_1, \dots, \ell_{2m} \in \Z^3} \frac{1}{\big( \prod_{j=1}^{2m} \scal{\ell_j}_\eps^2 \big) \cdot \big( \scal{\ell}_\eps^2 + \sum_{j=1}^{2m}\scal{\ell_j}_\eps^2 \big)}\;.
\end{split}
\end{equation*}
Comparing this with Proposition~\ref{pr:C2C3_expression}, we see that
\begin{equation}\label{eq:C2_formula}
C_{\eps}^{(2)} = \sum_{m=1}^{n-1} \frac{(a_m^{(\eps)})^{2}}{m^2 (2m+1)} \cdot \eps^{2m-2} \int_{-\infty}^{t} \widehat{G_{\eps,m}}(t-r,0) {\rm d}r\;,
\end{equation}
and hence we can express $\widehat{\tau_{\eps,m}}^{(1)}$ in terms of $\widehat{G_{\eps,m}}$ as
\begin{equation*}
\widehat{\tau_{\eps,m}}^{(1)}(t,k) = \eps^{2m-2} \Big[\int_{-\infty}^{t} \widehat{G_{\eps,m}}(t-r,k) \widehat{\<1>_\eps}(r,k) {\rm d}r - \int_{-\infty}^{t} \widehat{G_{\eps,m}}(t-r,0) {\rm d}r \cdot \widehat{\<1>_\eps}(t,k) \Big]\;.
\end{equation*}
The first chaos component of the limiting object $\widehat{\<32>}$ has the expression
\begin{equation*}
\widehat{\<32>}^{(1)}(t,k) = \int_{-\infty}^{t} \widehat{G}(t-r,k) \widehat{\<1>}(r,k) {\rm d}r - \int_{-\infty}^{t} \widehat{G}(t-r,0) {\rm d}r \cdot \widehat{\<1>}(t,k)\;,
\end{equation*}
where
\begin{equation*}
\widehat{G}(t-r,k) = \frac{3}{2} \sum_{\ell_1 + \ell_2 + k \sim -(\ell_1 + \ell_2)} \frac{e^{-(t-r) \big( \scal{\ell_1 + \ell_2 + k}^2 + \scal{\ell_1}^2 + \scal{\ell_2}^2 \big)}}{\scal{\ell_1}^2 \scal{\ell_2}^2}\;.
\end{equation*} 
This is formally $\widehat{G_{0,1}}$ (by setting $\eps=0$). In order to show the convergence of $\widehat{\tau_{\eps,m}}^{(1)}$ to $\widehat{\<32>}^{(1)}$, we rewrite it as
\begin{equation}\label{eq:tau_eps_m_1}
\begin{split}
\widehat{\tau_{\eps,m}}^{(1)}(t,k) &= \eps^{2m-2}  \int_{-\infty}^{t} \widehat{G_{\eps,m}}(t-r,k) \Big( \widehat{\<1>_\eps}(r,k) - \widehat{\<1>_\eps}(t,k) \Big) {\rm d}r\\
&+ \eps^{2m-2} \int_{-\infty}^{t} \Big( \widehat{G_{\eps,m}}(t-r,k) - \widehat{G_{\eps,m}}(t-r,0) \Big) {\rm d}r \cdot \widehat{\<1>_\eps}(t,k)\;.
\end{split}
\end{equation}
We control the second moment of the two terms separately. 
\medskip

\noindent
\textbf{(a) Analysis of the first term in~\eqref{eq:tau_eps_m_1}:}
\medskip

\noindent

We first show that for $m \geq 2$, its second moment vanishes as in \eqref{eq:32_Fourier_bound}. By positivity of $\widehat{G_{\eps,m}}$ and the triangle inequality, we have
\begin{equation*}
\begin{split}
&\eps^{4(m-1)} \E \Big| \int_{-\infty}^{t} \widehat{G_{\eps,m}}(t-r,k) \Big( \widehat{\<1>_\eps}(r,k) - \widehat{\<1>_\eps}(t,k) \Big) {\rm d}r \Big|^{2}\\
&\lesssim \eps^{4(m-1)} \bigg[ \int_{-\infty}^{t} \widehat{G_{\eps,m}}(t-r,k) \Big( \E |\widehat{\<1>_\eps}(r,k) - \widehat{\<1>_\eps}(t,k)|^{2} \Big)^{\frac{1}{2}} {\rm d}r \bigg]^{2}\;, 
\end{split}
\end{equation*}
Now, using \eqref{eq:ff_time}, we have
\begin{equation*}
\E |\widehat{\<1>}_\eps(r,k) - \widehat{\<1>}_\eps(t,k)|^{2} \lesssim (t-r)^{2\delta} \scal{k}_\eps^{-2+4\delta}\;. 
\end{equation*}
Substituting it back into the last expression, we obtain
\begin{equation} \label{eq:32_continuity}
\begin{split}
\eps^{4(m-1)} &\E \Big| \int_{-\infty}^{t} \widehat{G_{\eps,m}}(t-r,k) \Big( \widehat{\<1>_\eps}(r,k) - \widehat{\<1>_\eps}(t,k) \Big) {\rm d}r \Big|^{2}\\
&\lesssim \eps^{4(m-1)} \scal{k}_\eps^{-2+4\delta} \Big( \int_{-\infty}^{t} (t-r)^{\delta} \widehat{G_{\eps,m}}(t-r,k) {\rm d}r \Big)^{2}\;.
\end{split}
\end{equation}
Thus, it remains to bound the quantity
\begin{equation*}
\Big( \eps^{2(m-1)} \int_{-\infty}^{t} (t-r)^{\delta} \widehat{G_{\eps,m}}(t-r,k) {\rm d}r \Big)^{2}\;.
\end{equation*}
Note that by the expression of $\widehat{G_{\eps,m}}$, a change of variable yields
\begin{equation*}
\int_{-\infty}^{t} (t-r)^\delta \widehat{G_{\eps,m}}(t-r,k) {\rm d}r =_{m,\delta} \sum_{\stackrel{\ell+k}{\sim -\ell}} \bigg[ \Big( \prod_{j=1}^{2m} \frac{1}{\scal{\ell_j}_\eps^2} \Big) \cdot \frac{1}{\left( \scal{\ell+k}_\eps^2 + \sum_{j=1}^{2m} \scal{\ell_j}_\eps^2 \right)^{1+\delta}} \bigg]\;,
\end{equation*}
where $=_{m,\delta}$ means that both sides are equal modulo a multiplicative constant that depends on $m$ and $\delta$, but which is bounded in $\delta$ as $\delta\to 0$.
Now, for the denominator of the term inside the square bracket, we have by Lemma~\ref{le:estimatekeps} the bound
\begin{equation*}
\begin{split}
\Big(\prod_{j=1}^{2m}\scal{\ell_j}_\eps^{2} \Big) \cdot \Big( \scal{\ell+k}_\eps^2 + \sum_{j=1}^{2m}\scal{\ell_j}_\eps^2 \Big)^{1+\delta} 
&\gtrsim \prod_{j=1}^{2m} \scal{\ell_j}_\eps^{2+\frac{1+\delta}{m}} \gtrsim \prod_{j=1}^{2m} \Big( \eps^{\frac{1-\alpha}{2}} \scal{\ell_j}^{\frac{3-\alpha}{2}} \Big)^{2+\frac{1+\delta}{m}}\;
\end{split}
\end{equation*}
for any $\alpha \in (0,1)$. Hence, we have
\begin{equation*}
\begin{split}
\eps^{2(m-1)} \int_{-\infty}^{t} (t-r)^{\delta} \widehat{G_{\eps,m}}(t-r,k) {\rm d}r
&\lesssim \eps^{(2m+1+\delta) \alpha - (3+\delta)} \sum_{\ell} \prod_{j=1}^{2m} \frac{1}{\scal{\ell_j}^{(3-\alpha)(1+\frac{1+\delta}{2m})}}\;,
\end{split}
\end{equation*}
where we have relaxed the sum to all of $(\Z^3)^{2m}$. Now, we choose $\alpha = \frac{3+2\delta}{2m+1+\delta}$, which is smaller than $1$ for sufficiently small $\delta$ only when $m \geq 2$. This choice of $\alpha$ yields
\begin{equation*}
(2m+1+\delta)\alpha - (3+\delta) = \delta\;, \qquad (3-\alpha) \big( 1+\frac{1+\delta}{2m} \big) = \frac{6m+\delta}{2m}>3\;.
\end{equation*}
Hence, the above sum is finite, and the right hand side of \eqref{eq:32_continuity} has the bound of the form \eqref{eq:32_Fourier_bound} for $m \geq 2$ as long as we choose $\delta < \kappa$. 

For $m=1$, the convergence of the second moment of the difference
\begin{equation*}
\int_{-\infty}^{t} \widehat{G_{\eps,1}}(t-r,k) \Big( \widehat{\<1>_\eps}(r,k) - \widehat{\<1>_\eps}(t,k) \Big) {\rm d}r - \int_{-\infty}^{t} \widehat{G}(t-r,k) \Big( \widehat{\<1>}(r,k) - \widehat{\<1>}(t,k) \Big) {\rm d}r
\end{equation*}
can be obtained in a similar way by "borrowing" $\delta$ powers from the continuity of $\widehat{\<1>}$ in time. We omit the details. This completes the proof of the first term in \eqref{eq:tau_eps_m_1}.

\medskip

\noindent
\textbf{(b) Analysis of the second term in~\eqref{eq:tau_eps_m_1}:}
\medskip

\noindent
By Proposition~\ref{pr:ff_correlation} on the variance of the free field, it suffices to show that there exists a universal constant $C_0 > 0$ such that for every $\Lambda>0$, one has the bound
\begin{equation} \label{eq:32_shift}
\eps^{2m-2} \bigg| \int_{-\infty}^{t} \Big( \widehat{G_{\eps,m}}(t-r,k) - \widehat{G_{\eps,m}}(t-r,0) \Big) {\rm d}r \bigg| \leq \Big( C(\Lambda) \eps^{\delta} + \frac{C_0}{\Lambda} \Big) \scal{k}^{\delta}\;, \quad m \geq 2\;,
\end{equation}
and that the difference
\begin{equation*}
\bigg| \int_{-\infty}^{t} \Big( \widehat{G_{\eps,1}}(t-r,k) - \widehat{G_{\eps,1}}(t-r,0) \Big) {\rm d}r - \int_{-\infty}^{t} \Big( \widehat{G}(t-r,k) - \widehat{G}(t-r,0) \Big) {\rm d}r \bigg|
\end{equation*}
has the same upper bound. We give details for the bound \eqref{eq:32_shift} when $m \geq 2$, and the uniform in $\eps$ bound for $m=1$. The convergence to the limiting quantity (for $m=1$) can be shown in exactly the same way. 

We first compute the integral on the left hand side of \eqref{eq:32_shift}. By definition of $\widehat{G_{\eps,m}}$ in \eqref{eq:kernel_Gm}, we have
\begin{equation*}
\begin{split}
&\phantom{111}\frac{2^{2m}}{(2m+1)!} \Big( \widehat{G_{\eps,m}}(t-r,k) - \widehat{G_{\eps,m}}(t-r,0) \Big)\\
&= \sum_{k+\ell \sim -\ell} \bigg[ \Big( \prod_{j=1}^{2m} \frac{1}{\scal{\ell_j}_\eps^2} \Big) \Big( e^{-(t-r) \big( \scal{\ell+k}_\eps^2 + \sum_{j=1}^{2m} \scal{\ell_j}_\eps^2 \big)} - e^{-(t-r) \big( \scal{\ell}_\eps^2 + \sum_{j=1}^{2m} \scal{\ell_j}_\eps^2 \big)} \Big) \bigg]\\
&\phantom{11}- \sum_{k+\ell \nsim -\ell} \bigg[ \Big( \prod_{j=1}^{2m} \frac{1}{\scal{\ell_j}_\eps^2} \Big) \cdot e^{-(t-r) \big( \sum_{j=1}^{2m} \scal{\ell_j}_\eps^2 + \scal{\ell}_\eps^2  \big)} \bigg]\;,
\end{split}
\end{equation*}
where as before we used the abbreviation $\ell= \ell_1+\ldots \ell_{2m}$ with $(\ell_1, \dots, \ell_{2m}) \in (\Z^3)^{2m}$. 

Now, integrating out the $r$ variable, and using that by Remark~\ref{rem:resonance_restriction} we can enlarge the range of the sums to
\begin{equation*}
\aA_{m,k}^{(1)} = \left\{ (\ell_1, \dots \ell_{2m}): \scal{\ell} \geq \frac{\scal{k}}{10} \Big\} \right\}\;, \quad\text{and}\quad
\aA_{m,\eps}^{(2)} = \big\{ (\ell_1, \dots, \ell_{2m}): \scal{\ell} \leq 10 \scal{k} \big\}
\end{equation*}
respectively, we see the left hand side of \eqref{eq:32_shift} can be bounded (up to a constant multiple depending on $m$) by $\dD_{\eps,m}^{(1)} + \dD_{\eps,m}^{(2)}$, where
\begin{equation*}
\begin{split}
\dD_{\eps,m}^{(1)}(k) &= \eps^{2m-2} \sum_{\aA_{m,k}^{(1)}} \bigg[ \Big( \prod_{j=1}^{2m} \frac{1}{\scal{\ell_j}_\eps^2} \Big) \Big( \frac{1}{\scal{\ell+k}_\eps^2 + \sum_{j=1}^{2m} \scal{\ell_j}_\eps^2} - \frac{1}{\scal{\ell}_\eps^2 + \sum_{j=1}^{2m} \scal{\ell_j}_\eps^2} \Big) \bigg]\;,\\
\dD_{\eps,m}^{(2)}(k) &= \eps^{2m-2} \sum_{\aA_{m,k}^{(2)}} \Big( \prod_{j=1}^{2m} \frac{1}{\scal{\ell_j}_\eps^2} \Big) \cdot \frac{1}{\scal{\ell}_\eps^2 + \sum_{j=1}^{2m} \scal{\ell_j}_\eps^2}\;.
\end{split}
\end{equation*}
We treat $\dD_{\eps,m}^{(2)}$ first. By Lemma~\ref{le:estimatekeps}, we have
\begin{equation*}
\dD_{\eps,m}^{(2)}(k) \lesssim \eps^{2m-2} \sum_{\aA_{m,k}^{(2)}} \Big( \prod_{j=1}^{2m} \frac{1}{\scal{\ell_j}_\eps^{2+\frac{1}{m}}} \Big) \lesssim \eps^{(2m+1)\alpha-3} \sum_{\aA_{m,k}^{(2)}} \Big( \prod_{j=1}^{2m} \frac{1}{\scal{\ell_j}^{(3-\alpha)(1+\frac{1}{2m})}} \Big)
\end{equation*}
for $\alpha \in [0,1]$. For $m=1$, we have the uniform bound $\dD_{\eps,1}^{(2)}(k) \lesssim \scal{k}^{\delta}$ by choosing $\alpha = 1$. For $m \geq 2$, we choose $\alpha = \frac{3+\delta}{2m+1}$, which is smaller than $1$ if $\delta$ is sufficiently small. The exponent of $\scal{\ell_j}$ is then
\begin{equation*}
(3-\alpha)(1+\frac{1}{2m}) = 3 - \frac{\delta}{2m}\;,
\end{equation*}
which satisfies the assumption of Lemma~\ref{le:convolution}. Hence, we have
\begin{equation*}
\dD_{\eps,m}^{(2)}(k) \lesssim \eps^{\delta} \sum_{\ell: |\ell| \leq 10 |k|} \frac{1}{\scal{\ell}^{3-\delta}} \lesssim \eps^{\delta} \scal{k}^{\delta}\;, \qquad m \geq 2\;.
\end{equation*}
We now turn to $\dD_{\eps,m}^{(1)}$. We will first treat the case $m\geq 2$. A direct computation yields
\begin{equation}\label{eq:D1}
\dD_{\eps,m}^{(1)}(k) \lesssim \eps^{2m-2} \sum_{\aA_{m,k}^{(1)}} \bigg[ \Big( \prod_{j=1}^{2m} \frac{1}{\scal{\ell_j}_\eps^2} \Big) \cdot \frac{|\scal{\ell+k}_\eps^2 - \scal{\ell}_\eps^2|}{\big(\scal{\ell+k}_\eps^2+ \sum_{j=1}^{2m} \scal{\ell_j}_\eps^2 \big) \big(\scal{\ell}_\eps^2 + \sum_{j=1}^{2m} \scal{\ell_j}_\eps^2\big)} \bigg]\;.
\end{equation}
At this stage, we note that it is here that it turns out to be crucial that we do not use~\ref{eq:Q_derivative_growth} or any assumption that one might impose on the growth of the derivative of $\qQ$. Indeed, the difference $|\scal{\ell+k}_\eps^2 - \scal{\ell}_\eps^2|$ can then only be controlled by the maximum of these two if either $\scal{\ell}$ or $\scal{k}$ is larger than $\frac{1}{\eps}$. In this situation, the sum on the right hand side is critical with respect to $\eps$ and the dimension, and hence no smallness of $\eps$ can be obtained unless the sum for some $\ell_j$ starts far above $\frac{1}{\eps}$. This is the place where we can only obtain a qualitative convergence rather than a rate in terms of $\eps$. 

Fix $\Lambda >0$ as in the statement of the proposition. We decompose the domain of the sum into
\begin{equs}
\aA_{m,k}^{(1,1)} = \aA_{m,k}^{(1)} \cap \Big\{|\ell_j|\leq \frac{\Lambda}{\eps}\text{ for all }j\Big\},\quad \text{and} \quad \aA_{m,k}^{(1,2)}= \aA_{m,k}^{(1)} \setminus \aA_{m,k}^{(1,1)}.
\end{equs}
We denote the sum in~\eqref{eq:D1} over $\aA_{m,k}^{(1,i)}$ by $\dD_{\eps,m}^{(1,i)}$. We will first control $\dD_{\eps,m}^{(1,1)}$. We distinguish between the case $|k| \leq \frac{1}{\eps}$ and $|k| > \frac{1}{\eps}$. 
\begin{itemize}
	\item[(1)] We first treat the case $|k|\leq \frac{1}{\eps}$. By Assumption \eqref{eq:Q_origin}, the first derivative of $\qQ$ is locally Lipschitz continuous. Hence, for every $M>0$, there exists $C=C(M)$ such that $|\qQ'(z)| \leq C(M) \cdot |z|$ for all $z \in [0,M]$. Hence, if $|k| \leq \frac{1}{\eps}$ then on $\aA_{m,k}^{(1,1)}$, we have
	\begin{equation*}
	|\scal{\ell+k}_\eps^2-\scal{\ell}_\eps^2|
	= \frac{1}{\eps^2}\big|\qQ(2\pi \eps|k+\ell|)-\qQ(2\pi\eps |\ell|)\big| \lesssim_{\Lambda} \scal{\ell} \scal{k} \lesssim_{\Lambda} \scal{\ell}^{2-\delta} \scal{k}^{\delta}\;, 
	\end{equation*}
	where we used $\scal{\ell} \gtrsim \scal{k}$ in $\aA_{m,k}^{(1)}$ in the last inequality. Note that the proportionality constant depends on $\Lambda$, but the dependence cannot be quantified since we have no assumption on the growth of $\qQ'$. Hence, plugging it back into the right hand side of \eqref{eq:D1} and using Lemma~\ref{le:estimatekeps} so that $\sum_{j=1}^{2m} \scal{\ell_j}_\eps^{2+\delta} \gtrsim \prod_{j=1}^{2m} \scal{\ell_j}_\eps^{\frac{2+\delta}{2m}}$, we get
	\begin{equation*}
	\dD_{\eps,m}^{(1,1)}(k) \lesssim_{\Lambda} \eps^{2m-2} \scal{k}^{\delta} \prod_{j=1}^{2m} \Big( \sum_{|\ell_j| \leq \frac{\Lambda}{\eps}} \frac{1}{\scal{\ell_j}_\eps^{2+\frac{2+\delta}{2m}}} \Big)\;.
	\end{equation*}
	The sum in the parenthesis above (for each $j$) is uniformly bounded if $m=1$, and is bounded by $(\Lambda / \eps)^{1-\frac{2+\delta}{2m}}$ if $m \geq 2$. Hence, for $|k| \leq \frac{1}{\eps}$, we have
	\begin{equation*}
	\dD_{\eps,m}^{(1,1)}(k) \lesssim_\Lambda \begin{cases}
	\scal{k}^{\delta}\;, &\text{if } m = 1,\\
	\eps^{\delta} \scal{k}^{\delta}\;, &\text{if } m \geq 2,
	\end{cases}
	\end{equation*}
	
%
	
	\item[(2)] We now treat the case $|k|> \frac{1}{\eps}$. In that case we use the brutal estimate
	\begin{equation}\label{eq:D1brutalest}
	 \frac{|\scal{\ell+k}_\eps^2 - \scal{\ell}_\eps^2|}{\big(\scal{\ell+k}_\eps^2+ \sum_{j=1}^{2m} \scal{\ell_j}_\eps^2 \big) \big(\scal{\ell}_\eps^2 + \sum_{j=1}^{2m} \scal{\ell_j}_\eps^2\big)}\lesssim \frac{1}{\sum_{j=1}^{2m} \scal{\ell_j}_\eps^2}\;.
	\end{equation}
	Again, using Lemma~\ref{le:estimatekeps} to distribute the two powers to $\scal{\ell_j}$'s with $\frac{1}{m}$ each, we get
	\begin{equation*}
	\dD_{\eps,m}^{(1,1)}(k) \lesssim_{\Lambda} \eps^{2m-2} \prod_{j=1}^{2m} \Big( \sum_{|\ell_j| \leq \frac{\Lambda}{\eps}} \frac{1}{\scal{\ell_j}^{2+\frac{1}{m}}} \Big) \lesssim_{\Lambda} \begin{cases}
	|\log \eps|^{2} \lesssim \eps^{\delta} \scal{k}^{2\delta}\;, &\text{if } m = 1,\\
	1 \lesssim \eps^{\delta} \scal{k}^{\delta}\;, &\text{if } m \geq 2,
	\end{cases}
	\end{equation*}
	where we used $|k|>\frac{1}{\eps}$ in the last estimate. This concludes the bound for $\dD_{\eps,m}^{(1,1)}(k)$. 
\end{itemize}

We will now analyse the sum in~\eqref{eq:D1} over $\aA_{m,k}^{(1,2)}$. Note that up to a constant multiple of $2m$, we can instead consider the sum over
\begin{equation}
\Big\{(\ell_1,\ldots,\ell_{2m})\in (\Z^3)^{2m}:\, |\ell_1| > \frac{\Lambda}{\eps}\Big\}\;, 
\end{equation}
where now we also remove the restriction $\scal{\ell} \geq \frac{\scal{k}}{10}$. Using~\eqref{eq:D1brutalest} we see that
\begin{equation}
\frac{|\scal{\ell+k}_\eps^2 - \scal{\ell}_\eps^2|}{\big(\scal{\ell+k}_\eps^2+ \sum_{j=1}^{2m} \scal{\ell_j}_\eps^2 \big) \big(\scal{\ell}_\eps^2 + \sum_{j=1}^{2m} \scal{\ell_j}_\eps^2\big)}\lesssim \frac{1}{ \scal{\ell_1}_\eps^2}\;,
\end{equation}
Hence, we get
\begin{equs}
\dD_{\eps,m}^{(1,2)}(k)
\lesssim \eps^{2(m-1)} \Big(\sum_{|\ell_1|> \frac{\Lambda}{\eps}} \frac{1}{\scal{\ell_1}^4}\Big)
\Bigg(\prod_{j=2}^{2m}\Big(\sum_{\ell_j\in\Z^3}\frac{1}{\scal{\ell_j}_\eps^2}\Big)\Bigg)
\lesssim \frac{1}{\Lambda}\;, 
\end{equs}
where the proportionality constant is independent of $\eps$ and $\Lambda$, and the above bound is true for all $m \geq 1$. Here we have used the bounds
\begin{equation*}
\sum_{|\ell_1| > \frac{\Lambda}{\eps}} \frac{1}{\scal{\ell_1}^4} \lesssim \frac{\eps}{\Lambda}\;, \qquad \sum_{\ell_j \in \Z^3} \frac{1}{\scal{\ell_j}_\eps^2} \lesssim \sum_{|\ell_j| \leq \frac{1}{\eps}} \frac{1}{\scal{\ell_j}^2} + \sum_{|\ell_j| \geq \frac{1}{\eps}} \frac{1}{\eps^{1+\eta} \scal{\ell_j}^{3+\eta}} \lesssim \frac{1}{\eps}\;,
\end{equation*}
where the first inequality in the second term is a consequence of Assumption~\eqref{eq:Q_growth}.
This concludes the analysis of $\dD_{\eps,m}^{(1,2)}(k)$ and hence of $\dD_{\eps,m}^{(1)}(k)$ as well.

\medskip

\noindent
\textbf{Analysis of $\widehat{\tau_{\eps,m}}^{(4m+1)}$:}
\medskip

\noindent

We now give details on the bound for the highest chaos component of $\widehat{\tau_{\eps,m}}$. Note that
\begin{equation*}
\begin{split}
\widehat{\tau_{\eps,m}}^{(4m+1)}(t,k) = \eps^{2m-2} \sum_{\stackrel{\ell+\widetilde{\ell}=k}{\ell \sim \widetilde{\ell}}} \int_{-\infty}^{t} e^{-(t-r) \scal{\ell}_\eps^2} \Big( \bdiamond_{j=1}^{2m+1}\widehat{\<1>_\eps}(r,\ell_j) \Big) \diamond \Big( \bdiamond_{j=1}^{2m} \widehat{\<1>_\eps}(t,\widetilde{\ell}_j) \Big) {\rm d}r\;,
\end{split}
\end{equation*}
where this time we denote
\begin{equation*}
\ell = \sum_{j=1}^{2m+1} \ell_j \quad \text{and} \quad \widetilde{\ell} = \sum_{j=1}^{2m} \widetilde{\ell}_j\;,
\end{equation*}
and the sum is taken over the subset of $(\Z^3)^{4m+1}$ such that $\ell+\widetilde{\ell}=k$ and $\ell \sim \widetilde{\ell}$. 

The second moment of this quantity equals the sum of all possible contractions between different instances of $\widehat{\<1>_\eps}$, which yields a rather complicated expression. However, as observed in \cite[Section~10]{rs_theory} and in (\cite[Eq.(3.6)]{Phi43_pedestrians}), one can greatly simplify it by considering non-symmetric functions. In fact, one has the upper bound
\begin{equation}
\begin{split}
\E |\widehat{\tau_{\eps,m}}^{(4m+1)}&(t,k)|^2 \lesssim  
\eps^{4(m-1)} \sum_{\stackrel{\ell+\widetilde{\ell}=k}{\ell \sim \widetilde{\ell}}} \bigg[ \Big( \prod_{j=1}^{2m} \E |\widehat{\<1>_\eps}(t,\widetilde{\ell}_j)|^{2} \Big)\\
&\iint_{(-\infty,t)^{2}} \, e^{-(2t-r-r')\scal{\ell}_\eps^2} \Big( \prod_{j=1}^{2m+1} \E \big[ \widehat{\<1>_\eps}(r,\ell_j)\widehat{\<1>_\eps}(r',-\ell_j) \big] \Big) {\rm d}r' {\rm d}r \bigg]\;.
\end{split}
\end{equation}
Using Proposition~\ref{pr:ff_correlation} and successively integrating out $r$ and $r'$, we get the bound
\begin{equation} \label{eq:32_highest_chaos}
\E |\widehat{\tau_{\eps,m}}^{(4m+1)}(t,k)|^{2} \lesssim \eps^{4(m-1)} \sum_{\stackrel{\ell+\widetilde{\ell}=k}{\ell \sim \widetilde{\ell}}} \left[ \frac{1}{\scal{\ell}_\eps^2} \bigg( \prod_{j=1}^{2m+1} \frac{1}{\scal{\ell_j}_\eps^{2+\frac{2}{2m+1}}} \Big) \cdot \Big( \prod_{j=1}^{2m} \frac{1}{\scal{\widetilde{\ell}_j}_\eps^2} \bigg) \right]\;, 
\end{equation}
where we have used $\sum_{j=1}^{2m+1} \scal{\ell_j}_\eps^2 \gtrsim \prod_{j=1}^{2m+!} \scal{\ell_j}_\eps^{\frac{2}{2m+1}}$ to distribute the $\frac{2}{2m+1}$ powers to each $\scal{\ell_j}_\eps$. To control the sum on the right hand side above, we first note that by Lemma~\ref{le:estimatekeps}, we have
\begin{equation*}
\scal{\ell_j}_\eps \gtrsim \eps^{\frac{1-\alpha}{2}} \scal{\ell_j}^{\frac{3-\alpha}{2}}\;, \quad \scal{\widetilde{\ell_j}}_\eps \gtrsim \eps^{\frac{1-\widetilde{\alpha}}{2}} \scal{\widetilde{\ell}_j}^{\frac{3-\widetilde{\alpha}}{2}}\;, 
\end{equation*}
where $\alpha, \widetilde{\alpha} \in [0,1]$ will be specified later. Plugging it back into the right hand side of \eqref{eq:32_highest_chaos}, we get
\begin{equation*}
\E |\widehat{\tau_{\eps,m}}^{(4m+1)}(t,k)|^{2} \lesssim \eps^{(2m+2)\alpha + 2m \widetilde{\alpha} - 6} \sum_{\stackrel{\ell+\widetilde{\ell}=k}{\ell \sim \widetilde{\ell}}} \bigg[ \frac{1}{\scal{\ell}^2} \Big( \prod_{j=1}^{2m+1} \frac{1}{\scal{\ell_j}^{3 - \frac{(2m+2)\alpha-3}{2m+1}}} \Big) \cdot \Big( \prod_{j=1}^{2m} \frac{1}{\scal{\widetilde{\ell}_j}^{3-\widetilde{\alpha}}} \Big) \bigg]\;.
\end{equation*}
If $m \geq 2$ and $\delta>0$ is sufficiently small, we can choose $\alpha, \widetilde{\alpha} \in [0,1]$ such that all of the following hold:
\begin{equation} \label{eq:choice_alpha}
(2m+2)\alpha < 6\;, \quad 2m \widetilde{\alpha} < 3\;, \quad (2m+2)\alpha + 2m \widetilde{\alpha} - 6 = 2 \delta\;.
\end{equation}
The first two requirements guarantee that the exponents of $\scal{\ell_j}$ and $\scal{\widetilde{\ell}_j}$ satisfy the hypothesis of Lemma~\ref{le:convolution}. Hence, we can apply that lemma to sum up $\sum_{j} \ell_j = \ell$ and $\sum_{j} \widetilde{\ell} = \widetilde{\ell}$ first, and then use Lemma~\ref{le:convolution_resonance} to sum up $\ell$ and $\widetilde{\ell}$. This yields the bound
\begin{equation*}
\E |\widehat{\tau_{\eps,m}}^{(4m+1)}(t,k)|^{2} \lesssim \eps^{2\delta} \sum_{\stackrel{\ell+\widetilde{\ell}=k}{\ell \sim \widetilde{\ell}}} \left( \frac{1}{\scal{\ell}^{8-(2m+2)\alpha}} \cdot \frac{1}{\scal{\widetilde{\ell}}^{3-2m\widetilde{\alpha}}} \right) \lesssim \eps^{2\delta} \frac{1}{\scal{k}^{2-2\delta}}\;, \quad m \geq 2\;.
\end{equation*}
This is of the desired form. For $m=1$, the convergence can be shown in a similar way, and we omit the details.

\medskip

\noindent
\textbf{Analysis of $\widehat{\tau_{\eps,m}}^{(2j+1)}$ for $2 \leq j \leq 2m-1$:}
\medskip

\noindent
This case can be dealt with a mixture of the methods we used to analyse $\widehat{\tau_{\eps,m}}^{(1)}$ and $\widehat{\tau_{\eps,m}}^{(4m+1)}$. We omit the details.

\appendix

\section{Besov spaces, paraproducts and (perturbed) heat kernel estimates}
\label{sec:Besov}

We collect in this appendix some definitions and estimates on Besov spaces and paraproducts that are used throughout the article. Roughly speaking, these are normed spaces of functions/distributions characterised in terms of the behaviour of their Fourier transforms. They enjoy remarkable stability properties under paraproduct operations. 

A comprehensive account can be found in the book \cite{BookChemin}. Most of the statements below have also been cleanly stated in \cite{para_control_theory, phi43_CC, Phi43global}. Only minor modifications are made in statements concerning uniform (in $\eps$) regularisation properties of the perturbed heat kernel $(\d_t - \lL_\eps + 1)^{-1}$ and its difference with the heat kernel $(\d_t - \Delta + 1)^{-1}$.

\subsection{Besov spaces and Bony's paraproducts}

Let $\widetilde{\chi}, \chi$ be two $\cC_c^{\infty}(\R^d)$ functions taking values in $[0,1]$ such that
\begin{enumerate}
	\item $\supp (\widetilde{\chi}) \subset B(0,\frac{4}{3})$, and $\supp(\chi) \subset B(0,\frac{8}{3}) \setminus B(0,\frac{3}{4})$. 
	
	\item $\widetilde{\chi}(\xi) + \sum_{j=0}^{+\infty} \chi(\xi/2^j) = 1$ for all $\xi \in \R^d$. 
\end{enumerate}
We also define
\begin{equation*}
\chi_{-1} := \widetilde{\chi}\;, \quad\text{and}\quad \chi_j := \chi (\cdot / 2^j) \quad \text{for} \; j \geq 1\;. 
\end{equation*}
Let $\T^d = (\R / \Z)^{d}$ be the $d$-dimensional torus. For every function/distribution $f$ on $\T^d$, its Fourier transform $\widehat{f}: \Z^d \rightarrow \mathbf{C}$ is defined by
\begin{equation*}
\widehat{f}(k) := \int_{\T^d} f(x) e^{-2 \pi i k \cdot x} {\rm d}x\;.
\end{equation*}
For every integer $j \geq -1$ and $f$ on $\T^d$, we define the functions $\Delta_j f$ and $\sS_j f$ by
\begin{equation*}
\widehat{\Delta_{j} f} =  \chi_j \widehat{f}\; \quad \text{and} \quad \sS_{j} f := \sum_{i=-1}^{j-1} \Delta_i f\;. 
\end{equation*}
For every $\alpha \in \R$ and $f \in \cC^{\infty}(\T^d)$, we define the Besov norm $\|\cdot\|_{\bB^\alpha(\T^d)}$ of $f$ by
\begin{equation} \label{eq:norm_Besov}
\|f\|_{\bB^\alpha(\T^d)} = \sup_{j \geq -1} \big( 2^{\alpha j} \|\Delta_j f\|_{L^{\infty}(\T^d)} \big)\;. 
\end{equation}
The right hand side above is finite for every $f \in \cC^{\infty}(\T^d)$. 

\begin{defn} \label{de:Besov}
	For every $\alpha \in \R$, the Besov space $\bB^{\alpha} = \bB^\alpha (\T^d)$ is the completion of $\cC^{\infty}(\T^d)$ functions with respect to the norm $\|\cdot\|_{\bB^\alpha}$ given in \eqref{eq:norm_Besov}. 
\end{defn}

We write $\|\cdot\|_{\alpha} = \|\cdot\|_{\bB^\alpha}$ for simplicity. Note that for $\alpha \in \R^{+} \setminus \N$, the Besov norm $\|\cdot\|_{\alpha}$ is equivalent to the usual H\"older-$\alpha$ norm $\cC^\alpha$\protect\footnote{That is, the $L^\infty$-norm of the first $\lfl \alpha \rfl$ derivatives plus the H\"older-($\alpha - \lfl \alpha \rfl$) norm for the $\lfl \alpha \rfl$-th derivative.}. We refer to \cite[Page~99]{BookChemin} and \cite[Appendix A]{para_control_theory} for more discussions. 

Here, we define the Besov space to be the completion of smooth functions under the norm $\bB^\alpha$ rather than functions/distributions with a finite right hand side of \eqref{eq:norm_Besov}. This yields a slightly smaller space, but has the advantage that smooth approximations converge in the same space. In particular, this definition enables us to show that the approximated solutions $(v_\eps, w_\eps)$ converge in the same space where they are constructed with help of the fixed point map in Theorem~\ref{th:fixed_pt_convergence}. The difference between the two definitions are in complete analogy with the two versions of the H\"older space $\cC^\alpha$: completion of smooth functions with respect to the $\cC^\alpha$ metric, and functions that fluctuate locally at order $\alpha$.

\begin{lem} \label{le:embedding}
	For every $\alpha > 0$, we have the embeddings
	\begin{equation*}
	\|f\|_{0} \lesssim \|f\|_{L^{\infty}} \lesssim_\alpha \|f\|_{\alpha}. 
	\end{equation*}
\end{lem}
\begin{proof}
	See \cite[Remarks~3.4, 3.5 and 3.6]{Phi42global} or \cite[Remark~A.3]{Phi43global} for the first bound. The second inequality is a direct consequence of the fact that $\alpha >0$.
\end{proof}


For $f, g \in \cC^{\infty}(\T^d)$, we define the paraproducts $\prec$, $\succ$ and the resonant product $\circ$ as
\begin{equation*}
f \prec g = \sum_{i < j-1} \Delta_i f \; \Delta_j g = \sum_{j} \sS_{j-1}f \; \Delta_j g\;, \qquad f \succ g = g \prec f\;, 
\end{equation*}
and
\begin{equation*}
f \circ g = \sum_{|i-j| \leq 1} \Delta_i f \; \Delta_j g\;. 
\end{equation*}
The usual pointwise product can then be decomposed into the sum
\begin{equation*}
f g = f \prec g + f \succ g + f \circ g\;, 
\end{equation*}
at least when $f$ and $g$ are sufficiently regular. The following proposition states that the two paraproducts are always well defined regardless of the regularity of $f$ and $g$, while the resonance product requires the sum of the two regularities to be positive.

\begin{prop} \label{pr:Bony_estimates}
	We have the following bounds: 
	\begin{enumerate}
		\item $\|f \prec g\|_{\beta} \lesssim \|f\|_{L^\infty} \|g\|_{\beta}$; 
		
		\item $\|f \succ g\|_{\alpha+\beta} \lesssim \|f\|_{\alpha} \|g\|_{\beta}$ if $\beta < 0$; 
		
		\item $\|f \circ g\|_{\alpha+\beta} \lesssim \|f\|_{\alpha} \|g\|_{\beta}$ if $\alpha+\beta>0$. 
	\end{enumerate}
    The proportionality constants depend on $\alpha$ and $\beta$, but are uniform over $f,g$ in the respective function classes. 
\end{prop}

\begin{prop} \label{pr:commutator_estimate}
	Suppose $\alpha, \beta, \gamma \in \R$ satisfy $\alpha \in (0,1)$, $\beta+\gamma<0$ and $\alpha+\beta+\gamma>0$. Then the commutator $\Com$ defined by
	\begin{equation*}
	\Com(f,g,h) := (f \prec g) \circ h - f (g \circ h)
	\end{equation*}
	satisfies
	\begin{equation*}
	\|\Com(f,g,h)\|_{\alpha+\beta+\gamma} \lesssim \|f\|_{\alpha} \|g\|_{\beta} \|h\|_{\gamma}\;, 
	\end{equation*}
	uniformly over $f,g,h \in \cC_{c}^{\infty}(\T^d)$. 
\end{prop}

With the bounds in Propositions~\ref{pr:Bony_estimates} and~\ref{pr:commutator_estimate}, the paraproduct and commutator operations can be continuously extended to functions/distributions in the Besov spaces whose exponents satisfy the assumptions of these two propositions.

\subsection{Bounds for the (perturbed) heat semi-group}

We give bounds on the heat semi-group $e^{t\lL_\eps}$ that are used throughout Section~\ref{sec:PDE}. Note that the assumption~\ref{eq:Q_derivative_growth} allows an extra $\delta$ power of $\qQ$ to control its derivatives. This gives a small loss at $t=0$ in the regularisation estimates (see the statements below). We do not know whether this loss could indeed happen, or it is because our bounds are not sharp.

\begin{lem} \label{le:op_bd_general}
	Let $\qQ$ satisfies Assumption~\ref{as:Q} (in general dimension $d$, and all derivatives up to order $N$ satisfy the last item in that assumption). Then there exists $c=c(N)$ such that for every $\delta>0$, there exists $C=C(\delta,N)$ such that
	\begin{equation} \label{eq:op_bd_general}
	\sup_{|\zeta| \geq \frac{1}{20}} \big| \d_{\zeta}^{\ell} \big( e^{-r \qQ(\mu \zeta)} \big) \big| \leq C (1 + r^{-\frac{\delta}{2}}) e^{- c r \mu^2}
	\end{equation}
	for all $r>0$, $\mu>0$, and every multi-index $\ell \in \N^d$ with $|\ell| \leq N$. Here $|\ell|$ denotes the sum of individual components of $\ell$. 
\end{lem}
\begin{proof}
	By Fa\`a di Bruno's formula the quantity $\d_\zeta^{\ell} \big( e^{-r \qQ(\mu \zeta)} \big)$ is a linear combination of terms of the form
	\begin{equation*}
	\Big( \prod_{i=1}^{m} \left(r \mu^{|\ell_i|} (\d^{\ell_i} \qQ)(\mu \zeta) \right)^{n_i} \Big) \cdot e^{- r \qQ(\mu \zeta)}\;,
	\end{equation*}
	where $n_1, \dots, n_m \in \N$ and $\ell_1, \dots, \ell_m \in \N^d$ satisfy
	\begin{equation*}
	n_1 |\ell_1| + \cdots + n_m |\ell_m| = |\ell|\;.
	\end{equation*}
	If $|\mu\zeta|\leq 1$, the assumption on $\qQ$ near the origin and on the range of $\zeta$ imply
	\begin{equation*}
	\mu^{|\ell_i|} |(\d^{\ell_i} \qQ)(\mu \zeta)| \lesssim (\mu |\zeta|)^{\ell_i} |(\d^{\ell_i} \qQ)(\mu \zeta)|  \lesssim \qQ(\mu \zeta)\;, 
	\end{equation*}
	and hence
	\begin{equation*}
	\Big( \prod_{i=1}^{m} \left(r \mu^{|\ell_i|} (\d^{\ell_i} \qQ)(\mu \zeta) \right)^{n_i} \Big) \cdot e^{- r \qQ(\mu \zeta)} \lesssim \big( r \qQ(\mu \zeta) \big)^{} e^{-r \qQ(\mu \zeta)} \lesssim e^{-c r \qQ(\mu \zeta)}\;, 
	\end{equation*}
	where we set $n=\sum_{i=1}^{m} n_i$.
    If $|\mu\zeta| \geq 1$, then Assumption~\ref{eq:Q_derivative_growth} and the range of $\zeta$ considered imply
	\begin{equation*}
	\mu^{|\ell_i|} |(\d^{\ell_i} \qQ)(\mu \zeta)| \lesssim (\mu |\zeta|)^{\ell_i} |(\d^{\ell_i} \qQ)(\mu \zeta)| \lesssim |\qQ(\mu \zeta)|^{1+\delta}\;,
	\end{equation*}
	and hence
	\begin{equation*}
	\Big( \prod_{i=1}^{m} \left(r \mu^{|\ell_i|} (\d^{\ell_i} \qQ)(\mu \zeta) \right)^{n_i} \Big) \cdot e^{- r \qQ(\mu \zeta)} \lesssim r^{-n \delta} \big( r \qQ(\mu \zeta) \big)^{n(1+\delta)} e^{-r \qQ(\mu \zeta)} \lesssim r^{-n \delta} e^{-c r \qQ(\mu \zeta)}\;,
	\end{equation*}
	uniformly over $\mu \geq 20$ and $\zeta \geq \frac{1}{20}$. Since $\delta$ can be chosen arbitrarily (by adjusting the proportionality constant), we can replace $n \delta$ by $\frac{\delta}{2}$. The conclusion then follows by combining the two bounds and using $|\qQ(z)| \gtrsim |z|^2$ and $|\zeta| \geq \frac{1}{20}$. 
\end{proof}

\begin{lem} \label{le:heat_regularisation}
	For every $\gamma \geq \alpha$ and every $\delta \in (0,1)$, we have
	\begin{equation} \label{eq:heat_regularisation}
	\|e^{t(\lL_\eps-1)} f\|_{\gamma} \lesssim t^{-\frac{\gamma-\alpha}{2}} (1 + \eps^{\delta} t^{-\frac{\delta}{2}}) \|f\|_\alpha\;.
	\end{equation}
	Moreover, for every $\delta \in (0,1)$, we have
	\begin{equation} \label{eq:heat_regularisation_difference}
	\|e^{t(\lL_\eps-1)} f - e^{t(\Delta-1)} f\|_{\gamma} \lesssim \eps^{\delta} t^{-\frac{\gamma-\alpha+\delta}{2}} \|f\|_\alpha\;.
	\end{equation}
	Both proportionality constants depend on $\alpha$, $\gamma$ and $\delta$, but are uniform in $\eps \in [0,1]$, $t>0$ and $f \in \bB^\alpha$. 
\end{lem}
\begin{proof}
	Since the operation of $e^{-t \id}$ is multiplication by $e^{-t}$, it suffices to prove the lemma with the operator $e^{t \lL_\eps}$. 
	
	Let $\rho$ be a smooth cutoff function taking value in $[0,1]$, with support in the annulus $B(0,\frac{10}{3}) \setminus B(0,\frac{1}{8})$, and equals $1$ on the support of $\chi$. Then we have
	\begin{equation*}
	\Delta_j \big( e^{t \lL_\eps} f\big) = e^{t \lL_\eps} \big( \Delta_j f \big) = \phi_{t,j}^{(\eps)} * \Delta_j f\;, 
	\end{equation*}
	where $\phi_{t,j}^{(\eps)}$ has Fourier transform
	\begin{equation*}
	\widehat{\phi_{t,j}^{(\eps)}}(\zeta) = e^{t \widehat{\lL_\eps}(\zeta)} \rho (\zeta / 2^j)\;.
	\end{equation*}
	By Young's inequality, we have
	\begin{equation*}
	\|\Delta_j \big( e^{t \lL_\eps} f\big)\|_{L^\infty} \leq \|\phi_{t,j}^{(\eps)}\|_{L^1} \|\Delta_j f\|_{L^\infty}\;.
	\end{equation*}
	It then remains to control $\|\phi_{t,j}^{(\eps)}\|_{L^1}$. Taking the inversion Fourier transform and performing a change of variable, we get
	\begin{equation*}
	\| \phi_{t,j}^{(\eps)} \|_{L^1} = \int_{\R^d} |g_{t,j}^{(\eps)}(x)| {\rm d}x\;,
	\end{equation*}
	where
	\begin{equation*}
	g_{t,j}^{(\eps)}(x) = \int_{\R^d} e^{t \widehat{\lL_\eps}(2^j \zeta)} \rho(\zeta) e^{2 \pi i \zeta \cdot x} {\rm d} \zeta\;.
	\end{equation*}
	Now, for $N \geq \lfl \frac{d}{2} \rfl + 1$, we have
	\begin{equation*}
	\|g_{t,j}^{(\eps)}\|_{L^1} \lesssim \sup_{x \in \R^d} \Big| (1+|x|^2)^{N} g_{t,j}(x) \Big| \lesssim \sup_{x \in \R^d} \Big|  \int_{\R^d} (1-\Delta_{\zeta})^{N} \big( e^{t \widehat{\lL_\eps}(2^j\zeta)} \rho(\zeta) \big) e^{2 \pi i \zeta \cdot x} {\rm d} \zeta \Big|
	\end{equation*}
	Since $\rho$ is supported on $|\zeta| \in [\frac{1}{8}, \frac{5}{3}]$, we obtain
	\begin{equation*}
	\|\phi_{t,j}^{(\eps)}\|_{L^1} \lesssim \sup_{\zeta: |\zeta| \in [\frac{1}{8}, \frac{10}{3}]} \; \sum_{\ell: |\ell| \leq 2N} \Big| \d_\zeta^{\ell} \big( e^{t \widehat{\lL_\eps}(2^j \zeta)} \big) \Big|\;.
	\end{equation*}
	Recall that $\widehat{\lL_\eps}(\zeta) = - \eps^{-2} \qQ(2 \pi \eps 2^j \zeta)$. Applying Lemma~\ref{le:op_bd_general} with $r = \frac{t}{\eps^2}$ and $\mu = 2 \pi \cdot 2^j \eps$ gives the bound \eqref{eq:heat_regularisation}. 
	
	As for the bound \eqref{eq:heat_regularisation_difference}, it suffices to establish control for the quantity
	\begin{equation*}
	\|\phi_{t,j}^{(\eps)} - \phi_{t,j}\|_{L^1} \lesssim \sup_{\zeta: |\zeta| \in [\frac{1}{8}, \frac{10}{3}]} \Big| \d_{\zeta}^{\ell} \big( e^{-\frac{t}{\eps^2} \qQ(2 \pi \cdot 2^j \eps \zeta)} - e^{-t |\zeta|^2} \big) \Big|
	\end{equation*}
	for multi-indices $\ell \in \N^d$ with $|\ell| \leq 2N = 2 ( \lfl \frac{d}{2} \rfl + 1 )$. If $2^j \eps \geq 20$, the argument in Lemma~\ref{le:op_bd_general} already gives the factor $\eps^\delta t^{-\frac{\delta}{2}}$. For $2^j \eps \leq 20$, this factor can be obtained by controlling the difference of the derivatives between $\qQ$ and $|\cdot|^2$. This completes the bound \eqref{eq:heat_regularisation_difference}. 
\end{proof}

\begin{lem} \label{le:heat_continuity}
	For every $\gamma \geq \alpha$ and $\theta \in [0,2]$, we have
	\begin{equation} \label{eq:heat_continuity}
	\|e^{t(\lL_\eps-1)} f - e^{s(\lL_\eps-1)} f\|_{\gamma} \lesssim (t-s)^{\frac{\theta}{2}} s^{-\frac{\gamma-\alpha+\theta}{2}} (1 + \eps^{\delta} s^{-\frac{\delta}{2}}) \|f\|_{\alpha}\;.
	\end{equation}
	Moreover, if $\delta \geq 0$ is sufficiently small such that $\frac{\theta}{2} + \frac{\delta}{2} \leq 1$, then we have
	\begin{equation} \label{eq:heat_continuity_difference}
	\big\| \big( e^{t(\lL_\eps-1)} - e^{s(\lL_\eps-1)} \big) f - \big( e^{t(\Delta-1)} - e^{s(\Delta-1)} \big) f \big\|_{\gamma} \lesssim \eps^{\delta} (t-s)^{\frac{\theta}{2}} s^{-\frac{\gamma-\alpha+\theta+\delta}{2}} \|f\|_{\alpha}\;.
	\end{equation}
	All the proportionality constants are independent of $\eps$ and of $t-s>0$. 
\end{lem}
\begin{proof}
	We first explain how to obtain \eqref{eq:heat_continuity}. For simplicity, we again replace $\lL_\eps-1$ by $\lL_\eps$ without loss of generality. Similar to Lemma~\ref{le:heat_regularisation}, it suffices to control
	\begin{equation*}
	\begin{split}
	\|\phi_{t,j}^{(\eps)} - \phi_{s,j}^{(\eps)}\|_{L^1} &\lesssim \sup_{\zeta: |\zeta| \in [\frac{1}{8}, \frac{10}{3}]} \Big|  \d_\zeta^{\ell} \big( e^{t \widehat{\lL_\eps}(2^j \zeta)} - e^{s \widehat{\lL_\eps}(2^j \zeta)}  \big) \Big|\\
	&= \sup_{\zeta: |\zeta| \in [\frac{1}{8}, \frac{10}{3}]} \Big| \d_\zeta^{\ell} \Big( e^{s \widehat{\lL_\eps}(2^j \zeta)} \big( 1 - e^{(t-s) \widehat{\lL_\eps}(2^j \zeta)} \big)  \Big) \Big|
	\end{split}
	\end{equation*}
	A slight modification of the argument in Lemma~\ref{le:op_bd_general} allows to extract the factor $(t-s)^{\frac{\theta}{2}}$ from the second term inside the derivative and gives the bound \eqref{eq:heat_continuity}. 
	
	The bound \eqref{eq:heat_continuity_difference} can be obtained by combining \eqref{eq:heat_regularisation_difference} and \eqref{eq:heat_continuity} and using the inequality $a \wedge b \leq \sqrt{ab}$ for every $a, b \geq 0$. This completes the proof of the lemma. 
\end{proof}

\begin{lem} \label{le:heat_continuity_same_space}
	For every $\gamma \in \R$, every $\delta>0$, every $f \in \bB^{\gamma}$ and every $T>0$, we have
	\begin{equation*}
	\sup_{t \in [0,\eps^2]} \big( (\sqrt{t}/\eps)^{\delta} \big\| e^{t(\lL_\eps-1)} f - e^{t(\Delta-1)} f \big\|_{\gamma} \big) + \sup_{t \in [\eps^2, T]} \big\| e^{t(\lL_\eps-1)} f - e^{t(\Delta-1)} f \big\|_{\gamma} \rightarrow 0
	\end{equation*}
	as $\eps \rightarrow 0$. 
\end{lem}
\begin{proof}
	By Definition~\ref{de:Besov}, for every $\theta>0$, there exists $g$ with compact spectral support with $\|g-f\|_{\gamma} < \theta$. By the triangle inequality, we have
	\begin{equation*}
	\big\| e^{t(\lL_\eps-1)} f - e^{t(\Delta-1)} f \big\|_{\gamma} \leq \|e^{t(\lL_\eps-1)} (f-g)\|_{\gamma} + \big\| e^{t(\lL_\eps-1)} g - e^{t(\Delta-1)} g \big\|_{\gamma} + \|g-f\|_{\gamma}\;.
	\end{equation*}
	The third term on the right hand side above is smaller than $\theta$ by assumption on $g$. By Lemma~\ref{le:heat_regularisation}, the first one is also smaller than $C\theta$ in both time regimes (with the corresponding weight for $t \in [0,\eps^2]$), and $C$ does not depend on $\eps$. Finally, since $g$ has compactly supported Fourier transform, the second term can be made arbitrarily small (uniformly in $t$) by sending $\eps \rightarrow 0$. 
\end{proof}

\begin{prop} \label{pr:heat_commutator}
	Let $\alpha \in (0,1)$, $\beta \in \R$, and $\gamma > \alpha + \beta$. Then for every $\delta \in (0,1)$, we have the bound
	\begin{equation} \label{eq:heat_commutator}
	\|e^{t(\lL_\eps -1)} \big( f \prec g \big) - f \prec \big( e^{t(\lL_\eps -1)}g \big) \|_{\gamma} \lesssim t^{-\frac{\gamma-\alpha-\beta}{2}} (1 + \eps^{\delta} t^{-\frac{\delta}{2}}) \|f\|_{\alpha} \|g\|_{\beta}\;, 
	\end{equation}
	uniformly over $\eps \in [0,1]$, $t > 0$ and $f \in \cC^\alpha (\T^d)$ and $g \in \cC^\beta(\T^d)$. The case $\eps=0$ corresponds to $e^{t (\Delta-1)}$. Furthermore, for every $\delta \in [0,1]$, we have the bound
	\begin{equation} \label{eq:heat_commutator_difference}
	\begin{split}
	&\phantom{111}\left\| e^{t(\lL_\eps -1)} \big( f \prec g \big) - f \prec \big( e^{t(\lL_\eps -1)}g \big) - \left( e^{t(\Delta -1)} \big( f \prec g \big) - f \prec \big( e^{t(\Delta -1)}g \big) \right) \right\|_{\gamma}\\
	&\lesssim \eps^{\frac{\delta^2}{\alpha+\delta}} t^{-\frac{\gamma-\alpha-\beta+\delta}{2}} \|f\|_{\alpha} \|g\|_{\beta}\;.
	\end{split}
	\end{equation}
\end{prop}
\begin{proof}
	The first claim \eqref{eq:heat_commutator} follows from standard commutator estimate for the heat kernel (see for example \cite[Lemmas~2.5 \& A.1]{phi43_CC} and \cite[Proposition~A.16]{Phi43global}) and the bound in Lemma~\ref{le:op_bd_general}, except that we need to take one more derivative since we need to control the $L^1$-norm of $\widetilde{\phi}_{t,j}^{(\eps)}(x) = x \phi_{t,j}^{(\eps)}(x)$, and that the support of the function $\rho$ is a slightly larger annulus. 
	
	As for the second one, applying both \eqref{eq:heat_regularisation_difference} (with paraproduct estimate) and \eqref{eq:heat_commutator}, we see that the left hand side of \eqref{eq:heat_commutator_difference} is bounded by
	\begin{equation*}
	\big( \eps^{\delta} t^{-\frac{\gamma-\beta+\delta}{2}} \wedge t^{-\frac{\gamma-\alpha-\beta}{2}} \big) \|f\|_{\alpha} \|g\|_{\beta}\;. 
	\end{equation*}
	The conclusion then follows by optimising the above quantity. 
\end{proof}

\section{The standard dynamical $\Phi^4_3$ model}
\label{sec:standard_stochastic}

\subsection{Stochastic objects}
\label{app:stochastic}

In this short section we construct the diagrams alluded to at the beginning of Section~\ref{sec:main_convergence}. More precisely, we will give a meaning to the list of symbols
\begin{equation}\label{eq:standard_trees}
 \<1>, \; \<2>, \; \<30>, \; \<31>, \; \<22>, \; \<32>.
\end{equation}
These are the ones in Table~\ref{table:noise} (we did not include the constant $1$ here). The first symbol in this list refers to the solution of the stochastic heat equation, and has been defined in~\eqref{eq:free_field}.
To define the remaining symbols from it we need to invoke a limiting procedure. More precisely, let $\xi_\eps$ be a mollification of the white noise $\xi$, and denote by $\widetilde{\<1>_\eps}$ be the solution of the second equation in~\eqref{eq:free_field} with $\xi$ replaced by $\xi_\eps$. We then define
\begin{equation}  \label{eq:standard_trees_eps}
\begin{split}
\widetilde{\<2>}_\eps & := (\widetilde{\<1>}_\eps)^2 - c_\eps^{(1)}, \\
\widetilde{\<30>}_\eps & := \int_{-\infty}^{t} e^{-(t-r)(\Delta-1)} \big( \widetilde{\<1>}_\eps^3(r) - 3 c_\eps^{(1)} \widetilde{\<1>}_\eps \big) {\rm d}r, \\
\widetilde{\<31>}_\eps & := \widetilde{\<30>}_\eps \circ \widetilde{\<1>}_\eps, \\
\widetilde{\<22>}_\eps & := \iI (\widetilde{\<2>_\eps}) \circ \widetilde{\<2>_\eps}  - 2 c_\eps^{(2)}, \\
\widetilde{\<32>}_\eps & := \widetilde{\<30>}_\eps \circ \widetilde{\<2>}_\eps- 6 c_\eps^{(2)} \widetilde{\<1>}_\eps. 
\end{split}
\end{equation}
Here the constants $c_\eps^{(1)}$ and $c_\eps^{(2)}$ are defined via \begin{equation}
\label{e.def.ccn}
c_\eps^{(1)} := \E[ (\widetilde{\<1>}_\eps)^2 ] \quad \text{ and } \quad  c_\eps^{(2)} := \E[ \iI ( \widetilde{\<2>}_\eps) \circ \widetilde{\<2>}_\eps] .
\end{equation}
It then follows for instance from~\cite[Theorem 4.3]{phi43_CC}, and~\cite[Theorem 1.1]{Phi43_pedestrians} that the processes defined in~\eqref{eq:standard_trees_eps} converge in $L^p$ in the space $\cC^{|\tau|}$, and the limits of these processes are precisely those symbols listed in~\eqref{eq:standard_trees}. Here, $|\tau|$ refers to the degree of regularity defined in Table~\ref{table:noise}. 

The Fourier transform of the objects in \eqref{eq:standard_trees} can be explicitly written down. The expressions are the ones for $\tau_\eps$ (in the first line of Table~\ref{table:noise}) by formally setting $\eps=0$. 

Finally, we also let
\begin{equation*}
\<20>(t) := \int_{-\infty}^{t} e^{(t-r)(\Delta-1)} \<2>(r) {\rm d}r\;.
\end{equation*}
This is not part of $\Ups \in \xX$, but $\<20>(0)$ appears once in the PDE system.

\subsection{Solution theory}
\label{app:pde}

In this section we formulate the equivalent result to Theorem~\ref{th:fixed_pt_convergence} in the context of the standard $\Phi_3^4$-equation. 
To that end, very much as in Section~\ref{sec:fixed_point} we introduce the following objects.
\begin{equation} \label{eq:G_0_standard}
\begin{split}
G_0(\lambda,\Ups,u) := &\sum_{j=0}^{3} F_{j}(\lambda,\Upsilon) u^j - 3 \lambda (u- \lambda \<30>) \succ \<2>\\
&+ 9 \lambda^2 \Big[ \Com(u-\lambda\<30>; \iI(\<2>); \<2>) + \<2> \circ [\iI, \prec](u-\lambda \<30>, \<2>)\\
&- \big( \<2> \circ e^{t(\Delta-1)} \<20>(0) \big) \cdot (u-\lambda \<30>) \Big]\;, 
\end{split}
\end{equation}
where the functions $F_j$ are defined in~\eqref{eq:F_j_coefficients}. The following result is a consequence of~\cite[Theorem 2.1]{Phi43global}.

\begin{thm}
	\label{th:fixed_pt_standard}
	Recall the definition of the space $\yY_T$ in  \eqref{eq:space_solution_remainder}, and the definition of \eqref{eq:Ups_standard}. Consider the fixed point problem
	\begin{equation} \label{eq:fixed_pt_standard}
	\begin{split}
	&v(t) = e^{t (\Delta-1)} v(0) - 3 \lambda \int_{0}^{t} e^{(t-s)(\Delta-1)} \Big[ \big( v(r) + w(r) - \lambda \<30>(r) \big) \prec \<2>(r) \Big] {\rm d}r\;,\\
	&w(t) = e^{t(\Delta-1)} w(0) - 3 \lambda \int_{0}^{t} e^{(t-r)(\Delta-1)} \Big[ \big( e^{r(\Delta-1)}v(0) + w(r) \big) \circ \<2>(r) \Big] {\rm d}r\\
	&\phantom{11111}+ \int_{0}^{t} e^{(t-r)(\Delta-1)} G_0 \big(\lambda, \Ups(r), v(r) + w(r) \big) {\rm d}r\;,
	\end{split}
	\end{equation}
	Then for every $\lambda \in \R$, and $\big( v(0), w(0) \big) \in \bB^\kappa \times \bB^\kappa$, the fixed point problem \eqref{eq:fixed_pt_standard} has a unique solution $(v, w) \in \yY_{T}$ for any $T>0$. Moreover, the solution to~\eqref{eq:phi43}, which is defined as the limit of $\phi_\eps$ defined in Theorem~\ref{thm:standard}, can be written as $\phi= \<1> - \lambda \<30>+ v+w$, where the initidal data satisfies $\phi(0,\cdot) - \<1>(0,\cdot) \in \bB^\kappa$. 
\end{thm}

\endappendix

\bibliographystyle{Martin}
\bibliography{Refs}

\begin{thebibliography}{BCCH17}
\expandafter\ifx\csname url\endcsname\relax
  \def\url#1{\texttt{#1}}\fi
\expandafter\ifx\csname urlprefix\endcsname\relax\def\urlprefix{URL }\fi
\expandafter\ifx\csname href\endcsname\relax
  \def\href#1#2{#2}\fi
\expandafter\ifx\csname burlalt\endcsname\relax
  \def\burlalt#1#2{\href{#2}{\texttt{#1}}}\fi

\bibitem[BCCH17]{rs_equation}
\textsc{Y.~Bruned}, \textsc{A.~Chandra}, \textsc{I.~Chevyrev}, and
  \textsc{M.~Hairer}.
\newblock Renormalising {SPDE}s in regularity structures.
\newblock \emph{ArXiv e-prints} (2017).
\newblock \burlalt{arXiv:1711.10239}{http://arxiv.org/abs/1711.10239}.

\bibitem[BCD11]{BookChemin}
\textsc{H.~Bahouri}, \textsc{J.-Y. Chemin}, and \textsc{R.~Danchin}.
\newblock \emph{Fourier analysis and nonlinear partial differential equations},
  vol. 343 of \emph{Grundlehren der Mathematischen Wissenschaften}.
\newblock Springer, Heidelberg, 2011.

\bibitem[BHZ19]{rs_algebraic}
\textsc{Y.~Bruned}, \textsc{M.~Hairer}, and \textsc{L.~Zambotti}.
\newblock Algebraic renormalisation of regularity structures.
\newblock \emph{Invent. Math.} \textbf{215}, no.~3, (2019), 1039--1156.
\newblock \burlalt{arXiv:1610.08468}{http://arxiv.org/abs/1610.08468}.

\bibitem[BPRS93]{BPRS93}
\textsc{L.~Bertini}, \textsc{E.~Presutti}, \textsc{B.~R{\"u}diger}, and
  \textsc{E.~Saada}.
\newblock Dynamical fluctuations at the critical point: convergence to a
  nonlinear stochastic {PDE}.
\newblock \emph{Teor. Veroyatnost. i Primenen.} \textbf{38}, no.~4, (1993),
  689--741.

\bibitem[CC18]{phi43_CC}
\textsc{R.~Catellier} and \textsc{K.~Chouk}.
\newblock Paracontrolled distributions and the 3-dimensional stochastic
  quantization equation.
\newblock \emph{Ann. Probab.} \textbf{46}, no.~5, (2018), 2621--2679.
\newblock \burlalt{arXiv:1310.6869}{http://arxiv.org/abs/1310.6869}.

\bibitem[CH16]{rs_analytic}
\textsc{A.~Chandra} and \textsc{M.~Hairer}.
\newblock An analytic {BPHZ} theorem for regularity structures.
\newblock \emph{ArXiv e-prints} (2016).
\newblock \burlalt{arXiv:1612.08138}{http://arxiv.org/abs/1612.08138}.

\bibitem[CMW19]{Phi4_sub_local_bounds}
\textsc{A.~Chandra}, \textsc{A.~Moinat}, and \textsc{H.~Weber}.
\newblock A priori bounds for the $\phi^4$ equation in the full sub-critical
  regime.
\newblock \emph{ArXiv e-prints} (2019).
\newblock \burlalt{arXiv:1910.13854}{http://arxiv.org/abs/1910.13854}.

\bibitem[DPD03]{DPD03}
\textsc{G.~Da~Prato} and \textsc{A.~Debussche}.
\newblock Strong solutions to the stochastic quantization equations.
\newblock \emph{Ann. Probab.} \textbf{31}, no.~4, (2003), 1900--1916.

\bibitem[EH19]{general_discrete}
\textsc{D.~Erhard} and \textsc{M.~Hairer}.
\newblock Discretisation of regularity structures.
\newblock \emph{Ann. Inst. Henri Poincar\'e Probab. Stat.} \textbf{55}, no.~4,
  (2019), 2209--2248.
\newblock \burlalt{arXiv:1705.02836}{http://arxiv.org/abs/1705.02836}.

\bibitem[EO71]{EO}
\textsc{J.-P. Eckmann} and \textsc{K.~Osterwalder}.
\newblock On the uniqueness of the hamiltonian and of the representation of the
  {CCR} for the quartic {B}oson interaction in three dimensions.
\newblock \emph{Helv. Phys. Acta} \textbf{44}, no.~7, (1971), 884--909.

\bibitem[Fel74]{Feldman}
\textsc{J.~Feldman}.
\newblock The $\lambda$ $\varphi^4_3$ field theory in a finite volume.
\newblock \emph{Comm. Math. Phys.} \textbf{37}, no.~2, (1974), 93--120.

\bibitem[FG19]{Phi4_general}
\textsc{M.~Furlan} and \textsc{M.~Gubinelli}.
\newblock Weak universality for a class of 3{D} stochastic reaction-diffusion
  models.
\newblock \emph{Probab. Theory Relat. Fields} \textbf{173}, no. 3-4, (2019),
  1099–1164.
\newblock \burlalt{arXiv:1708.03118}{http://arxiv.org/abs/1708.03118}.

\bibitem[FO76]{FO}
\textsc{J.~Feldman} and \textsc{K.~Osterwalder}.
\newblock The wightman axioms and the mass gap for weakly coupled
  $(\phi^4)_{3}$ quantum field theories.
\newblock \emph{Ann. Phys.} \textbf{97}, no.~1, (1976), 80--135.

\bibitem[GIP15]{para_control_theory}
\textsc{M.~Gubinelli}, \textsc{P.~Imkeller}, and \textsc{N.~Perkowski}.
\newblock Paracontrolled distributions and singular {PDE}s.
\newblock \emph{Forum Math. Pi} \textbf{3}, (2015), e6, 75pp.
\newblock \burlalt{arXiv:1210.2684v3}{http://arxiv.org/abs/1210.2684v3}.
\newblock
  \burlalt{doi:10.1017/fmp.2015.2}{http://dx.doi.org/10.1017/fmp.2015.2}.

\bibitem[GJ73]{GJ}
\textsc{J.~Glimm} and \textsc{A.~Jaffe}.
\newblock Positivity of the $\phi^4_3$ hamiltonian.
\newblock \emph{Fortschr. Physik.} \textbf{21}, (1973), 327--376.

\bibitem[Gli68]{Glimm}
\textsc{J.~Glimm}.
\newblock Boson fields with the ${\Wick{\Phi^4}}$ interaction in three
  dimensions.
\newblock \emph{Comm. Math. Phys.} \textbf{10}, no.~1, (1968), 1--47.

\bibitem[GLP99]{GLP99}
\textsc{G.~Giacomin}, \textsc{J.~L. Lebowitz}, and \textsc{E.~Presutti}.
\newblock Deterministic and stochastic hydrodynamic equations arising from
  simple microscopic model systems.
\newblock In \emph{Stochastic partial differential equations: six
  perspectives}, vol.~64 of \emph{Math. Surveys Monogr.},  107--152. Amer.
  Math. Soc., Providence, RI, 1999.

\bibitem[GP16]{HQ_stationary}
\textsc{M.~Gubinelli} and \textsc{N.~Perkowski}.
\newblock The {H}airer-{Q}uastel universality result at stationarity.
\newblock In \emph{Stochastic analysis on large scale interacting systems},
  RIMS K\^oky\^uroku Bessatsu, B59,  101--115. Res. Inst. Math. Sci. (RIMS),
  Kyoto, 2016.

\bibitem[Hai12]{Singular_perturb_SHE}
\textsc{M.~Hairer}.
\newblock Singular perturbations to semilinear stochastic heat equations.
\newblock \emph{Probab. Theory Relat. Fields} \textbf{152}, no. 1-2, (2012),
  265--297.
\newblock \burlalt{arXiv:1002.3722}{http://arxiv.org/abs/1002.3722}.

\bibitem[Hai14]{rs_theory}
\textsc{M.~Hairer}.
\newblock A theory of regularity structures.
\newblock \emph{Invent. Math.} \textbf{198}, no.~2, (2014), 269--504.
\newblock \burlalt{arXiv:1303.5113}{http://arxiv.org/abs/1303.5113}.
\newblock
  \burlalt{doi:10.1007/s00222-014-0505-4}{http://dx.doi.org/10.1007/s00222-014-0505-4}.

\bibitem[HQ18]{HQ}
\textsc{M.~Hairer} and \textsc{J.~Quastel}.
\newblock A class of growth models rescaling to {KPZ}.
\newblock \emph{Forum Math. Pi} \textbf{6}, no.~3(2018).
\newblock \burlalt{arXiv:1512.07845}{http://arxiv.org/abs/1512.07845}.

\bibitem[HS17]{KPZCLT}
\textsc{M.~Hairer} and \textsc{H.~Shen}.
\newblock A central limit theorem for the {KPZ} equation.
\newblock \emph{Ann. Probab.} \textbf{45}, no.~6B, (2017), 4167--4221.
\newblock \burlalt{arXiv:1507.01237}{http://arxiv.org/abs/1507.01237}.
\newblock
  \burlalt{doi:10.1214/16-AOP1162}{http://dx.doi.org/10.1214/16-AOP1162}.

\bibitem[HX18]{Phi4_poly}
\textsc{M.~Hairer} and \textsc{W.~Xu}.
\newblock Large scale behaviour of three-dimensional continuous phase
  coexistence models.
\newblock \emph{Comm. Pure Appl. Math.} \textbf{71}, no.~4, (2018), 688--746.
\newblock \burlalt{arXiv:1601.05138}{http://arxiv.org/abs/1601.05138}.

\bibitem[HX19]{KPZ_general}
\textsc{M.~Hairer} and \textsc{W.~Xu}.
\newblock Large scale limit of interface fluctuation models.
\newblock \emph{Ann. Probab.} \textbf{47}, no.~6, (2019), 3478--3550.
\newblock \burlalt{arXiv:1802.08192}{http://arxiv.org/abs/1802.08192}.

\bibitem[Kup15]{phi43_Antti}
\textsc{A.~Kupiainen}.
\newblock Renormalization group and stochastic {PDE}s.
\newblock \emph{Annales Henri Poincar\'e}  1--39.
\newblock \burlalt{arXiv:1410.3094}{http://arxiv.org/abs/1410.3094}.
\newblock
  \burlalt{doi:10.1007/s00023-015-0408-y}{http://dx.doi.org/10.1007/s00023-015-0408-y}.

\bibitem[MW17a]{Ising2d}
\textsc{J.-C. Mourrat} and \textsc{H.~Weber}.
\newblock Convergence of the two-dimensional dynamic {I}sing-{K}ac model to
  $\phi^4_2$.
\newblock \emph{Comm. Pure Appl. Math.} \textbf{70}, no.~4(2017).

\bibitem[MW17b]{Phi43global}
\textsc{J.-C. Mourrat} and \textsc{H.~Weber}.
\newblock The dynamical {$\Phi^4_3$} model comes down from infinity.
\newblock \emph{Commun. Math. Phys.} \textbf{356}, no.~3, (2017), 673--753.
\newblock \burlalt{arXiv:1601.01234}{http://arxiv.org/abs/1601.01234}.

\bibitem[MW17c]{Phi42global}
\textsc{J.-C. Mourrat} and \textsc{H.~Weber}.
\newblock Global well-posedness of the dynamic ${\Phi^4}$ model in the plane.
\newblock \emph{Ann. Probab.} \textbf{45}, no.~4, (2017), 2398--2476.
\newblock \burlalt{arXiv:1501.06191}{http://arxiv.org/abs/1501.06191}.

\bibitem[MW18]{Phi43_local_bounds}
\textsc{A.~Moinat} and \textsc{H.~Weber}.
\newblock Space-time localisation for the dynamic $\varphi^4_3$ model.
\newblock \emph{ArXiv e-prints} (2018).
\newblock \burlalt{arXiv:1811.05764}{http://arxiv.org/abs/1811.05764}.

\bibitem[MWX17]{Phi43_pedestrians}
\textsc{J.-C. Mourrat}, \textsc{H.~Weber}, and \textsc{W.~Xu}.
\newblock Construction of $\phi^4_3$ diagrams for pedestrians.
\newblock In \emph{From particle systems to partial differential equations},
  vol. 209 of \emph{Springer Proc. Math. Stat.},  1--46. Springer, 2017.

\bibitem[Nua06]{Nua06}
\textsc{D.~Nualart}.
\newblock \emph{The {M}alliavin {C}alculus and {R}elated {T}opics}.
\newblock Probability and its Applications (New York). Springer-Verlag, Berlin,
  second ed., 2006.

\bibitem[SX18]{Phi4_non_Gaussian}
\textsc{H.~Shen} and \textsc{W.~Xu}.
\newblock Weak universality of dynamical $\phi^4_3$: non-{G}aussian noise.
\newblock \emph{Stoch.Partial Dffer. Equ. Anal. Comput.} \textbf{6}, no.~2,
  (2018), 211--254.
\newblock \burlalt{arXiv:1601.05724}{http://arxiv.org/abs/1601.05724}.

\bibitem[ZZ18]{Phi4_general_whole}
\textsc{R.~Zhu} and \textsc{X.~Zhu}.
\newblock Weak universality of the dynamical {$\Phi^4_3$} model on the whole
  space.
\newblock \emph{ArXiv e-prints} (2018).
\newblock \burlalt{arXiv:1811.01367}{http://arxiv.org/abs/1811.01367}.

\end{thebibliography}

\end{document}